\documentclass[preprint,12pt]{elsarticle}

\usepackage{lineno,hyperref}
\modulolinenumbers[5]

\usepackage{amssymb}
\usepackage{algorithmic}
\usepackage{algorithm}
\usepackage{amssymb}
\usepackage{graphicx}
\usepackage{amsmath,amsfonts,amssymb}
\usepackage{epsfig,makecell,float}
\usepackage{setspace,mathrsfs}
\usepackage{tocloft}
\usepackage{textcomp}
\usepackage{multirow,indentfirst,times,color}
\usepackage[caption=false,font=footnotesize]{subfig}
\usepackage{booktabs}
\usepackage{threeparttable}
\usepackage{amsthm}
\usepackage{extarrows}

\usepackage[titletoc]{appendix}
\usepackage{epstopdf}
\usepackage{natbib}
\biboptions{numbers,sort&compress}
\newtheorem{thm}{Theorem}
\newtheorem{defn}{Definition}
\newtheorem{remark}{Remark}
\newtheorem{lem}{Lemma}

\begin{document}

\begin{frontmatter}

\title{Discrete Linear Canonical Transform on Graphs: Uncertainty Principle and Sampling}

\author{Yu Zhang$^{a,b,c}$}
\author{Bing-Zhao Li$^{a,b}$\corref{mycorrespondingauthor}}
\cortext[mycorrespondingauthor]{Corresponding author}\ead{li\_bingzhao@bit.edu.cn}

\address{$^{a}$School of Mathematics and Statistics, Beijing Institute of Technology, Beijing 100081, China}
\address{$^{b}$Beijing Key Laboratory on MCAACI, Beijing Institute of Technology, Beijing 100081, China}
\address{$^{c}$Department of Mechanical Engineering, Keio University, Yokohama 223-8522, Japan}

%\linenumbers

\begin{abstract}
	With an increasing influx of classical signal processing methodologies into the field of graph signal processing, approaches grounded in discrete linear canonical transform have found application in graph signals. In this paper, we initially propose the uncertainty principle of the graph linear canonical transform (GLCT), which is based on a class of graph signals maximally concentrated in both vertex and graph spectral domains. Subsequently, leveraging the uncertainty principle, we establish conditions for recovering bandlimited signals of the GLCT from a subset of samples, thereby formulating the sampling theory for the GLCT. We elucidate interesting connections between the uncertainty principle and sampling. Further, by employing sampling set selection and experimental design sampling strategies, we introduce optimal sampling operators in the GLCT domain. Finally, we evaluate the performance of our methods through simulations and numerical experiments across applications.
\end{abstract}


\end{frontmatter}

\section{Introduction}
Advancements in information and communication technologies have led to the integration of various types of data in fields such as social networks, recommendation systems, medical image processing, and bioinformatics. Unlike traditional time series or digital images, these datasets often exhibit complex and irregular structures, posing challenges for conventional signal processing tools \cite{GFTlaplace, GFTadjacency1}. Over the past decade, significant progress has been made in the development of tools for analyzing signals defined on graphs, giving rise to the field of graph signal processing (GSP) \cite{GFTlaplace, GFTadjacency1,GFTadjacency2,Goverview,Ghistory}.

GSP relies on graph topology, extending traditional discrete signal processing to signals characterized by the complex and irregular underlying structures represented by corresponding graphs. Emphasizing the relationship, interaction, and effects between signals and their graph structures in various real-world scenarios, GSP aims to build a theoretical framework for high-dimensional and large-scale data analysis tasks from a signal processing perspective. This includes graph transforms \cite{GFTlaplace,GFTadjacency1,GFTadjacency2}, frequency analysis \cite{Gfrequency,Gvertex}, filtering \cite{Graphonfilter,Gfilter}, uncertainty principles on graphs \cite{GUncertainty,GspectralUC,GshapesUC}, sampling and interpolation \cite{GUncertainty,GFTsamp,GFTEfficient,GFTSSS,GPractical,Gdualizing,GFRFTsamp,Ggeneralizedsamp,Ginterpolation}, reconstruction and recovery \cite{GParallel,GNon-Bayesian,GBayes,Jdirectedsampling,Gnoisesampling,Grecovery}, fast computation algorithms \cite{GFTSSS,FGFT,GLCTfast}, and more. In GSP, spectral analysis plays a crucial role, inspired by harmonic analysis, using the graph Laplacian operator as its core theoretical foundation, naturally extending concepts such as frequency and filter banks to the graph domain \cite{GFTlaplace}. On the other hand, inspired by algebraic methods, the multiplication of the graph signal with the adjacency matrix of the underlying graph yields the fundamental shift operation for graph signals \cite{GFTadjacency1,GFTadjacency2}. These two approaches bring complementary perspectives and tools, collectively fostering attention and development in the GSP field. Although the adjacency matrix does not directly represent signal variations as the Laplacian operator does, its eigenvectors still reveal the smoothness and variation characteristics of signals on the graph \cite{Gfrequency}. In this paper, we adopt an algebraic method framework to broaden its scope.

A crucial tool in this framework is the graph Fourier transform (GFT), which expands signals into eigenvectors of the adjacency matrix, defining the spectrum through corresponding eigenvalues. It is a direct generalization of the discrete Fourier transform (DFT) to graphs \cite{GFTadjacency2}. Although the GFT can perform various operations on real-world graph signals, it cannot capture the transformation process from the vertex domain to the spectral domain and cannot handle graph signals with chirp-like characteristics \cite{GFRFT,GFRFTconvol}. In recent years, approaches like windowed graph Fourier transform \cite{Gvertex}, graph fractional Fourier transform (GFRFT) \cite{GFRFT}, windowed graph fractional Fourier transform \cite{WGFRFT}, and the directed graph fractional Fourier transform \cite{GFRFTlaplace} have been proposed to address issues with these methods. However, these approaches still suffer from insufficient degrees of freedom, lack of flexibility, underutilization of parameters, and inability to effectively handle non-stationary signals and non-integer-order situations.

Inspired by the handling of nonlinear and non-stationary signals in discrete signal processing (DSP), we also introduce the discrete linear canonical transform (DLCT) \cite{DLCT} that can solve the limitations of the DFT \cite{DFT} and discrete fractional Fourier transform (DFRFT) \cite{DFRFT} into the GSP. Generalizing the DFT to the DLCT is not a straightforward process; it presents significant challenges. This generalization encompasses a range of transformations including the DFT, DFRFT, Scaling, Fresnel Transform, and chirp multiplication (CM). Consider a chirp signal characterized by a frequency variation over time, as Fig. \ref{fig1}, which is prevalent in various engineering and scientific domains, including radar, communications, and medical imaging applications \cite{DLCT,DFT,DFRFT,LCTapplications}. Since DFT treats the entire signal as a stationary signal and cannot capture temporal dynamics, its magnitude spectrum lacks detailed information about how the frequency changes over time. In contrast, DLCT shows excellent adaptability to such nonlinear and non-stationary signals, and can more clearly distinguish the frequency changes of the signal while having higher frequency resolution.
\begin{figure}[b]
	\begin{center}
		\begin{minipage}[t]{0.9\linewidth}
			\centering
			\includegraphics[width=\linewidth]{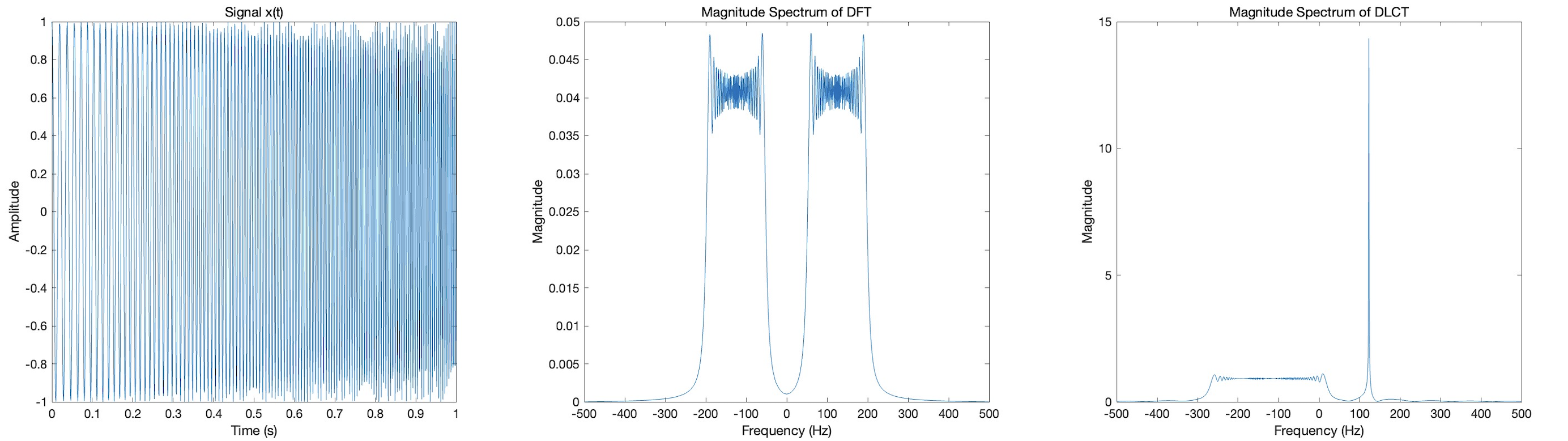}
		\end{minipage}
	\end{center}
	\vspace*{-20pt}
	\caption{The amplitude of signal $x(t)=\exp(2j\pi (f_0 t + 0.5 \mu t^2 ))$ and its spectrum after DFT and DLCT of $[0.1483,-0.9889;0.9889,0.1483]$, with sampling frequency $f_s = 1000$ Hz, time range $t = [0,1]$, initial frequency $f_0 = 50$ Hz, and frequency change rate $\mu = 150$.}
	\vspace*{-3pt}
	\label{fig1}
\end{figure}
In our previous research \cite{GLCT}, we introduced the graph linear canonical transform (GLCT) as a generalization of the DLCT in GSP. Serving as an extension of the GFT and the GFRFT, the GLCT, characterized by its three free parameters, exhibits strong adaptability and flexibility in signal processing. It demonstrates clear advantages in handling non-stationary signals and non-integer-order cases. We prove that it satisfies all expected properties, unifying GFT, GFRFT, graph chirp transform, and graph scale transform, and provide two examples of GLCT.

In this paper, we further explore the uncertainty principle in the GLCT domain. The Heisenberg uncertainty principle in continuous-time signals \cite{UC}, highlights the fundamental trade-off between the signal's spread in time and its spread in frequency. In \cite{GspectralUC}, the uncertainty principle for graph signals was first introduced using geodesic distance, aiming to establish a connection between the signal spread at the vertices of the graph and its spectrum spread defined by GFT in the graph spectral domain. However, a key difference between graph signals and time signals is that, while time or frequency has a well-defined distance concept, graphs are not metric spaces. Additionally, the vertices of a graph can represent more complex signals, such as those in Hilbert spaces \cite{GHilbert}. To overcome the complexity of defining distances, M. Tsitsvero et al. \cite{GUncertainty} creatively introduced prolate spheroidal wave functions \cite{PSWfunctions1,PSWfunctions2} into graph signals, defining a new graph uncertainty principle using the energy percentage. Leveraging this approach, we delineate the uncertainty principle in the GLCT domain, which proves to be broader in scope compared to those in the GFT domain. As the parameters of the GLCT undergo modifications, the scope of the uncertainty principle correspondingly shifts. This dynamic interplay furnishes pivotal insights for the formulation of the sampling theorem in the GLCT domain.

Furthermore, based on the proposed uncertainty principles of GLCT and its localization properties, we consider sampling and reconstruction in the GLCT domain. Most existing methods focus on smooth or bandlimited signals, where the signal's energy tends to be concentrated in a subset of the eigenvectors of the Laplacian or adjacency matrix \cite{GUncertainty,GFTsamp,GFTEfficient,GFTSSS,GPractical,Gdualizing,GFRFTsamp,Ggeneralizedsamp}. Therefore, sampling becomes a problem of selecting the optimal rows of the eigenvector matrix corresponding to the GFT \cite{Gfrequency,GUncertainty,GFTsamp,GFTEfficient,GFTSSS,GPractical,Gdualizing}. Investigations into the relationship between signal spread on graphs and its spectral spread \cite{GUncertainty} have been proposed, along with optimal sampling techniques for signals with known frequency support \cite{GFTsamp}. Methods for noise-robust recovery and blue noise sampling have been developed \cite{Gnoisesampling}. In addition, a novel scalable sampling reconstruction method with parallelization has been proposed \cite{GParallel}, along with non-Bayesian or variational Bayesian estimators based on the Cramer-Rao bound for signal recovery \cite{GNon-Bayesian,GBayes}. Joint sampling and reconstruction of time-varying signals on directed graphs and the graph signal generalized sampling with prior information based on the GFRFT are proposed in \cite{Jdirectedsampling} and \cite{GFRFTsamp,Ggeneralizedsamp} respectively, with most of these bandlimited signal sampling methods employing greedy algorithms to select the optimal sampling sets. There is an intriguing connection between sampling theory and uncertainty principle in the GLCT domain, leading to conditions for recovering bandlimited signals from subsets of their values. We present several signal recovery algorithms and sampling strategies aimed at finding the optimal sampling set. Building on the relationship between the GFT, the GFRFT and, the GLCT, this paper establishes a framework for graph signal sampling based on prior information in the GLCT domain.

Our contributions are summarized as follows:
\begin{itemize}%\textbf{Contributions }
	\item{Firstly, based on conditions for perfect localization in the vertex and graph spectral domains, we propose an uncertainty principle in the GLCT domain and discuss its intriguing connection with sampling.}
	\item{Secondly, by defining bandlimited signals in the GLCT domain, we propose conditions for sampling and perfect recovery in the GLCT domain. We present a method for designing and selecting optimal sampling operators from qualified sampling operators, comparing various sampling strategies.}
	\item{Finally, through simulation experiments, we compare our proposed sampling framework with GFT and GFRFT approaches. The results demonstrate that our sampling framework produces minimal errors and showcases its competitive performance in the application of semi-supervised classification for online blogs and clustering of IEEE 118 bus test cases.}
\end{itemize}

%\textbf{Outline of the paper.}
This paper is structured as follows. Section \ref{2} provides a brief overview of our prior work in \cite{GLCT}, the graph uncertainty principle, and the graph sampling framework. In Section \ref{3}, based on conditions for perfect localization, we propose the uncertainty principle in the GLCT domain and discuss their relationship with sampling. In Section \ref{4}, we begin by defining $\mathbf{M}$-bandlimited signals, introduce conditions for sampling and perfect recovery in the GLCT domain, and select the optimal sampling operator from qualified sampling operators. We present various sampling strategies and compare them. Section \ref{5} evaluates the classification performance of GLCT sampling on online blogs and clustering of IEEE 118 bus test cases, comparing it with GFT and GFRFT sampling. Finally, Section \ref{6} concludes this paper.\footnote{Code Available in: \url{https://github.com/Zhangyubit/GLCTsampling}}

% %Table 1
%\begin{table}[h]
%	\begin{center}
%		\begin{threeparttable}
%			\caption{Key notations in the paper}
%			\begin{tabular}{lll}
%				%				\hlinew{1.2pt}
%				Symbol & Description & Dimension\\
%				\hline
%				$\mathbf{A}$ & adjacency matrix & $N\times N$ \\
%				$x$ & graph signal & $N$ \\
%				$\mathbf{V}^{-1}$ & graph Fourier transform matrix & $N\times N$ \\
%				$\mathbf{O}^{\mathbf{W}}$ & graph linear canonical transform matrix & $N\times N$ \\
%				$\hat{x}$ & transformed graph signal & $N$ \\
%				$\mathrm{\Psi}$ & sampling operator & $M\times N$ \\
%				$\mathrm{\Phi}$ & interpolation operator & $N\times M$ \\
%				$\mathcal{M}$ & sampled indices &  \\
%				$x_\mathcal{M}$ & sampled signal & $M$ \\
%				$\hat{x}_{(K)}$ & fist $K$ coefficients of $\hat{x}$ & $K$ \\
%				$\mathbf{O}^{\mathbf{W}^{-1}}_{(K)}$ & fist $K$ columns of $\mathbf{O}^{\mathbf{W}^{-1}}$ & $N\times K$ \\
%				\hline
%			\end{tabular}
%		\end{threeparttable}
%	\end{center}
%\end{table}

\section{Preliminaries}
\label{2}
In this section, we briefly review concepts in DSP on the GLCT \cite{GLCT} and introduce the uncertainty principle of graph signals \cite{GspectralUC} and the basic theory of sampling \cite{GUncertainty,GFTsamp,GFTEfficient,GFTSSS,GPractical}.

\subsection{Discrete Linear Canonical Transform on Graphs}
In DSP on graphs, the high-dimensional structure of a signal is represented by a graph $\mathcal{G}=(\mathcal{V}, \mathbf{A}),$ where $\mathcal{V}=\{ v_0, \dots, v_{N-1} \}$ denotes the set of vertices and $\mathbf{A}\in \mathbb{C}^{N\times N}$ denotes the weighted adjacency matrix. The graph signal is defined as a mapping that maps the vertex $v_n$ to a signal coefficient $x_n\in\mathbb{C}$, which can also be written as a complex vector $
\mathbf{x} =\left[ x_{0}\  x_{1}\  \ldots \  x_{N-1}\right]^{\top } \in \mathbb{C}^{N}$.

\textit{1) Graph Fourier Transform:} It can be defined through the eigendecomposition of either the adjacency matrix or the Laplacian matrix \cite{GFTlaplace,GFTadjacency1,GFTadjacency2}. Given that the DFT is equivalent to the GFT in the context of cyclic graphs (directed graphs) \cite{GFTadjacency1,GFTadjacency2}, we opt for the adjacency matrix to broaden the scope of application and derive the equivalent transformation of the DLCT \cite{DLCT} on graphs. Consequently, the GFT is defined as $\mathbf{\hat{x}} = \mathbf{V}^{-1} \mathbf{x} $, where $ \mathbf{V}^{-1} $ is the inverse of the matrix composed of the eigenvectors obtained from the eigendecomposition of the adjacency matrix $\mathbf{A}$. Analogous to how the GFT extends the DFT, we propose the extension of the DLCT to GSP, introducing the GLCT. Utilizing the center discrete dilated Hermite functions eigendecomposition \cite{DLCT}, the DLCT can be decomposed into three stages: the DFRFT, scaling, and CM. Therefore, the definition of the GLCT also encompasses three stages: the GFRFT, graph scaling, and graph CM \cite{GLCT}.

\textit{2) Graph Fractional Fourier Transform:} By performing eigendecomposition on the orthogonal and diagonalized GFT matrix $\mathbf{V}^{-1}$, the GFRFT operator can be obtained as \cite{GFRFT, GFRFTconvol}
\begin{equation}
	\mathbf{F^{\alpha }}=(\mathbf{V}^{-1})^{\alpha}=\mathbf{Q}\mathbf{\Lambda}^{\alpha} \mathbf{Q}^{\top},\label{GFRFT}
\end{equation}
where the columns $q_{0},\dots, q_{N-1} $ of $\mathbf{Q}$ are the eigenvectors of $\mathbf{V}^{-1}$, and the eigenvalues in the corresponding diagonal matrix $\mathbf{\Lambda}$ are $\lambda_{0}, \dots, \lambda_{N-1}$. And $\mathbf{\Lambda}^{\alpha}=\mathrm{diag}(\lambda^{\alpha}_{i}),i=0,\dots, N-1$.

\textit{3) Graph Scaling:} 
Inspired by the concept of ``graph time'' obtained through the discrete ordered diffusion of the graph shift operator, we consider the scaled adjacency matrix $\mathbf{S} = \frac{1}{\beta} \mathbf{A}$ as the graph scaling operator \cite{Graphonfilter}. Similar to the GFT, by performing eigendecomposition on $\mathbf{S}$, we obtain the scaled GFT matrix $\mathbf{V}^{-1}_{\beta}$, from which further eigendecomposition yields
\begin{equation}
	\mathbf{V}^{-1}_{\beta}=\mathbf{Q_{\beta }\Lambda_{S} Q^{\top}_{\beta }}, \label{Scale}
\end{equation}
where $\beta$ is the scaling parameter, and $\mathbf{Q}_{\beta }$ and $\mathbf{\Lambda_{S}}$ denote the eigenvectors of $\mathbf{V}^{-1}_{\beta}$ and its corresponding eigenvalues, respectively.

\textit{4) Graph Chirp Multiplication:} Similar to the definition of the discrete chirp-Fourier transform \cite{DCFT}, the graph CM is performed on the GFT matrix $\mathbf{V}^{-1}$, as
\begin{equation}
	\mathbf{F^{\xi}}=\mathbf{Q}\mathbf{\Lambda}^{\xi} \mathbf{Q}^{\top},  \label{CM}
\end{equation}
where $\xi=lk+f$ is the CM parameter which is a linear function of $k$, with $l$ and $f$ being constants, and the matrix $\mathbf{\Lambda}^{\xi}$ is $\mathrm{diag}(\lambda^{\xi}_{k}) = \mathrm{diag}(\exp \left( j\frac{\pi}{2} k \xi \right))$. Therefore, extending the GFT to the GLCT poses certain challenges.

\begin{remark}
Theoretically, in traditional DSP, the DLCT typically comprises at least three components: the DFRFT, scaling transform, and CM. Consequently, in GSP, the GLCT cannot be as straightforwardly extended from the GFT as the GFRFT. Instead, it necessitates the proposal and identification of new graph-based counterparts for scaling transform and CM.
\end{remark}

\textit{5) Graph Linear Canonical Transform:} Based on the aforementioned four transforms and Eqs. \eqref{GFRFT}-\eqref{CM}, we summarize their interrelationships in Fig. \ref{fig2} and define the GLCT as follows.
\begin{figure}[h]
	\begin{center}
		\begin{minipage}[t]{0.9\linewidth}
			\centering
			\includegraphics[width=\linewidth]{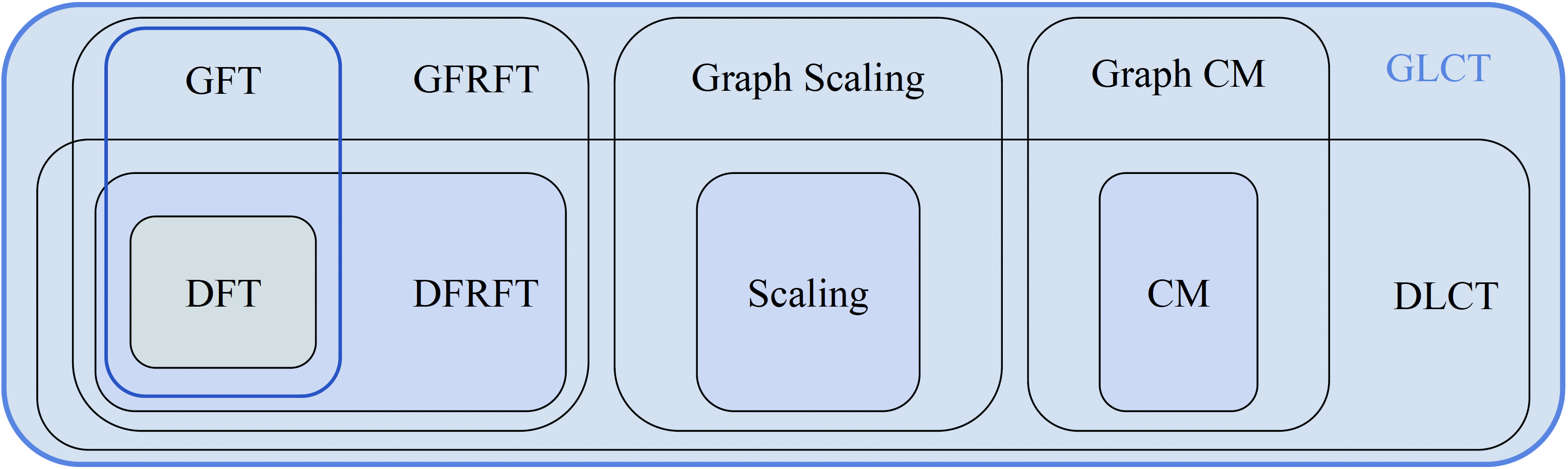}
		\end{minipage}
	\end{center}
	\vspace*{-20pt}
	\caption{Relationship between GLCT and other related transforms.}
	\vspace*{-3pt}
	\label{fig2}
\end{figure}

\begin{defn}
	Parallel to the GFT, the GLCT of $\mathbf{x}\in\mathbb{C}^{N}$ can be defined as \cite{GLCT}
	\begin{equation}
		\mathbf{\hat{x}} =\mathbf{O}^{\mathbf{M}} \mathbf{x}= \mathbf{\Lambda }^{\xi } \mathbf{Q}_{\beta } \mathbf{\Lambda }^{\alpha } \mathbf{Q}^{\top }  \mathbf{x} ,\label{OM}
	\end{equation}
		where the operator $\mathbf{O}^{\mathbf{M}}$ is a operator of the GLCT, and the matrix $\mathbf{M}$ entries are $[a,b;c,d]$, which is called an GLCT parameter matrix, and $ad-bc=1, \ a,b,c,d\in \mathbb{R}$. The parameter relations between $\mathbf{M}$ and $(\xi,\beta,\alpha)$ are
	\begin{equation}
		\xi =\frac{ac+bd}{a^{2}+b^{2}} ,\ \  \beta =\sqrt{a^{2}+b^{2}}, \ \  \alpha =\frac{2}{\pi} \cos^{-1} \left( \frac{a}{\beta } \right)  =\frac{2}{\pi} \sin^{-1} \left( \frac{b}{\beta } \right)  .
	\end{equation}
\end{defn}
Furthermore, its inverse GLCT (IGLCT) is
	\begin{equation}
		\mathbf{x} =\mathbf{O}^{-\mathbf{M}} \mathbf{\hat{x}}  = \mathbf{Q} \mathbf{\Lambda }^{-\alpha } \mathbf{Q}^{\top }_{\beta } \mathbf{\Lambda }^{-\xi }  \mathbf{\hat{x}}  .
	\end{equation}
An example is given using the Minnesota road graph \cite{GLCT}, as shown in Fig. \ref{fig3}. Panel (a) depicts the original signal $\mathbf{x}=\mathbf{O^{M}}\cdot[1,1,...,-1,-1]^{\top}$ is a sequence of length $N$, panel (b) and (c) show the spectrum after the GFT and the GLCT, respectively. By adjusting the parameters of the GLCT, different perspectives of the graph spectrum can be obtained, similar to the DLCT, which can make the spectrum more concentrated.
	\begin{figure}[h]
		\begin{center}
			\begin{minipage}[t]{0.3\linewidth}
				\centering
				\includegraphics[width=\linewidth]{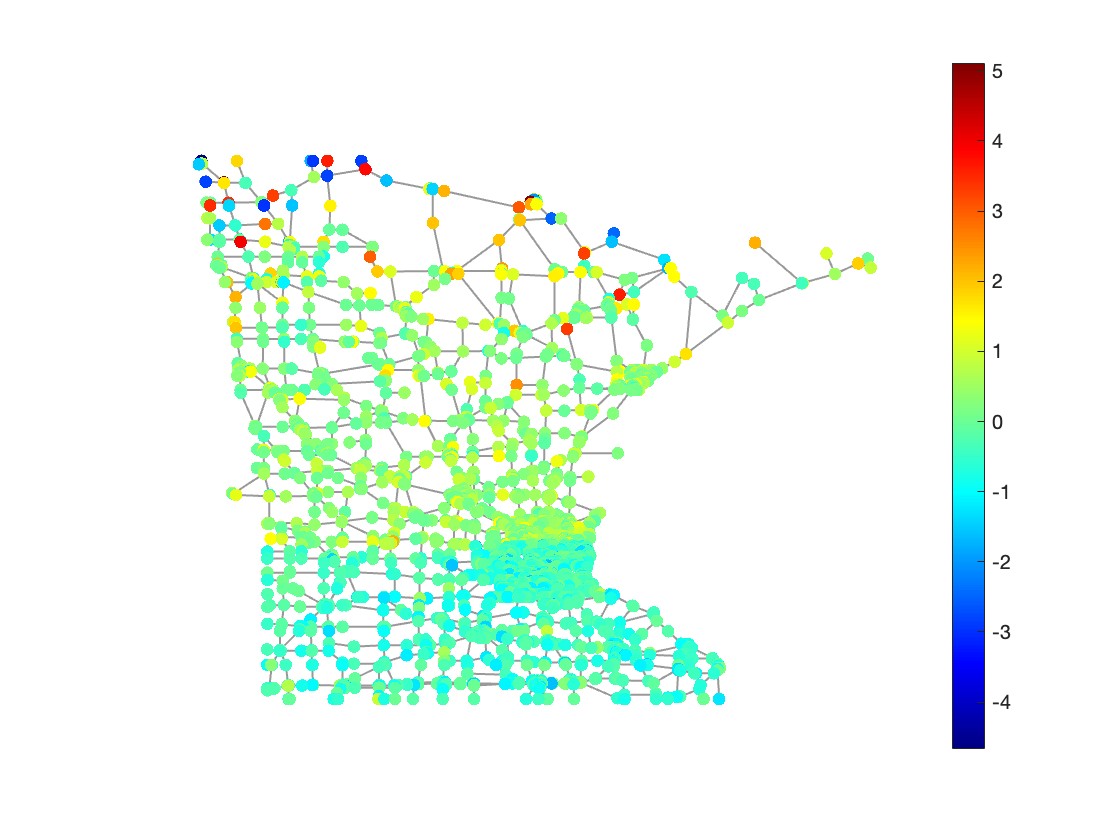}
				\parbox{2.5cm}{\tiny(a) Original graph signal.}
			\end{minipage}
			\begin{minipage}[t]{0.3\linewidth}
				\centering
				\includegraphics[width=\linewidth]{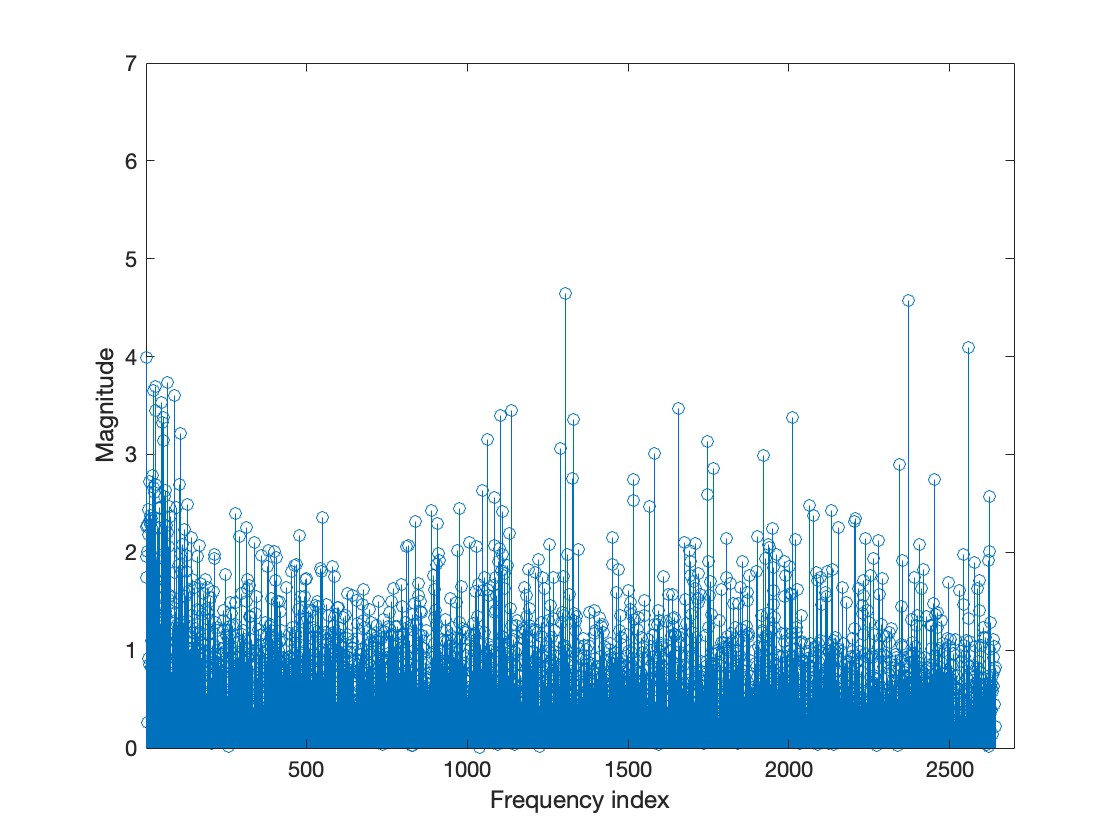}
				\parbox{2.5cm}{\tiny (b) Spectrum after the GFT.}
			\end{minipage}
			\begin{minipage}[t]{0.3\linewidth}
				\centering
				\includegraphics[width=\linewidth]{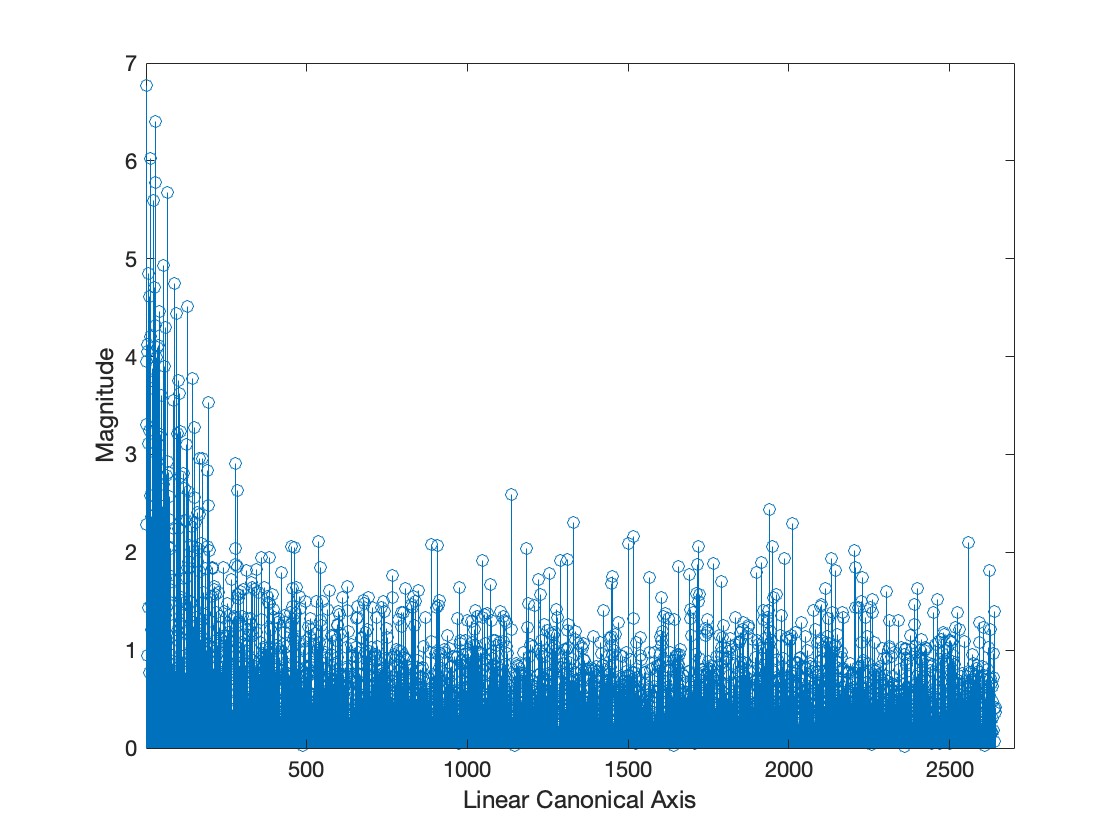}
				\parbox{2.5cm}{\tiny (c) Spectrum after the GLCT.}
			\end{minipage}
		\end{center}
		\vspace*{-10pt}
		\caption{GLCT of bipolar rectangular signals with different parameter matrices on the Minnesota road map, where the parameter matrix of (a) is $[1,0;0,1]$, (b) is $[0,1;-1,0]$, and (c) is $[\sqrt{2}/2, \sqrt{2}/2;-\sqrt{2}/4,3\sqrt{2}/4]$.}
		\vspace*{-3pt}
		\label{fig3}
	\end{figure}

\begin{remark}
	In particular, when $a=\cos \alpha, b=\sin \alpha, c=-\sin \alpha,$ and $d=\cos \alpha$, the GLCT operator will be the GFRFT operator \cite{GFRFT}, when $a=0, b=1, c=-1,$ and $d=0$, the GLCT operator will be the GFT operator.
\end{remark}

\begin{remark}Traditional GFT-based spectral analysis, limited by the Laplacian matrix, overlooks signals' vertex-frequency geometry. The GLCT offers flexibility in shaping spectral characteristics, enabling signal compression, expansion, and rotation in the vertex-frequency plane to form different smoothness patterns.
\end{remark}

\subsection{Graph Uncertainty Principle}
A fundamental property of signals is the Heisenberg uncertainty principle \cite{UC}, which states that there is a fundamental trade-off between the spread of a signal in time and the spread of its spectrum in frequency. The traditional uncertainty principle of time signals is
\begin{equation*}
	\Delta t^{2} \Delta w^{2} \geq \frac{1}{4} ,
\end{equation*}
where $\Delta t^{2}$ and $\Delta w^{2}$ denote the time and frequency spreads, respectively, of the classical continuous-time signal $x(t)$ and its Fourier domain counterpart $\hat{x}(w)$.

After the GFT was proposed, the signal uncertainty principle defined on undirected connected graphs was first derived in \cite{GspectralUC}. The spreads of vector $\mathbf{x}$ in the vertex domain and the GFT domain are respectively defined as
\begin{equation}
\begin{cases}\Delta g^{2}:=\min_{u_{0}\in \mathcal{V} } \frac{1}{\left| \left| \mathbf{x} \right|  \right|^{2}  } \mathbf{x}^{\top } \mathbf{P}^{2}_{u_{0}} \mathbf{x} ,&\\ \Delta s^{2}:=\frac{1}{\left| \left| \mathbf{x} \right|  \right|^{2}  } \sum_{i} \lambda_{\mathbf{A}_{i} } \left| \hat{\mathbf{x} }_{i} \right|^{2}  ,&\end{cases}
\end{equation}
where $\mathbf{P}_{u_{0}}:=\mathrm{diag}(d(u_{0},v_{1}),d(u_{0},v_{2}),...,d(u_{0},v_{N}))$, and $d(u, v)$ denotes the geodesic distance between nodes $u$ and $v$. Thus, the graph uncertainty principle can be represented by the admissible region of $\Delta g^{2}$ and $\Delta s^{2}$, i.e.,
\[
\Gamma_{u_{0}} = \left\{ (s,g): \Delta s^{2}(\mathbf{x}) = s, \Delta g^{2}_{u_{0}}(\mathbf{x}) = g \right\},
\]
	where $\Delta g^{2}_{u_{0}}\left( \mathbf{x}\right)  =\frac{1}{\left| \left| \mathbf{x}\right|  \right|^{2}  } \mathbf{x}^{\top}\mathbf{P}^{2}_{u_{0}}\mathbf{x}$, and the curve can be expressed as
\[
\gamma_{u_{0}} = \min_{\mathbf{x}} \Delta g^{2}_{u_{0}}(\mathbf{x})~\text{subject to}~\Delta s^{2}(\mathbf{x}) = s.
\]

The above uncertainty principle studies the trade-off between the graph signal distribution and its spectral domain based on the specific definition of graph distances. However, a key distinction between the graph signal and the time signal is that time exists in a metric space with a well-defined distance concept, whereas the graph is not a metric space, and, its vertices can be assigned mathematical objects endowed with a richer structure, rather than assigning real numbers to each vertex. An example of such mathematical objects originates from Hilbert space \cite{GHilbert}. In this case, defining distances between vertices is not straightforward.

To overcome the limitations associated with graph distance definitions, this paper adopts an alternative definition based on spread \cite{GUncertainty}. We derive an uncertainty principle for the GLCT domain that does not require any additional distance definitions.

The method characterizes temporal and frequency domain spread based on the percentage of energy falling within the interval $[-T/2, T/2]$ and $[-W/2, W/2]$, respectively, defined as
\begin{equation}
	\frac{\int^{T/2}_{-T/2} \left| x\left( t\right)  \right|^{2}  \mathrm{d} t}{\int^{+\infty }_{-\infty } \left| x\left( t\right)  \right|^{2}  \mathrm{d} t} =\zeta^{2} , \text{and}\ \  \frac{\int^{W/2}_{-W/2} \left| \hat{x} \left( w\right)  \right|^{2}  \mathrm{d} w}{\int^{+\infty }_{-\infty } \left| \hat{x} \left( w\right)  \right|^{2}  \mathrm{d} w} =\eta^{2}. \label{zetaeta}
\end{equation}
Using these two defined spreads, we can further study the uncertainty principle in the GLCT domain in Section \ref{3}.

\subsection{Sampling on Graph Signals}
Sampling is one of the fundamental challenges in GSP. Its objective is to identify conditions for recovering bandlimited graph signals from a subset of values and to devise suitable sampling and recovery strategies \cite{GFTsamp}.

Consider the sampling set $\mathcal{S}=\left( \mathcal{S}_0,\dots,\mathcal{S}_{|\mathcal{S}|-1}\right) , \mathcal{S}_i\in\{0,1, \dots,N-1\}$, its dimension coefficient is defined as $|\mathcal{S}|$. This coefficient is the size of the sampling indexes for obtaining a sampled signal $\mathbf{x}_\mathcal{S}\in\mathbb{C}^{|\mathcal{S}|} (|\mathcal{S}| <N)$ from a bandlimited graph signal $\mathbf{x}\in\mathbb{C}^{N}$. The sampling operator $\mathbf{D}$ is defined as a linear mapping from $\mathbb{C}^N$ to $\mathbb{C}^{|\mathcal{S}|}$, which is expressed as
\begin{equation}
	\mathbf{D}_{i,j}  = \left\{ \begin{array}{rl}
		1, & j=\mathcal{S}_i, \\
		0, & \mathrm{otherwise}.
	\end{array} \right.\label{DS}
\end{equation}
We then recover $\mathbf{x}$ from $\mathbf{x}_\mathcal{S}$ with interpolation operator $\mathbf{R}$, which is a linear mapping from $\mathbb{C}^{|\mathcal{S}|}$ to $\mathbb{C}^N$. Sampling is denoted as $\mathbf{x}_\mathcal{S}=\mathbf{D}\mathbf{x}\in\mathbb{C}^{|\mathcal{S}|}$, and interpolation is represented by $\mathbf{x}_{\mathcal{R}}=\mathbf{Rx}_\mathcal{S}=\mathbf{RDx}\in\mathbb{C}^N$, where $\mathbf{x}_{\mathcal{R}}$ recovers $\mathbf{x}$ either approximately or exactly.

\section{Uncertainty Principle of the GLCT}
\label{3}
In the realm of continuous-time signals, Landau, Pollak, and Slepian extensively explored space-frequency analysis linked to projection operators. Their groundbreaking work, spanning the 1960s, encompassed the uncertainty principle, eigenvalue distributions, and prolate ellipsoid wave functions \cite{PSWfunctions1,PSWfunctions2}. By extending this uncertainty principle to graph signals, we circumvent the need for any extra distance definition, mitigating the drawbacks associated with graph distance definitions. Furthermore, we apply this methodology to establish the uncertainty principle in the GLCT domain.

\subsection{Localization in Vertex and Spectral Domains}
Before defining the uncertainty principle, we first introduce two localization (bandlimited) operators \cite{Glocation}. For a subset of vertices $\mathcal{S}\subseteq \mathcal{V}$, the \textit{vertex-limiting} operator is defined as follows
\begin{equation}
	\mathbf{D}_{\mathcal{S}}=\mathrm{diag}\left( d_1,\ldots ,d_i, \ldots, d_N\right)  ,\label{DSi}
\end{equation}
where, $d_{i}=1$, if $i\in \mathcal{S}$, and $d_{i}=0$ if $i\notin \mathcal{S}, i=1,\ldots,N$. For this operator, it is the corresponding sampling operator $\mathbf{D}_{\mathcal{S}}$ in \eqref{DS}.

Similarly, given the unitary matrix $\mathbf{O}^{\mathbf{M}}$ used in \eqref{OM} and the index subset $\mathcal{F}\subseteq \mathcal{\hat{G}}$, where $\mathcal{\hat{G}} =\left\{ {1,...,N}  \right\}  $ represents the set of all graph spectral  indices, we introduce the \textit{spectral-limiting} operator as
\begin{equation}
\mathbf{B}^{\mathbf{M}}_\mathcal{F}=\mathbf{O}^{-\mathbf{M}}\mathbf{\Sigma}_{\mathcal{F}} \mathbf{O}^{\mathbf{M}}, \label{BMF}
\end{equation}
where, $\mathbf{\Sigma}_{\mathcal{F}}$ is a diagonal matrix whose definition is consistent with $\mathbf{D}_{\mathcal{S}}$, i.e. $\mathbf{\Sigma}_{ii}=1$, if $i\in \mathcal{F}$, and $\mathbf{\Sigma}_{ii}=0$ if $i\notin \mathcal{F}$.
 
According to the properties of diagonal matrices, it is evident that the operators $\mathbf{D}_{\mathcal{S}}$ and $\mathbf{B}^{\mathbf{M}}_\mathcal{F}$ are symmetric and positive semi-definite, with a spectral norm of exactly 1 for both operators. To simplify notation without loss of generality, we use $\mathbf{D}$ and $\mathbf{B}^{\mathbf{M}}$ to represent these two operators, respectively.
 
Therefore, the vector $\mathbf{x}$ is perfectly localized (bandlimited) over $\mathcal{S}$ if
\begin{equation}
	\mathbf{Dx=x},\label{Dxx}
\end{equation}
and perfectly localized (bandlimited) over $\mathcal{F}$ if 
 \begin{equation}
 	\mathbf{B}^{\mathbf{M}}\mathbf{x=x}.\label{Bxx}
 \end{equation}

Here we have a lemma \cite{GUncertainty} about perfect localization.
\begin{lem}
	\label{lem1}
	A vector $\mathbf{x}$ is perfectly localized over both the vertex set $\mathcal{S}$ and the spectral set $\mathcal{F}$ if and only if
	\begin{equation}
		\lambda_{\max} \left( \mathbf{B}^{\mathbf{M}}\mathbf{DB}^{\mathbf{M}}\right) =1.
	\end{equation}
	In such a case, $\mathbf{x}$ is the eigenvector associated with the unit eigenvalue.
\end{lem}
\begin{proof}
The proof is reported in \ref{AA}.
\end{proof}

Equivalently, the perfect localization over sets $\mathcal{S}$ and $\mathcal{F}$ can also be achieved if and only if
\begin{equation}
	\left| \left| \mathbf{B}^{\mathbf{M}}\mathbf{D}\right|  \right|_{2}  =\left| \left| \mathbf{DB}^{\mathbf{M}}\right|  \right|_{2}  =1.
\end{equation}

\subsection{Uncertainty Principle in the GLCT domain}
Using the energy expectations of Eq. \eqref{zetaeta}, we can extend them to the GLCT domain similar to \cite{GUncertainty,GshapesUC}.

Suppose a vector $\mathbf{x}$, vertex subset $\mathcal{S}$ and graph spectral subset $\mathcal{F}$, two subsets are used \eqref{DSi} and \eqref{BMF}. The vectors $\mathbf{Dx}$ and $\mathbf{B}^{\mathbf{M}}\mathbf{x}$ represent the projection of $\mathbf{x}$ on the vertex set $\mathcal{S}$ and the spectral set $\mathcal{F}$ respectively. Then, similar to Eq. \eqref{zetaeta}, $\zeta^{2}$ and $(\eta^{\mathbf{M}})^{2}$ are used to denote the energy expectations falling within the sets $\mathcal{S}$ and $\mathcal{F}$ respectively, as follows
\begin{equation}
\frac{\left| \left| \mathbf{Dx}\right|  \right|^{2}_{2}  }{\left| \left| \mathbf{x}\right|  \right|^{2}_{2}  } =\zeta^{2}, \frac{\left| \left| \mathbf{B}^{\mathbf{M}}\mathbf{x}\right|  \right|^{2}_{2}  }{\left| \left| \mathbf{x}\right|  \right|^{2}_{2}  } =(\eta^{\mathbf{M}})^2.\label{zeta_eta}
\end{equation}

In this paper, we find closed form region boundaries for all admissible pairs $\phi(\zeta, \eta^{\mathbf{M}})$ by generalizing the uncertainty principle of the GFT to the case of the GLCT. In \eqref{zeta_eta}, the graph topology is captured by the matrix $\mathbf{O}^{\mathbf{M}}$, which appears in the definition of the GLCT in \eqref{OM}, inside the operator $\mathbf{B}^{\mathbf{M}}$. Our goal is to establish a trade-off between $\zeta$ and $\eta^{\mathbf{M}}$ by localizing operators $\mathbf{D}$ and $\mathbf{B}^{\mathbf{M}}$, and find a signal that obtains all admissible pairs to prove the uncertainty principle.

We begin with the first uncertainty relation for the localization operators $\mathbf{D}$ and $\mathbf{B^{M}}$, where localization depends on the maximal eigenvalue $\lambda_{\max}$ of the operator $\mathbf{B^{M}DB^{M}}$ \cite{GshapesUC}. That is, $\left\| \mathbf{B^{M}DB^{M}}\right\|_{2} = \lambda_{\max}\left( \mathbf{B}^{\mathbf{M}} \mathbf{DB}^{\mathbf{M}} \right) < 1$. It can be thought that the admissible region $\Gamma$ provides clear bounds for specifying this uncertainty relation on $\mathcal{G}$.

\begin{lem}
	\label{lem2}
	Suppose for a vector $\mathbf{x}$, there are $\left| \left| \mathbf{x} \right|  \right|_{2}   = 1, \left| \left| \mathbf{Dx}\right|  \right|_{2}   = \zeta, \left| \left| \mathbf{B}^{\mathbf{M}}\mathbf{x}\right|  \right|_{2}   = \eta^{\mathbf{M}}$, and $\lambda_{\max } <1$ and if $\lambda_{\max } \leq \zeta^{2} \eta^{2} $, there is an inequality
	\begin{equation}
		\arccos \zeta +\arccos \eta^{\mathbf{M}} \geq \arccos \sqrt{\lambda_{\max } \left( \mathbf{B^{M}DB^{M}}\right)  } , \label{arccoszeta_eta}
	\end{equation}
 and finally reached the upper bound
 \begin{equation}
 	\eta^{\mathbf{M}} \leq \zeta \sqrt{\lambda_{\max} } +\sqrt{\left( 1-\zeta^{2} \right)  \left( 1-\lambda_{\max} \right)  } . \label{etaM}
 	\end{equation}
\end{lem}
\begin{proof}
The proof is reported in \ref{AB}.
\end{proof}

\textbf{Lemma} \ref{lem2} provides a general restriction on the set $\phi (\zeta,\eta^{\mathbf{M}} )$ in the upper right corner of the unit square $[\lambda_{\max}, 1]^2$. By simple generalization extension, we obtain similar results for the other three corners. For a subset $\mathcal{S}$, we denote its complement as $\bar{\mathcal{S}}$ such that $\mathcal{V} = \mathcal{S}\cup \bar{\mathcal{S}}$ and $\mathcal{S}\cap \bar{\mathcal{S}} =\emptyset$. Accordingly, we define the vertex projection on $\bar{\mathcal{S}}$ as $\bar{\mathbf{D}}$. Similarly, the projection on the complement $\bar{\mathcal{F}}$ is denoted by $\bar{\mathbf{B}}^\mathbf{M}$. Considering all subdomains of the square $[0, 1]^2$, we obtain the uncertainty principle for the GLCT.
\begin{thm}
	\label{thm1}
	There exists a vector $\mathbf{x}$ such that $\left| \left| \mathbf{x} \right|  \right|_{2}   = 1, \left| \left| \mathbf{Dx}\right|  \right|_{2}   = \zeta, \left| \left| \mathbf{B}^{\mathbf{M}}\mathbf{x}\right|  \right|_{2}   = \eta^{\mathbf{M}}$, thus the admissible pair $\phi(\zeta, \eta^{\mathbf{M}}) \in \Gamma$ is obtained, where
		\begin{equation}
				\begin{aligned}\Gamma &=   \biggl\{  \phi \left( \zeta ,\eta^{\mathbf{M}} \right)  :\\
					& \arccos \zeta +\arccos \eta^{\mathbf{M}} \geq \arccos \sqrt{\lambda_{\max } \left( \mathbf{B^{M}DB^{M}}\right)  } ,\\ &\arccos \sqrt{1-\zeta^{2} } +\arccos \eta^{\mathbf{M}} \geq \arccos \sqrt{\lambda_{\max } \left( \mathbf{B^{M}}\bar{\mathbf{D}} \mathbf{B^{M}}\right)  } ,\\ &\arccos \zeta +\arccos \sqrt{1-(\eta^{\mathbf{M}})^{2} } \geq \arccos \sqrt{\lambda_{\max } \left( \bar{\mathbf{B}}^{\mathbf{M}}\mathbf{D}\bar{\mathbf{B}}^{\mathbf{M}} \right)  } ,\\ &\arccos \sqrt{1-\zeta^{2} } +\arccos \sqrt{1-(\eta^{\mathbf{M}})^{2} } \geq \arccos \sqrt{\lambda_{\max } \left( \bar{\mathbf{B}}^{\mathbf{M}}\bar{\mathbf{D}}\bar{\mathbf{B}}^{\mathbf{M}} \right)}.    \biggr\}     \end{aligned} 
			\end{equation}
\end{thm}
%	\begin{equation}
%		\begin{aligned}\Gamma &=   \biggl\{  \phi \left( \zeta ,\eta^{\mathbf{M}} \right)  :\\ &\arccos \zeta +\arccos \eta^{\mathbf{M}} \geq \arccos \sqrt{\lambda_{\max } \left( \mathbf{B^{M}DB^{M}}\right)  } ,\\ &\arccos \sqrt{1-\zeta^{2} } +\arccos \eta^{\mathbf{M}} \geq \arccos \sqrt{\lambda_{\max } \left( \mathbf{B^{M}}\bar{\mathbf{D}} \mathbf{B^{M}}\right)  } ,\\ &\arccos \zeta +\arccos \sqrt{1-(\eta^{\mathbf{M}})^{2} } \geq \arccos \sqrt{\lambda_{\max } \left( \bar{\mathbf{B}}^{\mathbf{M}}\mathbf{D}\bar{\mathbf{B}}^{\mathbf{M}} \right)  } ,\\ &\arccos \sqrt{1-\zeta^{2} } +\arccos \sqrt{1-(\eta^{\mathbf{M}})^{2} } \geq \arccos \sqrt{\lambda_{\max } \left( \bar{\mathbf{B}}^{\mathbf{M}}\bar{\mathbf{D}}\bar{\mathbf{B}}^{\mathbf{M}} \right)}.    \biggr\}     \end{aligned} 
%	\end{equation}
\begin{proof}
	The proof follows a similar structure to that of \textbf{Lemma} \ref{lem2}, demonstrating the inequalities for the remaining three corners.
\end{proof}

Fig. \ref{fig4} illustrates the uncertainty principle of \textbf{Theorem} \ref{thm1}, highlighting some noteworthy characteristics. The blue portion in the figure corresponds to the range of $\Gamma$ in the GFT domain, while in the GLCT domain, the ranges in the lower left and upper right respectively tend to $(0,0)$ and $(1,1)$ as the parameter continuously changes, as indicated by the additional red portion. If we alter the parameters of the GLCT, its range can either expand or shrink, which contrasts sharply with the fixed range of the GFT. And if we set the inequalities given by \textbf{Theorem} \ref{thm1} to equality, we can determine the equations of the curve at the four corners of Fig. \ref{fig4}, namely the upper right, upper left, lower right, and lower left corners. Notably, the upper right corner of the admissible region $\Gamma$ specifies the domain where the maximum energy concentration occurs in the vertex and graph spectral for $\phi(\zeta, \eta^{\mathbf{M}})$. The equation for this curve is obtained
\begin{equation}
		\eta^{\mathbf{M} } =\sqrt{\lambda_{\max } \left( \mathbf{B^{M}DB^{M}}\right)  } \zeta +\sqrt{1-\lambda_{\max } \left( \mathbf{B^{M}DB^{M}}\right)  } \sqrt{1-\zeta^{2} } .
\end{equation}

When the sets $\mathcal{S}$ and $\mathcal{F}$ yield matrices $\mathbf{D}$ and $\mathbf{B}^{\mathbf{M}}$ satisfying the perfect localization condition in \textbf{Lemma} \ref{lem1}, the curve collapses to a single point, precisely the upper right corner. Consequently, similar to \cite{GUncertainty,GshapesUC}, any curve originating from a corner of the admissible region $\Gamma$ in Fig. \ref{fig4} can be mapped onto the corresponding corner, provided that the associated operator meets the criterion for perfect localization.
\begin{figure}[h]
	\centering
	\includegraphics[scale=0.44]{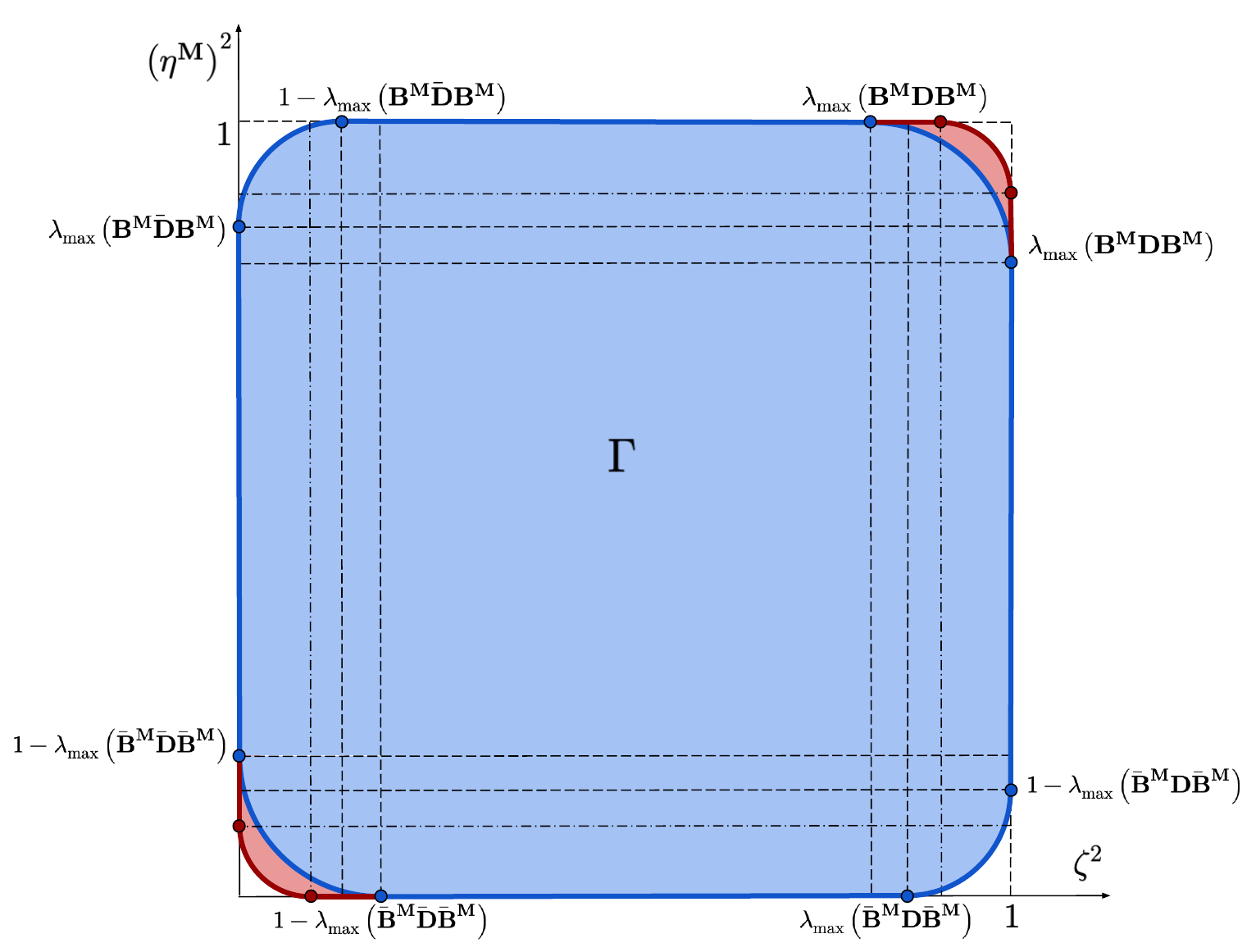}
	\vspace*{-3pt}
	\caption{The admissible region $\Gamma$ in Theorem \ref{thm1}.}
	\label{fig4}
\end{figure}

\subsection{The Connection between the Uncertainty Principle and Sampling}
For the proposed uncertainty principle and as described in \cite{GshapesUC}, there are intriguing connections with the sampling conditions. Examining the upper left corner of the admissible region in Fig. \ref{fig4}, it is evident that if the signal is perfectly localized in the graph spectral subset $\mathcal{F}$, then $\left( \eta^{\mathbf{M}}\right)^2=1$. However, for the vertex domain subset $\mathcal{S}$, we require $\zeta^{2} \neq 0$ to avoid perfect localization of the signal on the complement of $\mathcal{S}$, making it unrecoverable. This is analogous to the perfect recovery conditions given in Section \ref{4}.

If we allow for some energy dissipation in the graph spectral domain, i.e., $\left( \eta^{\mathbf{M}}\right)^2<1$, to ensure $\zeta^{2} \neq 0$, we need to check 
\[
1-\lambda_{\max } \left( \mathbf{B^{M}\bar{D}B^{M}}\right) >\left( \eta^{\mathbf{M}}\right)^2.
\]

This form of condition is highly useful for designing potential sampling strategies. Specifically, the largest eigenvalue, $\lambda_{\max} \left( \mathbf{B^{M}\bar{D}B^{M}} \right)$, of the operator $\mathbf{B^{M}\bar{D}B^{M}}$ is equivalent to the square of the maximum singular value, $\sigma^{2}_{\max} \left( \mathbf{\bar{D} B^{M}} \right)$, of the operator $\mathbf{\bar{D}B^{M}}$. This equivalence suggests the use of $\sigma^{2}_{\max} \left( \mathbf{\bar{D} B^{M}} \right)$ or $\lambda_{\max} \left( \mathbf{B^{M}\bar{D}B^{M}} \right)$ as the objective function for minimization, as exemplified in Table \ref{tab1}. Specifically, $\sigma^{2}_{\max} \left( \mathbf{\bar{D} B^{M}} \right)$ also serves as a criterion for perfect signal recovery, requiring that $\sigma^{2}_{\max} \left( \mathbf{\bar{D} B^{M}} \right)<1$ always holds. These aspects will be further explored in Section \ref{4}.

\section{Sampling Theory of the GLCT}
\label{4}
In this section, we present a novel sampling theory for the GLCT, drawing inspiration from the established sampling theories of the GFT and the GFRFT \cite{GUncertainty,GFTsamp,GFTEfficient,GFTSSS,GPractical,Gdualizing,GFRFTsamp,Ggeneralizedsamp}. Our approach to graph sampling, rooted in the graph linear canonical basis and graph frequency-based sampling methods, tackles the challenge of reconstructing bandlimited graph signals from their subsampled counterparts.

An $\mathbf{M}$-bandlimited graph signal is characterized as a signal with zero graph linear canonical coefficients aligned with the eigenvalues of $\lambda_{|\mathcal{F}|-1}$. This definition implies that $f(i)=0$ for $\lambda, i\geq |\mathcal{F}|$, where $\mathcal{F}$ denotes the set of indices associated with non-zero graph linear canonical coefficients. Specifically, the indices in $\mathcal{F}$ identify those components of the graph linear canonical coefficients that contribute significantly to the signal. We further elaborate on the concept of $\mathbf{M}$-bandlimited graph signals in the subsequent discussion.

%In practice, the sampling method introduced in the GLCT domain not only addresses the limitations of GFT in handling signals with nonlinearity, non-stationarity, and chirp-like features but also demonstrates robustness in processing various real-world signals. By judiciously selecting the optimal GLCT parameters, we can enhance the performance of signal reconstruction, filtering, classification, and clustering tasks.

\subsection{$\mathbf{M}$-Bandlimited Graph Signals}
In the context of GSP, the notion of $\mathbf{M}$-bandlimited signals becomes imperative, necessitating adherence to the previously introduced conditions of perfect localization. Assuming an arbitrary signal $\mathbf{y}$ lacks bandlimited characteristics, preprocessing is required in the form of bandlimiting.

This involves the utilization of a graph spectral domain filter $\mathbf{\Sigma}_{\mathcal{F}}$, defined analogously to Eq. \eqref{BMF}. First, the GLCT is applied to the signal $\mathbf{y}$, yielding $\mathbf{\hat{y}}=\mathbf{O}^{\mathbf{M}}\mathbf{y}$. Subsequently, $\mathbf{\hat{y}}$ undergoes filtering
\[
\mathbf{\hat{x}}=\mathbf{\Sigma}_{\mathcal{F}}\mathbf{\hat{y}}=\mathrm{diag}\left(1,...,1 \right) \mathbf{\hat{y}} .
\]
Following this, an IGLCT is applied to $\mathbf{\hat{x}}$, resulting in $\mathbf{x}=\mathbf{O}^{-\mathbf{M}} \mathbf{\hat{x}}$. Consolidating these steps, we arrive at

\begin{equation}
	\mathbf{x}=\mathbf{O}^{-\mathbf{M}} \mathbf{\Sigma}_{\mathcal{F}} \mathbf{O}^{\mathbf{M}}\mathbf{y},
\end{equation}
where, it is discernible that $\mathbf{O}^{-\mathbf{M}} \mathbf{\Sigma}_{\mathcal{F}} \mathbf{O}^{\mathbf{M}}=\mathbf{B^{M}}$. Through this methodology, a $\mathbf{M}$-bandlimited signal $\mathbf{x}$ is attained, satisfying the conditions of perfect localization as delineated in Eqs. \eqref{Dxx} and \eqref{Bxx}. Therefore, we introduce the following definition for the $\mathbf{M}$-bandlimited set.

\begin{defn}
	The closed subspace of graph signals in $\mathbb{C}^{N}$ with a bandwidth of at most $|\mathcal{F}|$ is denoted as $\mathrm{BL}_{|\mathcal{F}|}(\mathbf{O}^{\mathbf{M}})$, where $\mathbf{O}^{\mathbf{M}}$ is defined as in Eq. \eqref{OM}.
\end{defn}

\subsection{Sampling and Perfect Recovery}
We initiate our exploration by addressing the fundamental challenge of determining conditions and methodologies for the perfect recovery of $\mathbf{x}$ from the sampled graph signal $\mathbf{x}_{\mathcal{S}}$. Drawing upon certain definitions of graph signal sampling outlined in the preliminaries, we present the following theorem, mirroring the principles introduced in \cite{GFTsamp}.
\begin{thm}\label{thm2}
	Let $\mathbf{D} \in \mathbb{C}^{|\mathcal{S}|\times N}$ denote the sampling operator. The interpolation (recovery) operator $\mathbf{R} \in \mathbb{C}^{N\times |\mathcal{S}|}$ is characterized by the following conditions: (1) $\mathbf{R}$ spans the space $\mathrm{BL}_{|\mathcal{F}|}(\mathbf{O}^{\mathbf{M}})$; (2) $\mathbf{RD}$ functions as a projection operator, achieving perfect recovery
	\begin{equation}
		\mathbf{x}=\mathbf{RDx}=\mathbf{Rx}_{\mathcal{S}},\  \  \text{for all} \  \mathbf{x}\in \mathrm{BL}_{|\mathcal{F}|} \left(  \mathbf{O}^{\mathbf{M}} \right),\label{recovery}
	\end{equation}
	where \(\mathbf{x}_{\mathcal{S}}\) denotes the sampled graph signal.
\end{thm}

If the original signal $\mathbf{x}$ can be perfectly recovered, then
\begin{equation*}
	\mathbf{x-RDx=x-R\left( I-\bar{D} \right)  x=x-R}\left( \mathbf{I-\bar{D}}\mathbf{ B}^{\mathbf{M}}\right)  \mathbf{x}.
\end{equation*}

Hence, the matrix $\mathbf{R=\left( I-\bar{D} B^{M}\right)} ^{-1}$. The perfect recoverability of $\mathbf{x}$ implies the existence of $\mathbf{R}$, which is tantamount to the invertibility of $\mathbf{\left( I-\bar{D} B^{M}\right)}$. Consequently, we infer that when $\left| \left| \mathbf{\bar{D} B^{M}}\right|  \right|  <1$, the original signal can be perfectly recovered. Herein, the maximum norm of $\mathbf{\bar{D} B^{M}}$ being $1$ restricts our consideration to this scenario, where perfect localization on the complement of the sampling set prevails, rendering signal recovery unattainable. This is closely related to the uncertainty principles discussed in Section \ref{3}. Furthermore, since $\mathbf{x}\in \mathrm{BL}_{|\mathcal{F}|} \left( \mathbf{O}^{\mathbf{M}} \right)$,
\begin{equation*}
\mathbf{\left( I-\bar{D} B^{M}\right)  x=\left( I-\bar{D} \right)  x=Dx=DB^{M}x}.
\end{equation*}

By analyzing the preceding equation, it is evident that $\mathbf{\left( I-\bar{D} B^{M}\right)}$ is equivalent to $\mathbf{DB^{M}}$, leading to the conclusion that $\mathbf{R=\left( DB^{M}\right)}^{-1}$. However, in the case of a degenerate matrix, the rank of $\mathbf{DB^{M}}$ may be less than full. In such instances, we employ the pseudo-inverse to define $\mathbf{R=\left( DB^{M}\right)}^{\dag}$.

We have the following theorem regarding sampling and perfect recovery.
\begin{thm}\label{thm3}
	For $\mathbf{D}$ such that
	\begin{equation}
		\mathrm{rank}\left( \mathbf{DO}^{-\mathbf{M}}_{|\mathcal{F}|} \right)=|\mathcal{F}|,\quad |\mathcal{F}|<N,
	\end{equation}
	where $\mathbf{DO}^{-\mathbf{M}}_{|\mathcal{F}|}$ denotes the first $|\mathcal{F}|$ columns of $\mathbf{DO}^{-\mathbf{M}}$. For all $\mathbf{x} \in \mathrm{BL}_{|\mathcal{F}|}(\mathbf{O^{M}})$, perfect recovery, $\mathbf{x=RDx}$, is achieved by choosing
	\[
\mathbf{R=\left( DB^{M}\right)^{\dag } } =\left( \mathbf{DO^{-M}\Sigma}_{\mathcal{F}} \mathbf{O^{M}}\right)^{\dag }  =\mathbf{O}^{-\mathbf{M}}_{ |\mathcal{F}|}\left( \mathbf{DO}^{-\mathbf{M}}_{ |\mathcal{F}|} \right)^{\dag }  ,
	\]
	where let $\mathbf{P}=( \mathbf{DO}^{-\mathbf{M}}_{ |\mathcal{F}| })^{\dag }$, thus $\mathbf{P} \mathbf{D}\mathbf{O}^{-\mathbf{M}}_{|\mathcal{F}|}=\mathbf{I}_{|\mathcal{F}|\times |\mathcal{F}|}$ is a $|\mathcal{F}|\times |\mathcal{F}|$ identity matrix.
\end{thm}
\begin{proof}
	The proof is reported in \ref{AC}.
\end{proof}

\begin{remark}
 When $|\mathcal{S}|<|\mathcal{F}|$, it holds that $\mathrm{rank}(\mathbf{PD} \mathbf{O}^{-\mathbf{M} }_{|\mathcal{F}|}) \leq \mathrm{rank}(\mathbf{P}) \leq |\mathcal{S}| < |\mathcal{F}|$, thus rendering $\mathbf{PDO}^{-\mathbf{M} }_{|\mathcal{F}|}$ incapable of being an identity matrix, and perfect recovery of the original signal is unattainable. In the case of $|\mathcal{S}|=|\mathcal{F}|$, for $\mathbf{PDO}^{-\mathbf{M} }_{|\mathcal{F}|}$ to be the identity matrix, $\mathbf{P}$ must be the inverse of $\mathbf{DO}^{-\mathbf{M} }_{|\mathcal{F}|}$. For $|\mathcal{S}|>|\mathcal{F}|$, $\mathbf{P}$ is the pseudo-inverse of $\mathbf{DO}^{-\mathbf{M}}_{|\mathcal{F}|}$, and in cases where $|\mathcal{S} |\geq |\mathcal{F} |$, perfect recovery of the original signal is possible. For simplicity, we consider the scenario $|\mathcal{S} |= |\mathcal{F} |$.
\end{remark}

\subsection{Qualified Sampling Operator}
The implications of \textbf{Theorem} \ref{thm3} underscore that achieving perfect recovery may be unattainable, even for $\mathbf{M}$-bandlimited graph signals, with arbitrary sampling operators. A qualified sampling operator must, at a minimum, select $|\mathcal{F}|$ linearly independent rows from within $\mathbf{O}^{-\mathbf{M}}_{|\mathcal{F}|}$. We will introduce the definition of a qualified sampling operator, ensuring its adherence to the conditions delineated in \textbf{Theorem} \ref{thm3}. To achieve this, we formally define the adjacency matrix in the graph linear canonical domain, thus
\begin{equation}
	\mathbf{A^{M}=O^{-M}\Lambda_{A} O^{M}},
\end{equation}
where $\mathbf{A^{M}}$ lacks physical significance as it represents a complex matrix, specifically the adjacency matrix. However, when $\mathbf{M}=[0,1;-1,0]$, $\mathbf{O}^{\mathbf{M}}$ degenerates into $\mathbf{V}^{-1}$, and in this case, $\mathbf{A^{M}}$ serves as the shift operator.

For any sampling and recovery operators satisfying \textbf{Theorem} \ref{thm3}, when $\mathbf{x}$ is $\mathbf{M}$-bandlimited, we have
\[
\mathbf{x=RDx=O}^{-\mathbf{M}}_{|\mathcal{F}|}\mathbf{Px}_{\mathcal{S}}=\mathbf{O}^{-\mathbf{M}}_{|\mathcal{F}|}\mathbf{\hat{x}}_{|\mathcal{F}|},
\]
where $\mathbf{\hat{x}}_{|\mathcal{F}|}$ represents the first $|\mathcal{F}|$ values of $\mathbf{\hat{x}}$. Therefore, we derive
\[
\mathbf{x}_{\mathcal{S}}=\mathbf{P}^{-1}\mathbf{\hat{x}}_{|\mathcal{F}|}=\mathbf{P}^{-1}\mathbf{\hat{x}}_{\mathcal{S}}.
\]
As observed, the sampled signal $\mathbf{x}_{\mathcal{S}}$ and the frequency signal $\mathbf{\hat{x}}_{|\mathcal{F}|}$ in the dual domain form a GLCT operator. Therefore, we obtain a new adjacency matrix, leading to the following theorem similar to \cite{GFTsamp}.
\begin{thm}\label{thm4}
	For $\mathbf{x} \in \mathrm{BL}_{|\mathcal{F}|}(\mathbf{O^{M}})$, where $|\mathcal{S}|=|\mathcal{F}|$, and if $\mathbf{D}$ is a qualified sampling operator, then
	\begin{equation}
	\mathbf{A}^{\mathbf{M}}_{\mathcal{S}}=\mathbf{P^{-1}\Lambda_{A_{|\mathcal{F}|}}P} \in \mathbb{C}^{|\mathcal{F}|\times |\mathcal{F}|},
\end{equation}
	where $\mathbf{P}=( \mathbf{DO}^{-\mathbf{M}}_{ |\mathcal{F}| })^{-1}$, and $\mathbf{\Lambda_{A_{|\mathcal{F}|}}}$ is the diagonal matrix of the first $|\mathcal{F}|$ eigenvalues of $\mathbf{\Lambda_{A}}$.
\end{thm}

In the graph linear canonical domain, $\mathbf{A}^{\mathbf{M}}_{\mathcal{S}}$ is obtained through $\mathbf{A}^{\mathbf{M}}$ sampling. The original graph signal $\mathbf{x}$ and the sampled graph signal $\mathbf{x}_{\mathcal{S}}$ have equivalent frequency content after undergoing the corresponding GLCT. The information preserved by $\mathbf{A}^{\mathbf{M}}_{\mathcal{S}}$ can be expressed as follows
\[
\begin{aligned}
	\mathbf{x}_{\mathcal{S}}-\mathbf{A}^{\mathbf{M}}_{\mathcal{S}}\mathbf{x}_{\mathcal{S}}=&\mathbf{P^{-1}\hat{x}}_{\mathcal{S}} -\mathbf{P^{-1}\Lambda_{A_{|\mathcal{F}|}} PP^{-1}\hat{x}}_{\mathcal{S}} \\
	=&\mathbf{P^{-1}\left( \hat{x}_{\mathcal{S}} -\Lambda_{A_{|\mathcal{F}|}} \hat{x}_{\mathcal{S}} \right)} \\
	=&\mathbf{DO^{-M}_{|\mathcal{F}|}\left( I-\Lambda_{A_{|\mathcal{F}|}} \right)  O^{M}_{|\mathcal{F}|}x} \\
	=&\mathbf{D\left( x-A^{M}x\right)},
\end{aligned}
\]
where, $\mathbf{x}\in \mathrm{BL}_{|\mathcal{F}|} \left(  \mathbf{O}^{\mathbf{M}} \right)$, and $\mathbf{\hat{x}}_{|\mathcal{F}|}=\mathbf{\hat{x}}_{\mathcal{S}}$.

By reordering and rearranging the eigenvalues in the matrix of the GLCT, \textbf{Theorem} \ref{thm4} applies to all graph signals that are $\mathbf{M}$-bandlimited in the GLCT domain.

\subsection{Optimal Sampling Operator}
When sampling graph signals, aside from determining the appropriate sample size, devising a strategy for selecting sampling node locations is crucial as the positions of the sampled nodes play a pivotal role in the performance of the reconstruction algorithm. Given the multiple choices of $|\mathcal{F}|$ linearly independent rows in $\mathbf{O}^{-\mathbf{M}}_{|\mathcal{F}|}$, our objective is to opt for an optimal set that minimizes the impact of noise. One strategy involves selecting sampling positions to minimize the normalized MSE (NMSE) \cite{GUncertainty,GFTsamp,GFTEfficient,GFTSSS,GPractical,Gdualizing,GFRFTsamp,Ggeneralizedsamp}. Consider the noise $e$ introduced during the sampling process
\[ \mathbf{x}_\mathcal{S} = \mathbf{Dx} + e, \]
where $\mathbf{D}$ is qualified. The recovered signal $\mathbf{x}_{\mathcal{R}}$ is then given by
\[ \mathbf{x}_{\mathcal{R}} = \mathbf{Rx}_\mathcal{S} = \mathbf{RDx} + \mathbf{R}e = \mathbf{x} + \mathbf{R}e. \]
Consequently, the bound of the recovery error, obtained through the norm and Cauchy-Schwarz inequality, is
\[ ||\mathbf{x}_{\mathcal{R}} - \mathbf{x}||_2 = ||\mathbf{R}e||_2 = ||\mathbf{O}^{-\mathbf{M}}_{|\mathcal{F}|}\mathbf{P}e||_2 \leq ||\mathbf{O}^{-\mathbf{M}}_{|\mathcal{F}|}||_2 ||\mathbf{P}||_2||e||_2. \]
Given that $||\mathbf{O}^{-\mathbf{M}}_{|\mathcal{F}|}||_2$ and $||e||_2$ are fixed, the objective is to minimize the spectral norm of $\mathbf{P}$, which is the pseudo-inverse of $\mathbf{D} \mathbf{O}^{-\mathrm{M}}_{|\mathcal{F}|}$. The norm considered here can be either the 2-norm or the Frobenius norm. The optimal norm is selected based on optimization algorithms and numerical methods.

In the following, we provide several alternative strategies for selecting sampling sets in the GLCT domain based on the A, D, E, T-optimal design approaches proposed by \cite{GUncertainty,GFTsamp,GFTEfficient,GFTSSS,Design}.

\textit{1) Minimization of the Frobenius Norm of $\mathbf{P}$ (MinFro):} In this strategy, the objective is to directly minimize the Frobenius norm of the matrix $\mathbf{P}$, aiming to minimize the NMSE between the original and recovery signals. The method involves selecting columns from the matrix $\mathbf{O}^{-\mathbf{M}}_{|\mathcal{F}|}$ to minimize the Frobenius norm of the pseudo-inverse matrix $\mathbf{P}$, which can reduce the impact of noise in the restored signal, thereby improving the energy concentration of the signal,
	\begin{equation}
		\mathbf{D}^\mathrm{opt}=\arg\min_\mathbf{D}~||\mathbf{P}||_{F}=\arg\min_\mathbf{D}~||\mathbf{\mathbf{\Sigma_{\mathcal{F}}O^{M}D}}||_{F}.
	\end{equation}
	
\textit{2) Maximizing signal energy concentration in the vertex domain (MaxVertex):} This approach aims to select sampling nodes that capture the maximum signal energy, ensuring that the signal energy is maximized at the sampling nodes,
		\begin{equation}
			\mathbf{D}^\mathrm{opt} = \arg\max_{\mathbf{D}} \sum^{ \mathcal{S}}_{i=1} \left| \mathbf{x}_i \right|^2=\arg\max_\mathbf{D}~\left| \mathbf{Dx}  \right|^{2}.
	\end{equation}
	
\textit{3) Maximizing signal energy concentration in the spectral domain (MaxSpec):} This method selects nodes in the spectral domain that capture the maximum signal energy, i.e., the set of nodes with the highest spectral energy, ensuring optimal concentration of signal at sampling points,
		\begin{equation}
			\mathbf{D}^\mathrm{opt}=\arg\max_\mathbf{D}~\sum^{\left| \mathcal{S}\right|  }_{i=1} \left| \left( \mathbf{O^{M}x}\right)_{i}  \right|^{2}=\arg\max_\mathbf{D}~\left| \mathbf{O^{M}Dx}  \right|^{2}.
		\end{equation}
	
\textit{4) Maximization of the Volume of the Parallelepiped Formed With the Columns of $\mathbf{O}^{-\mathbf{M}}_{|\mathcal{F}|}$ (MaxVol):} This strategy aims to select a set $\mathcal{S}$ of columns from the matrix $\mathbf{O}^{-\mathbf{M}}_{|\mathcal{F}|}$ to maximize the volume of the parallelepiped constructed with the chosen columns. Essentially, this approach involves completing the matrix $\mathbf{O}^{-\mathbf{M}}_{|\mathcal{F}|}$. The volume can be computed as the determinant of the matrix $\mathbf{O}^{\mathbf{M}}_{|\mathcal{F}|}  \mathbf{D}\mathbf{O}^{-\mathbf{M}}_{|\mathcal{F}|}$, which will improve the condition number of the sampling matrix and thus enhance the energy concentration of the signal,
	\begin{equation}
		\mathbf{D}^{\mathrm{opt} } =\arg \max_{\mathbf{D} }~\prod^{|\mathcal{F}|}_{i=1} \lambda_{i} \left( \mathbf{O}^{\mathbf{M}}_{|\mathcal{F}|} \mathbf{D}\mathbf{O}^{-\mathbf{M}}_{|\mathcal{F}|}\right).
	\end{equation}			
	
\textit{5) Minimization of the singular value of $\mathbf{P}$ (MinPinv):} Due to $\lambda_{i} \left( \mathbf{B^{M}DB^{M}}\right)  =\sigma^{2}_{i} \left( \mathbf{B^{M}D}\right)  =\sigma^{2}_{i} \left( \mathbf{\Sigma_{\mathcal{F}} O^{M}D}\right)  =\sigma^{2}_{i} \left( \mathbf{P}\right) $, in the case of uncorrelated noise, this is equivalent to minimizing $\sum^{\left| F\right|  }_{i} \sigma^{2}_{i}$, which will reduce the amplification of noise in the signal recovery process,
	\begin{equation}
		\mathbf{D}^\mathrm{opt}=\arg\min_\mathbf{D}~\sum^{\left| \mathcal{F}\right|  }_{i=1} \sigma^{2}_{i}(\mathbf{P}).
	\end{equation}
	
\textit{6) Maximization of the minimum singular value of $\mathbf{D} \mathbf{O}^{-\mathrm{M}}_{|\mathcal{F}|}$ (MaxSigMin):} The strategy is designed to leverage the 2-norm to minimize NMSE, which is equivalent to minimizing the largest singular value of $\mathbf{P}$. In other words, we aim to maximize the smallest singular value $\sigma_\mathrm{min}$ of $\mathbf{D}\mathbf{O}^{-\mathbf{M}}_{|\mathcal{F}|}$ for every qualified $\mathbf{D}$, thereby reducing the impact of noise in the recovered signal,
	\begin{equation}
		\mathbf{D}^\mathrm{opt}=\arg\max_\mathbf{D}~\sigma_\mathrm{min}(\mathbf{DO}^{-\mathbf{M}}_{|\mathcal{F}|}).
	\end{equation}
	
\textit{7) Maximization of the singular value of $\mathbf{\mathbf{D} \mathbf{O}^{-\mathbf{M}}_{|\mathcal{F}|}}$ (MaxSig):} The strategy aims to select columns of the matrix $\mathbf{O}^{-\mathrm{M}}_{|\mathcal{F}|}$ to maximize its singular values. Although this strategy is not directly related to NMSE optimization, it is easy to implement and demonstrates good performance in practice, improving the condition number of the sampling matrix and thus enhancing the energy concentration of the signal,
		\begin{equation}
			\mathbf{D}^\mathrm{opt}=\arg\max_\mathbf{D}~\sum^{\left| \mathcal{F}\right|  }_{i=1} \sigma^{2}_{i}(\mathbf{\mathbf{D} \mathbf{O}^{-\mathbf{M}}_{|\mathcal{F}|}}).
		\end{equation}
The maximization of signal energy concentration in 2) and 3) is equivalent to ensuring perfect localization of the graph signal in the vertex and spectral domains, as described by Eqs. \ref{Dxx} and \ref{Bxx}.

	\begin{remark}
	Given $ \mathbf{x = B^{M}y}$, where $\mathbf{y}$ is fixed, maximizing signal energy concentration is tantamount to maximizing $|\mathbf{DB^{M}}|^2$ or $|\mathbf{O^{M}DB^{M}}|^2$. Since the operator matrix $\mathbf{B^{M}DB^{M}}$ is a positive semi-definite, 2) and 3) are approximately equivalent to the A-optimal design and D-optimal design \cite{Design}.
\end{remark}

By achieving perfect localization of signals in both vertex and spectral domains, it implies that the signals primarily concentrate on specific parts or nodes of the graph. Maximizing their energy on the sampling subset $\mathcal{S}$ ensures that the sampled signals encompass the main local structures and features of the graph, thereby capturing the essential information and dynamics of the signals more effectively. Thus, the perfect localization of signals in vertex and spectral domains is closely intertwined with maximizing signal energy on the sampling subset $\mathcal{S}$. We summarize the A, D, E, T-optimal designs related to the seven sampling strategies along with their objective functions in Table \ref{tab1}.
\begin{table}
	\caption{Optimal design and objective function of sampling methods}
	\label{tab1}
	%\small
	\footnotesize
	\begin{tabular}{clll}
		\toprule
		Optimal Designs&Objectives&Objective Functions & Methods\\
		\midrule
		A-optimal & $\min\mathrm{tr} \left[ \left(\mathbf{B^{M}DB^{M}} \right)^{-1}  \right]$  &$\min \  \left| \left| \mathbf{\Sigma_{\mathcal{F}} O^{M}D} \right|  \right|_{F}$&MinFro \\
		A-optimal & $\min\mathrm{tr} \left[ \left(\mathbf{B^{M}DB^{M}} \right)^{-1}  \right]$ &$	\max \left|  \mathbf{Dx}  \right|^{2}$ &MaxVertex\\
		D-optimal & $\min \mathrm{det} \left[\left(\mathbf{B^{M}DB^{M}}\right)^{-1}\right]$ &$	\max \left|  \mathbf{O^{M}Dx}  \right|^{2}$ &MaxSpec\\
		D-optimal & $\min \mathrm{det} \left[\left(\mathbf{B^{M}DB^{M}}\right)^{-1}\right]$ &$\max \prod^{|\mathcal{F}|}_{i=1} \lambda_{i} \left( \mathbf{O}^{\mathbf{M}}_{|\mathcal{F}|} \mathbf{D}\mathbf{O}^{-\mathbf{M}}_{|\mathcal{F}|}\right)$ &MaxVol\\
		E-optimal & $\min \left| \left| \left(  \mathbf{DB^{M}}\right)^{\dag }  \right| \right|_{2}$ &$\min \sum^{\left| \mathcal{F}\right|  }_{i=1} \sigma^{2}_{i} \left( \mathbf{\Sigma_{\mathcal{F}} O^{M}D}\right)  $&MinPinv \\
		E-optimal & $\min \left| \left| \left(  \mathbf{DB^{M}}\right)^{\dag }  \right| \right|_{2}$&$\max \  \sigma_{\min } \left(\mathbf{D O^{-M}_{|\mathcal{F}|}}\right)$ &MaxSigMin\\
		T-optimal & $\max \mathrm{tr} \left[ \left(\mathbf{B^{M}DB^{M}} \right)  \right]$ &$\max \sum^{\left| \mathcal{F}\right|  }_{i=1} \sigma^{2}_{i}\left(\mathbf{\mathbf{D} \mathbf{O}^{-\mathbf{M}}_{|\mathcal{F}|}}\right)$ &MaxSig\\
		\bottomrule
	\end{tabular}
\end{table}

As the graph signals are all $\mathbf{M}$-bandlimited after satisfying the localization condition, i.e., $\mathbf{x} \in \mathrm{BL}_{|\mathcal{F}|}(\mathbf{O^{M}})$, in other words, for the vertex set $\mathcal{V}$, its sampling subset $\mathcal{S}$ contains all the features and information of its graph signal. It satisfies the property of submodular functions, meaning that the marginal benefit of adding locally decreases with the increase of selected elements. More specifically, when we add an element to subset $\mathcal{S}$, although it may increase the number of covered elements, the marginal contribution of each newly added element gradually diminishes as the selected elements increase. Therefore, to address the sampling strategies mentioned above, we employ a greedy algorithm to tackle this selection challenge, aiming to determine the \textit{optimal sampling operator}. Greedy algorithms have been shown to closely approximate the global optimum when solving similar matrix approximation optimization problems \cite{Greedy}. Since strategies 1), 2), and 7) are essentially equivalent, as are strategies 3) and 4), and strategies 5) and 6), we present Algorithm \ref{alg1}, which summarizes the three sampling strategies 4), 6), and 7).
\begin{algorithm}[H]
	\caption{Optimal Sampling Operator based on MaxSig, MaxSigMin, or MaxVol}\label{alg1}
		\footnotesize
	\begin{algorithmic}
		\STATE \textbf{Input:} $\mathbf{O}^{-\mathbf{M}}_{|\mathcal{F}|}$ (first $|\mathcal{F}|$ columns of $\mathbf{O}^{-\mathbf{M}}$), $|\mathcal{S}|$ (number of samples), $N$ (number of rows in $\mathbf{O}^{-\mathbf{M}}_{|\mathcal{F}|}$)
		\STATE \textbf{Output:} $\mathcal{S}$ (sampling set)
		\FOR{$k = 1:|\mathcal{S}|$}
		\STATE $s = \arg\max_{j\in \{1:N\}-\mathcal{S}}\sum^{|\mathcal{F}|}_{i=1} \sigma^{2}_{i} \left(\left( \mathbf{O}^{-\mathbf{M}}_{|\mathcal{F}|}\right)_{\mathcal{S}\cup \{j\}} \right)$
		\STATE \quad \textbf{or} $s = \arg\max_{j\in \{1:N\}-\mathcal{S}}\sigma_\mathrm{min} \left(\left(\mathbf{O}^{-\mathbf{M}}_{|\mathcal{F}|}\right)_{\mathcal{S}+\{j\}}\right)$
		\STATE \quad \textbf{or} $s = \arg\max_{j\in \{1:N\}-\mathcal{S}} \prod^{|\mathcal{F}|}_{i=1} \lambda_{i} \left( \left( \left( \mathbf{O}^{-\mathbf{M}}_{|\mathcal{F}|} \right)^{\ast }_{\mathcal{S}\cup \{j\}} \left( \mathbf{O}^{-\mathbf{M}}_{|\mathcal{F}|} \right)_{\mathcal{S}\cup \{j\}} \right) \right)$
		\STATE $\mathcal{S} = \mathcal{S} + \{s\}$
		\ENDFOR
		\RETURN $\mathcal{S}$
	\end{algorithmic}
\end{algorithm}

We compared our proposed seven sampling strategies with random sampling under the GLCT sampling framework. As shown in Fig. \ref{fig5}, the NMSE is reported, defined as the error between the original and recovered signals divided by the original signal for each node. (a) is based on a ring graph with N = 32 nodes, and the signal is a $\mathbf{M}$-bandlimited signal with parameters $\alpha=0.8, \beta=32, \xi=0.5k+1, k=1,...,N$. The frequency bandwidth is $|\mathcal{F}|=4$. (b) is based on a swiss roll graph with N = 256 nodes, and the signal is a $\mathbf{M}$-bandlimited signal with same parameters. The frequency bandwidth is $|\mathcal{F}|=32$. We compare the NMSE of eight sampling methods as the number of sampled nodes increases for two different graph signals. We observed that, with a small number of sampling nodes, all methods except random sampling could effectively recover the signal.
% Alogrithm 1 of strategy 2)
\begin{figure}[h]
	\begin{center}
		\begin{minipage}[t]{0.45\linewidth}
			\centering
			\includegraphics[width=\linewidth]{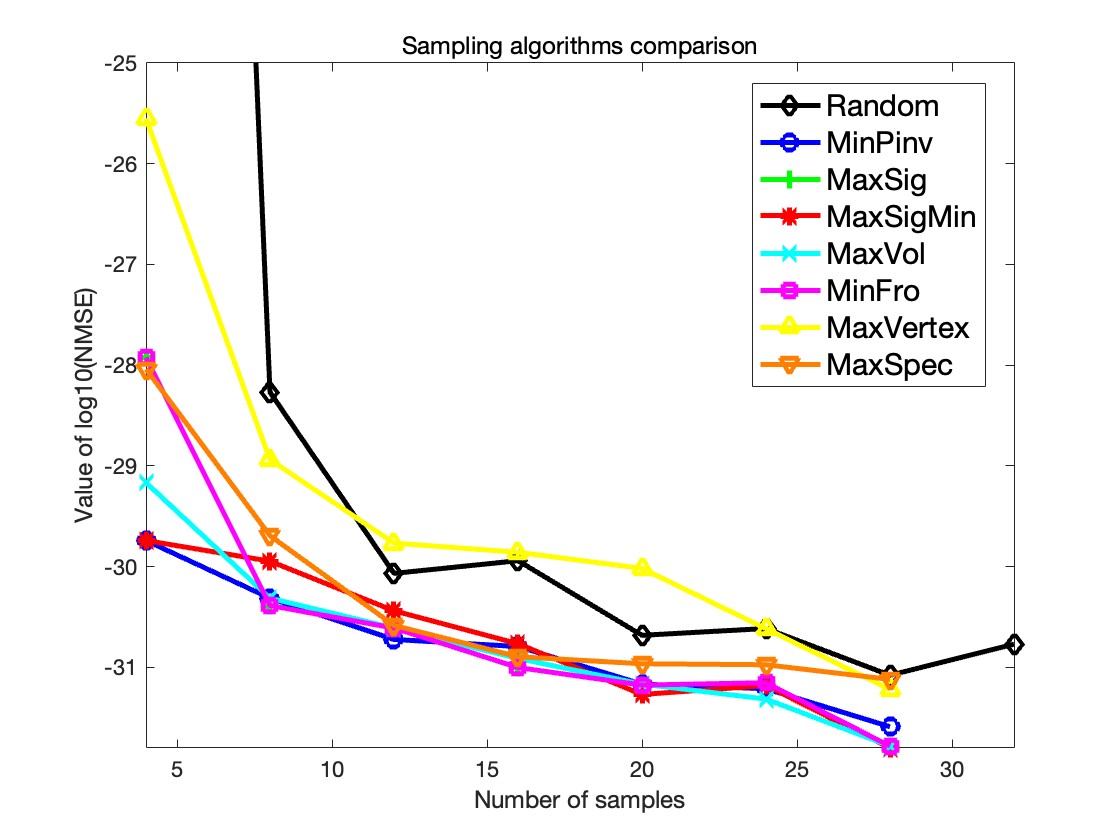}
			\parbox{1.5cm}{\tiny(a) Ring graph.}
		\end{minipage}
		\begin{minipage}[t]{0.45\linewidth}
			\centering
			\includegraphics[width=\linewidth]{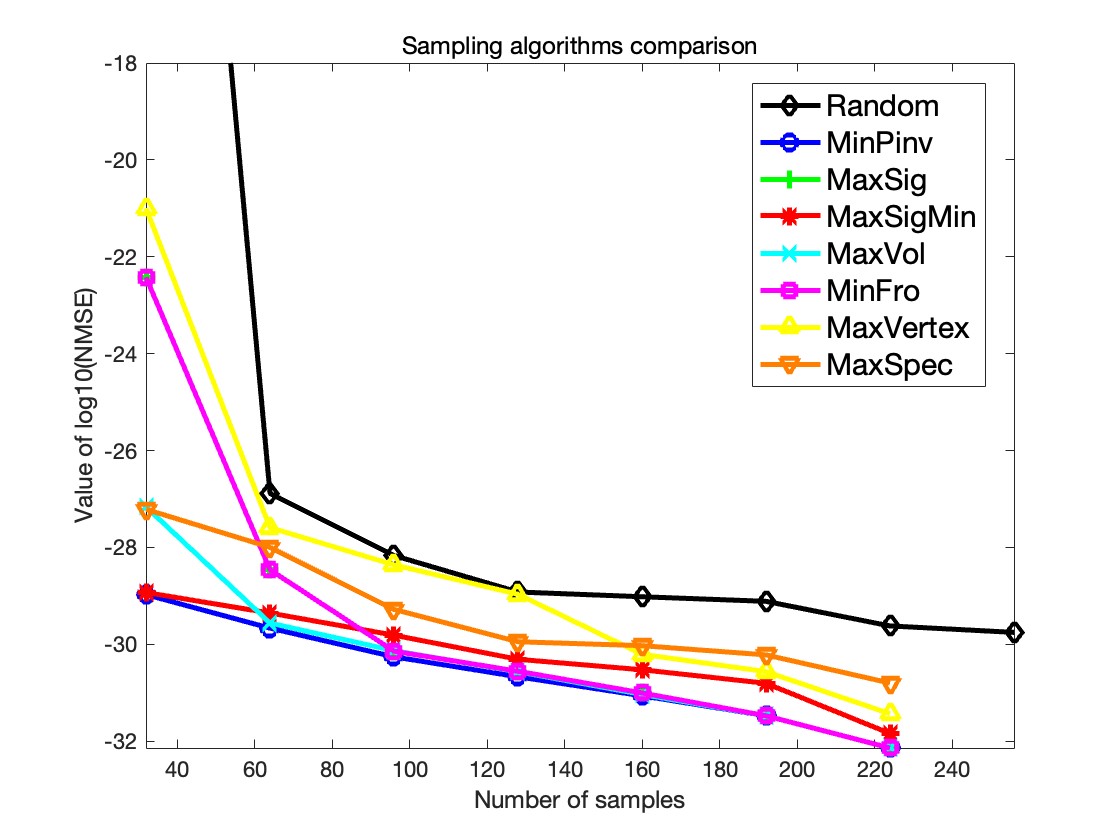}
			\parbox{1.7cm}{\tiny (b) Swiss roll graph.}
		\end{minipage}
	\end{center}
	\caption{NMSE vs. number of samples for different sampling strategies.}
	\vspace*{-3pt}
	\label{fig5}
\end{figure}

%\subsection{Fast Sampling Strategy}
\section{Applications and Numerical Experiments}
\label{5}
In this section, we conduct simulation experiments and test GLCT sampling and recovery in a semi-supervised classification application \cite{GFRFTsamp} and clustering of bus test cases \cite{IEEE118}. Semi-supervised classification seeks to categorize a vast dataset with the assistance of a limited number of labeled instances. In the context of GSP, the smoothness of labels across the graph can be observed through low-pass characteristics in the graph spectral domain. Simultaneously, in the case of the bus test case with 118 nodes, the generator dynamically produces smooth graph signals, justifying the validity of the $\mathbf{M}$-bandlimited assumption. We utilize four frameworks (Laplacian-based GFT sampling, GFT sampling, GFRFT sampling, and GLCT sampling) for GSP, comparing the NMSE of graph signal recovery.
\subsection{Simulations}
We conducted simulations using the graph signal processing toolbox (GSPBox) \cite{GSPBox} for two examples. The first example involves a David sensor network with $N=500$ nodes. We set the bandwidth $|\mathcal{F}|=60$, and the $\mathbf{M}$-bandlimited parameters to $\alpha=0.7, \beta=N/2, \xi=0.5k+0.5, k=1,...,N$. \textcolor{red}{The graph signal $\mathbf{x}$ is the real part of $\mathbf{f}_1 + 0.5 \mathbf{f}_2 + 2\mathbf{f}_3$, i.e., $\mathbf{x}=\mathrm{real}(\mathbf{f}_1 + 0.5 \mathbf{f}_2 + 2\mathbf{f}_3)$} obtained through $\mathbf{M}$-bandlimited operation, where $\mathbf{f}_i$ represents the $i$-th column of $\mathbf{O^{-M}}$. The second example involves the vehicular traffic data in Rome. We set the bandwidth $|\mathcal{F}|=50$, and the $\mathbf{M}$-bandlimited parameters to $\alpha=0.9, \beta=N, \xi=0.5k+0.7, k=1,...,N$. The graph signal is the real vehicular traffic data in Rome \cite{GsampRome} obtained through $\mathbf{M}$-bandlimited operation. The original graph signals are depicted in Fig. \ref{fig6}.
\begin{figure}[h]
	\begin{center}
		\begin{minipage}[t]{0.45\linewidth}
			\centering
			\includegraphics[width=\linewidth]{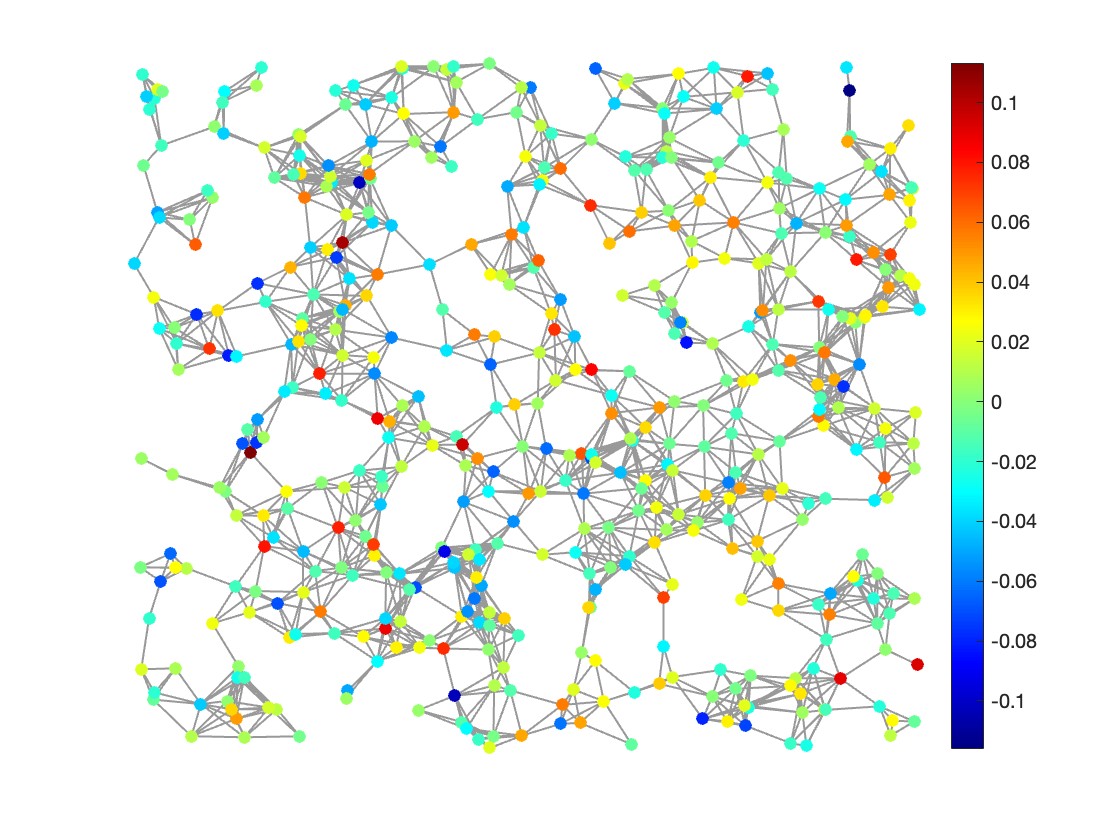}
			\parbox{5cm}{\tiny(a) Example 1: $500$-node David sensor network graph.}
		\end{minipage}
		\begin{minipage}[t]{0.45\linewidth}
			\centering
			\includegraphics[width=\linewidth]{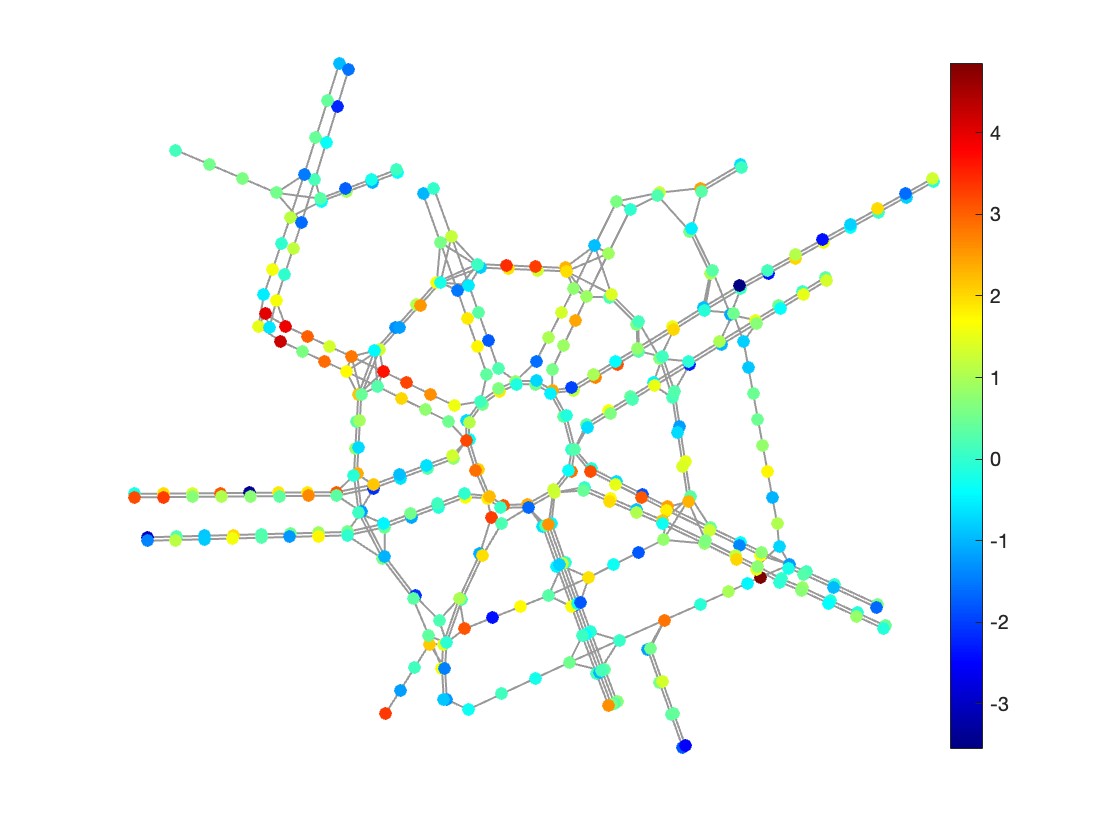}
			\parbox{5cm}{\tiny (b) 	Example 2: Roman vehicular traffic data graph.}
		\end{minipage}
	\end{center}
	\caption{Original graph signal.}
	\vspace*{-3pt}
	\label{fig6}
\end{figure}

For both examples, we performed sampling and recovery of the graph signals using four sampling frameworks based on seven optimal sampling strategies. All signal recovery formulas utilize Eq. \eqref{recovery} from \textbf{Theorem} \ref{thm2}. The accuracy of the recovery was measured using the NMSE. Fig. \ref{fig7} illustrates the locations of sampled nodes using MinPinv sampling strategy. The number of sampled nodes, $|\mathcal{S}|$, is 60 for (a), (b), (c), and (d), and 50 for (e), (f), (g), and (h). Fig. \ref{fig8} displays the signals recovered through Laplacian-based GFT sampling in (a) and (e), GFT sampling in (b) and (f), GFRFT sampling in (c) and (g), and GLCT sampling in (d) and (h). The NMSE results for the David sensor network and Roman vehicular traffic data, across four sampling frameworks employing seven optimal strategies, are detailed in Tables \ref{tab2} and \ref{tab3}, respectively. Note that the parameters selected for our GFRFT and GLCT are consistent with those specified in the aforementioned two examples. 
%sampling nodes graph
\begin{figure}[!h]
	\begin{center}
		\begin{minipage}[t]{0.23\linewidth}
			\centering
			\includegraphics[width=\linewidth]{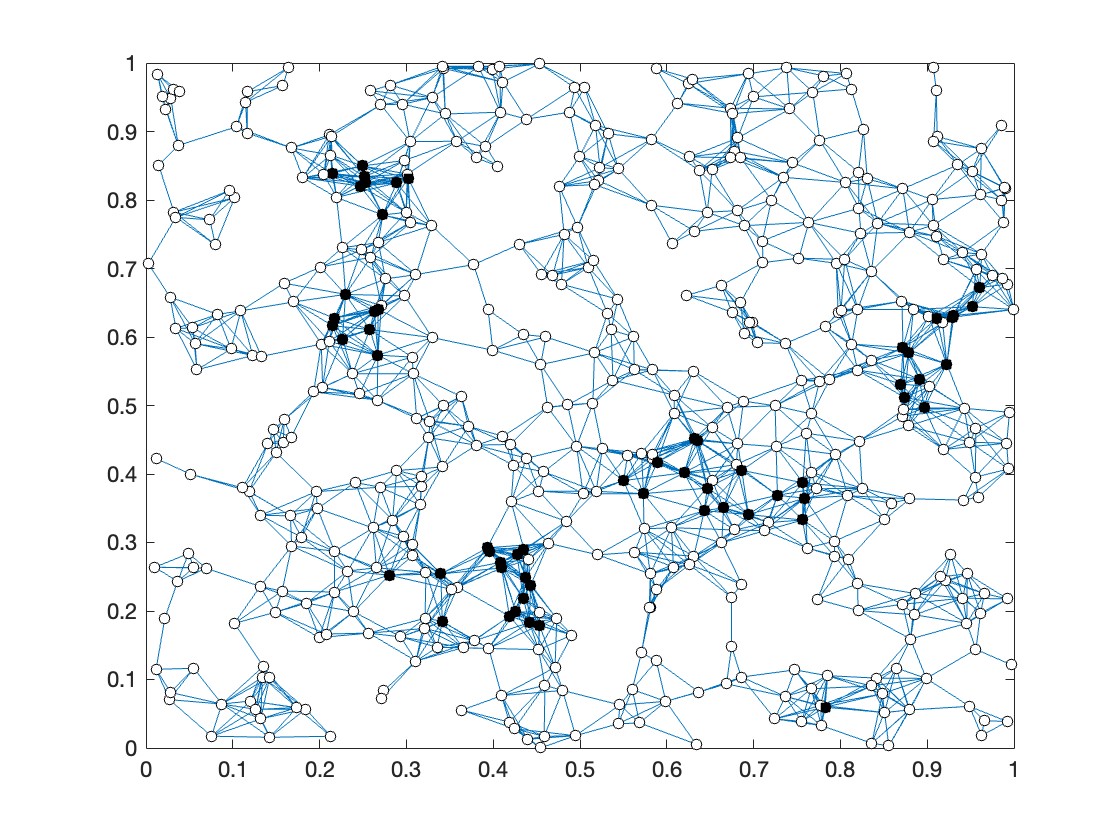} 
			\parbox{3.2cm}{\tiny (a) David sensor network samples from Laplacian-based GFT sampling.}
		\end{minipage}
		\begin{minipage}[t]{0.23\linewidth}
			\centering
			\includegraphics[width=\linewidth]{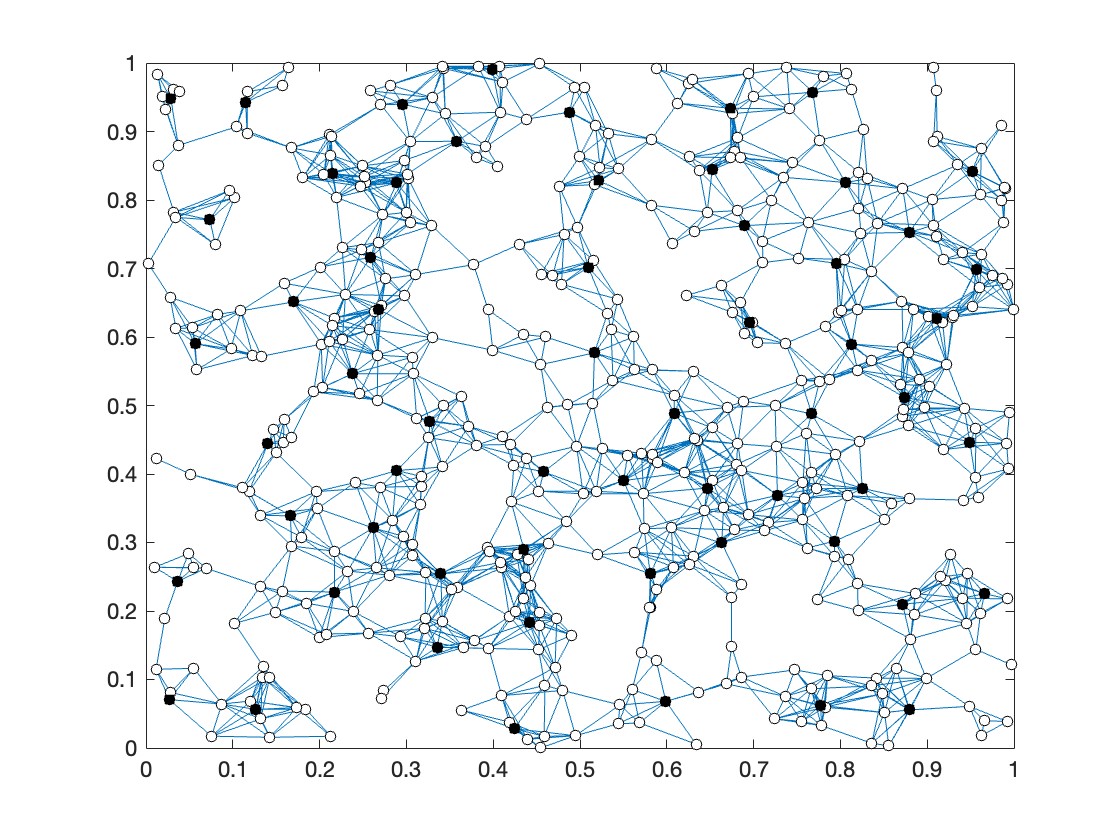}
			\parbox{3cm}{\tiny (b) David sensor network samples from GFT sampling.}
		\end{minipage}
		\begin{minipage}[t]{0.23\linewidth}
			\centering
			\includegraphics[width=\linewidth]{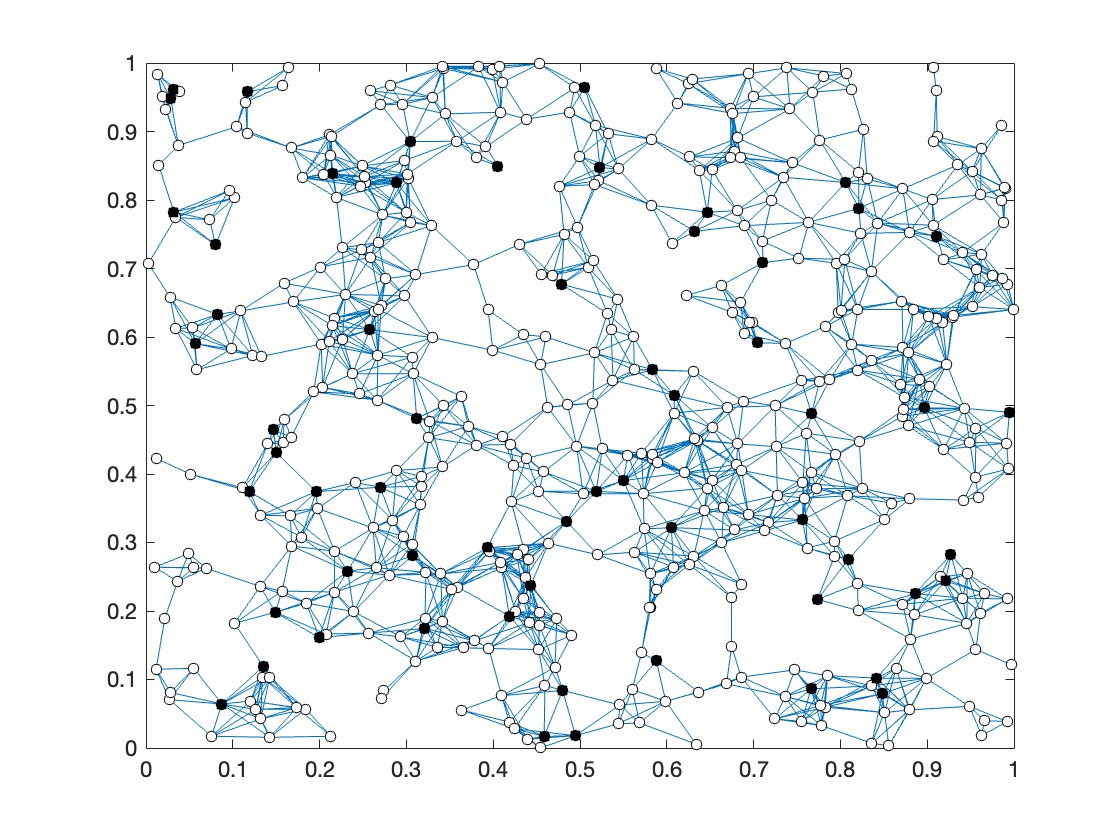}
			\parbox{3cm}{\tiny (c) David sensor network samples from GFRFT sampling.}
		\end{minipage}
		\begin{minipage}[t]{0.23\linewidth}
			\centering
			\includegraphics[width=\linewidth]{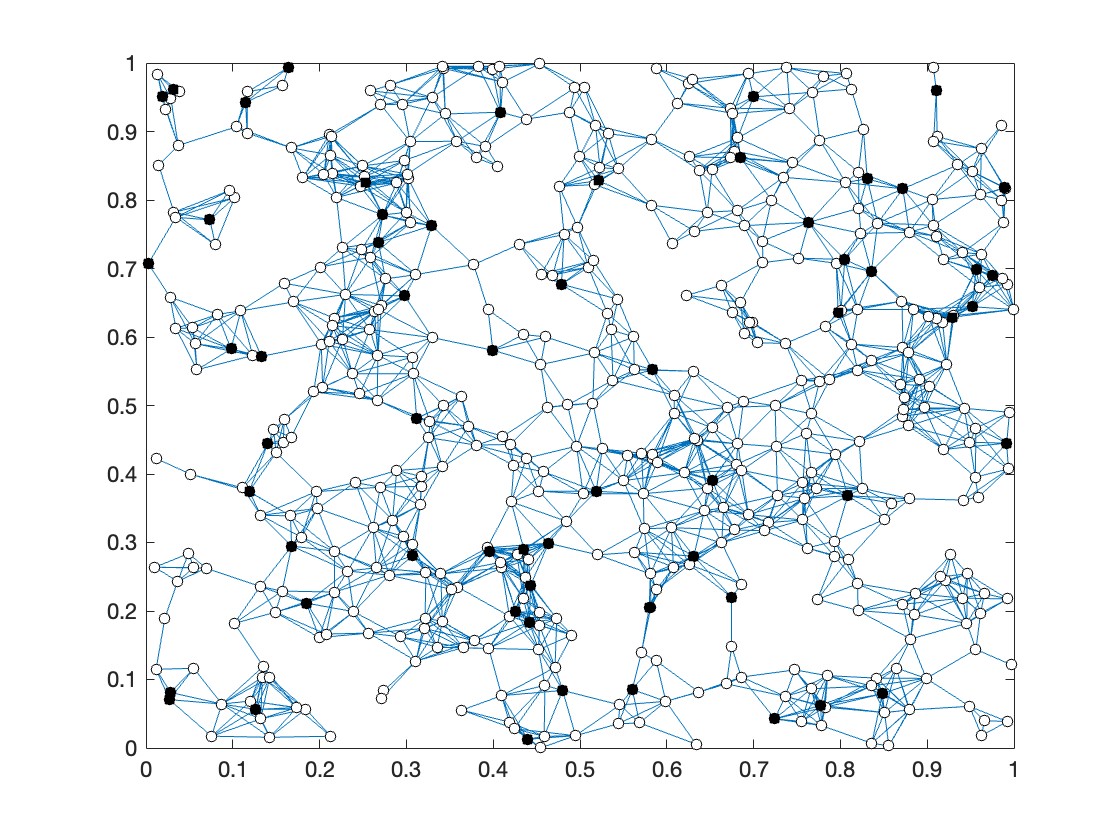} 
			\parbox{3cm}{\tiny (d) David sensor network samples from GLCT sampling.}
		\end{minipage}
		\begin{minipage}[t]{0.23\linewidth}
			\centering
			\includegraphics[width=\linewidth]{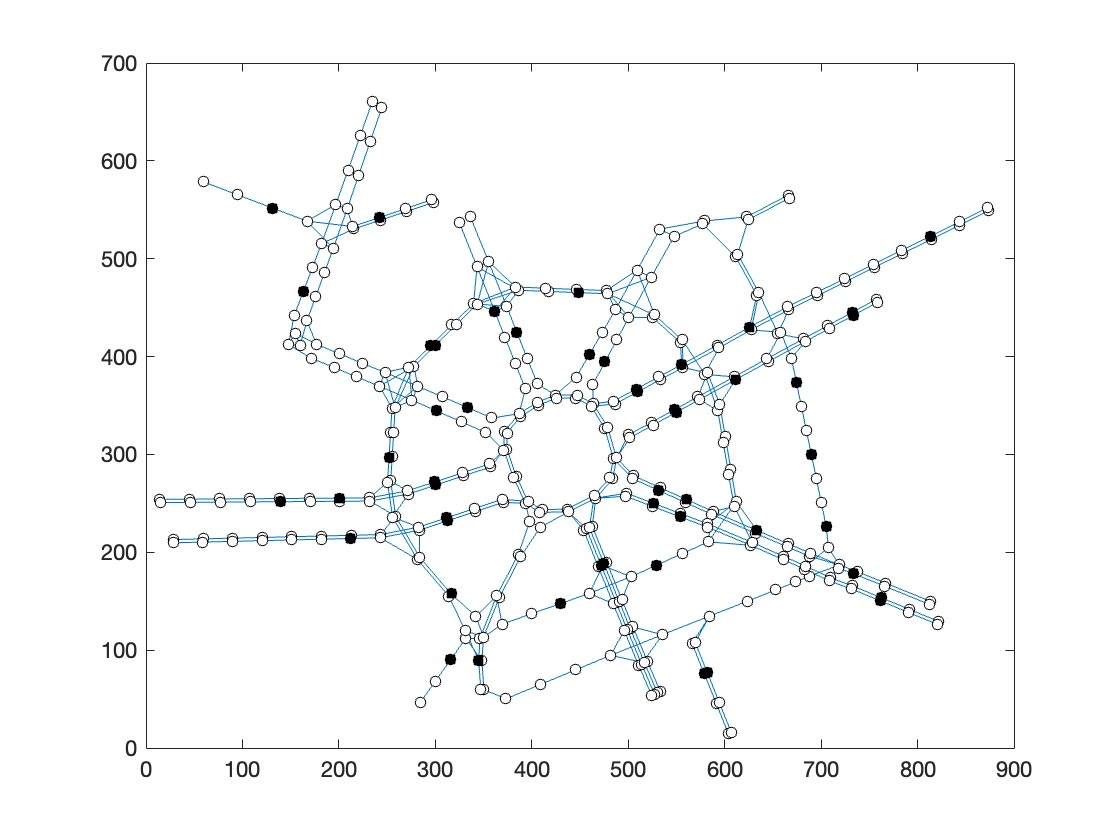} 
			\parbox{3.4cm}{\tiny (e) Roman vehicular traffic data samples from Laplacian-based GFT sampling.}
		\end{minipage}
		\begin{minipage}[t]{0.23\linewidth}
			\centering
			\includegraphics[width=\linewidth]{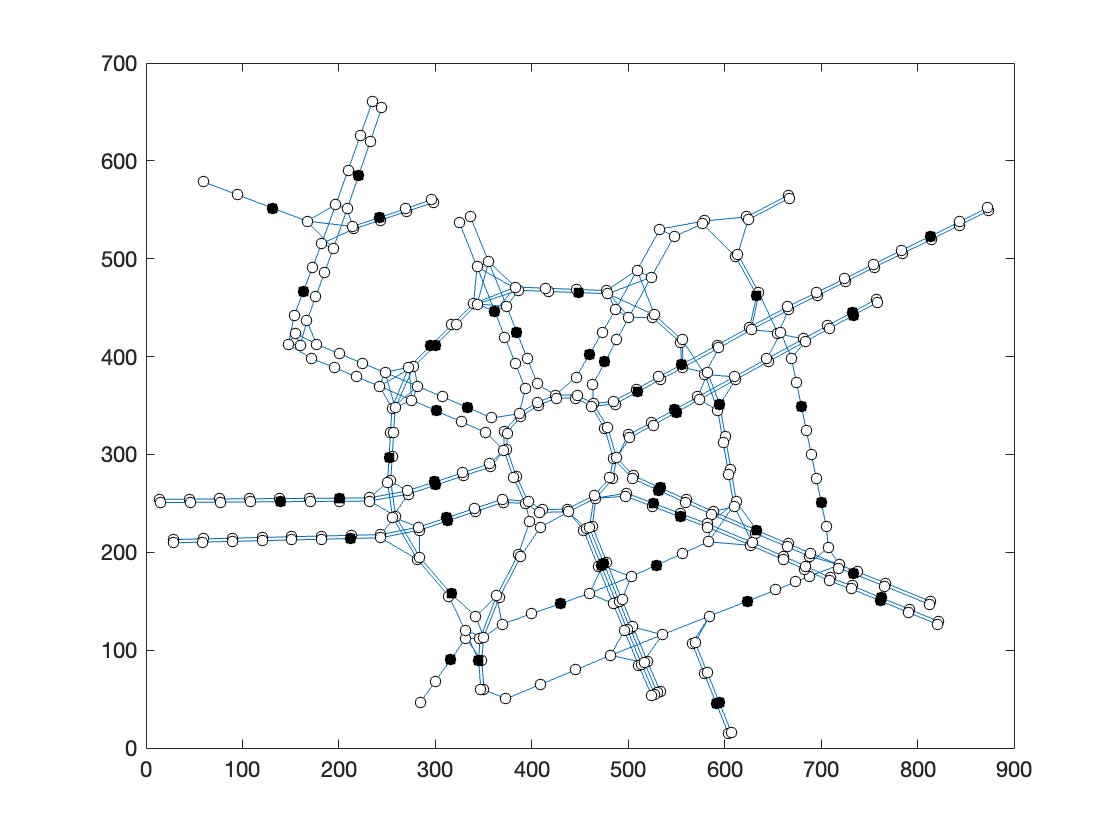}
			\parbox{2.7cm}{\tiny (f) Roman vehicular traffic data samples from GFT sampling.}
		\end{minipage}
		\begin{minipage}[t]{0.23\linewidth}
			\centering
			\includegraphics[width=\linewidth]{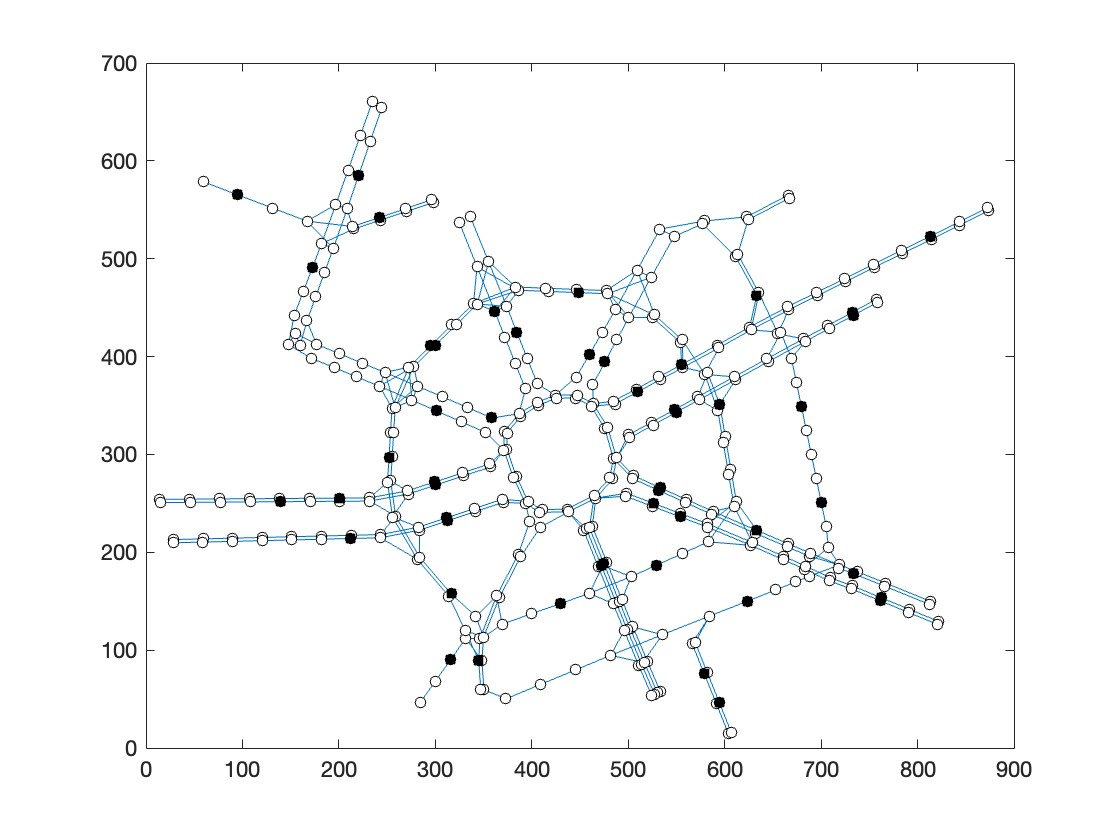}
			\parbox{3cm}{\tiny (g) Roman vehicular traffic data samples from GFRFT sampling.}
		\end{minipage}
		\begin{minipage}[t]{0.23\linewidth}
			\centering
			\includegraphics[width=\linewidth]{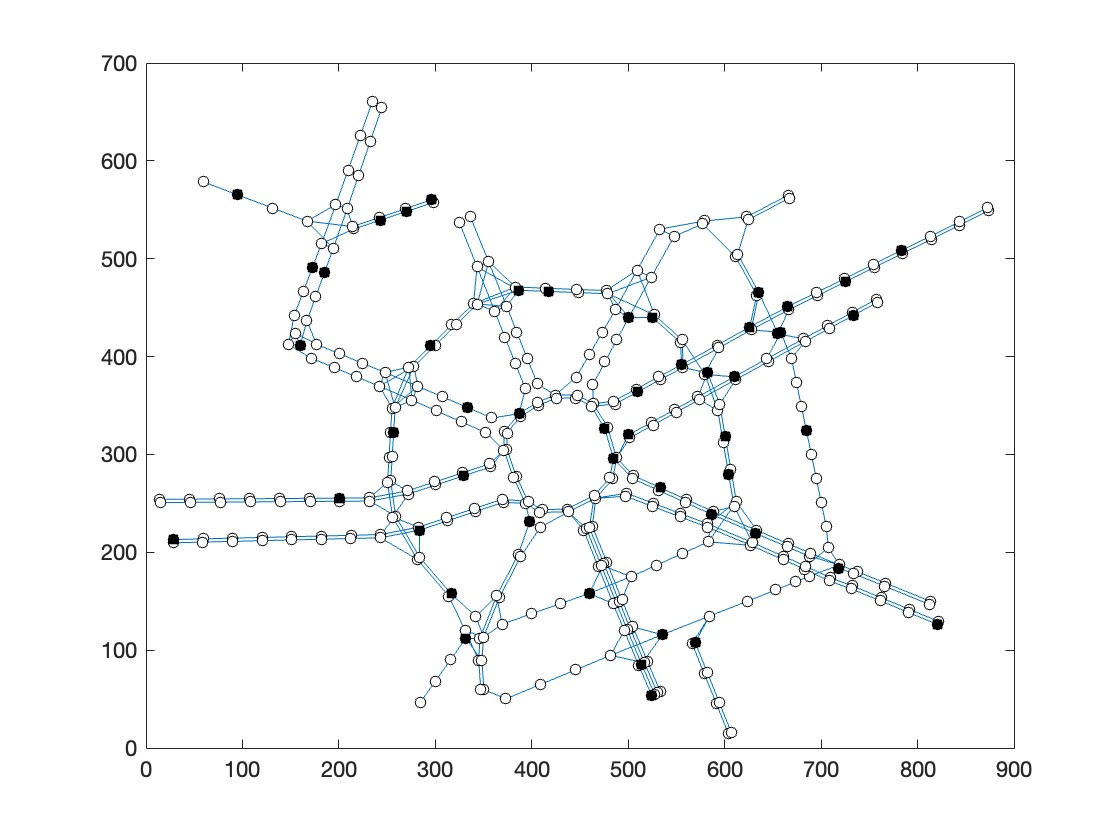} 
			\parbox{3cm}{\tiny (h) Roman vehicular traffic data samples from GLCT sampling.}
		\end{minipage}
	\end{center}
	\caption{The location of the sampled signals.}
	\vspace*{-3pt}
	\label{fig7}
\end{figure}
%recovery graph signals
\begin{figure}[!h]
	\begin{center}
		\begin{minipage}[t]{0.23\linewidth}
			\centering
			\includegraphics[width=\linewidth]{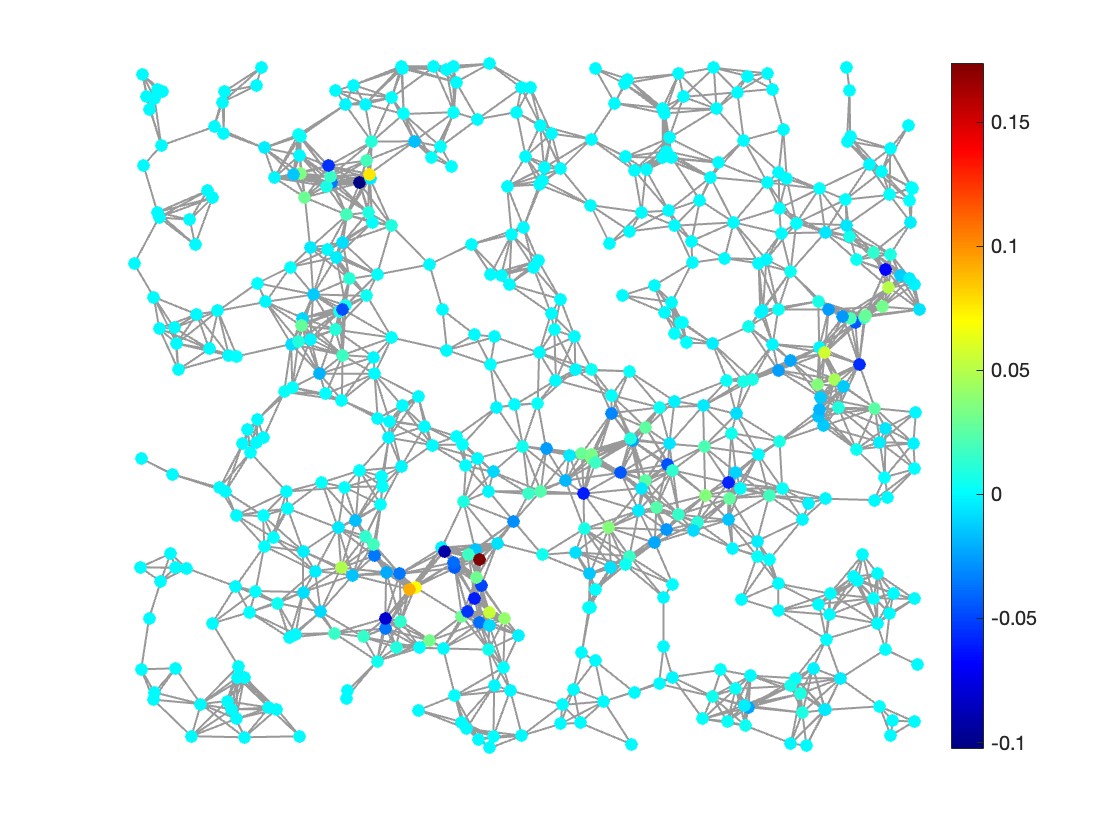}
			\parbox{3.2cm}{\tiny (a) David sensor network recovery from Laplacian-based GFT sampling.}
		\end{minipage}
		\begin{minipage}[t]{0.23\linewidth}
			\centering
			\includegraphics[width=\linewidth]{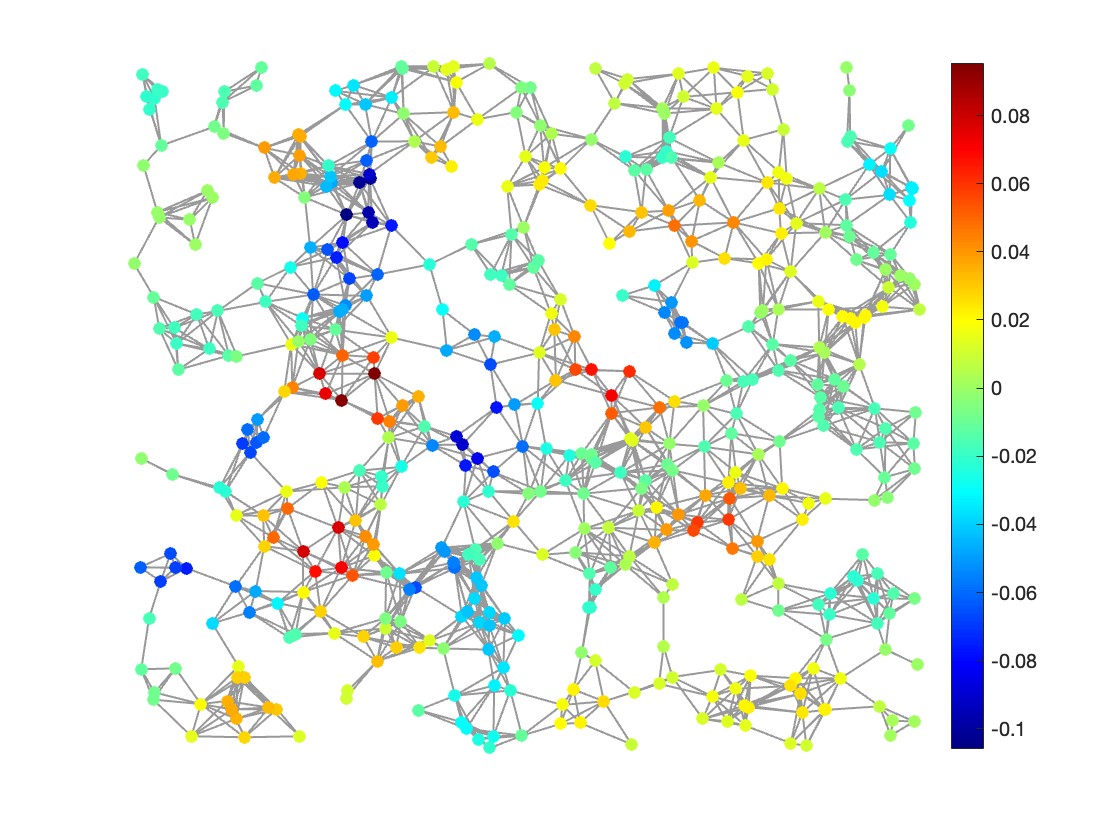}
			\parbox{3cm}{\tiny (b) David sensor network recovery from GFT sampling.}
		\end{minipage}
		\begin{minipage}[t]{0.23\linewidth}
			\centering
			\includegraphics[width=\linewidth]{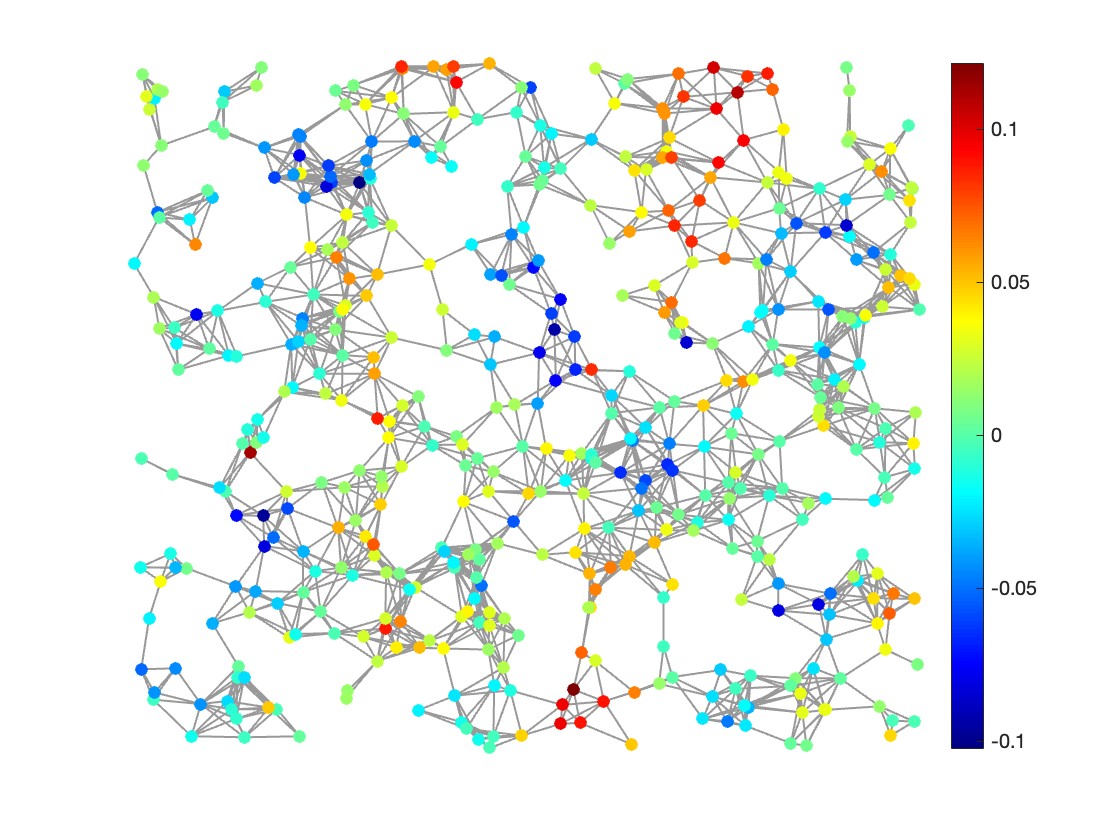}
			\parbox{3cm}{\tiny (c) David sensor network recovery from GFRFT sampling.}
		\end{minipage}
		\begin{minipage}[t]{0.23\linewidth}
			\centering
			\includegraphics[width=\linewidth]{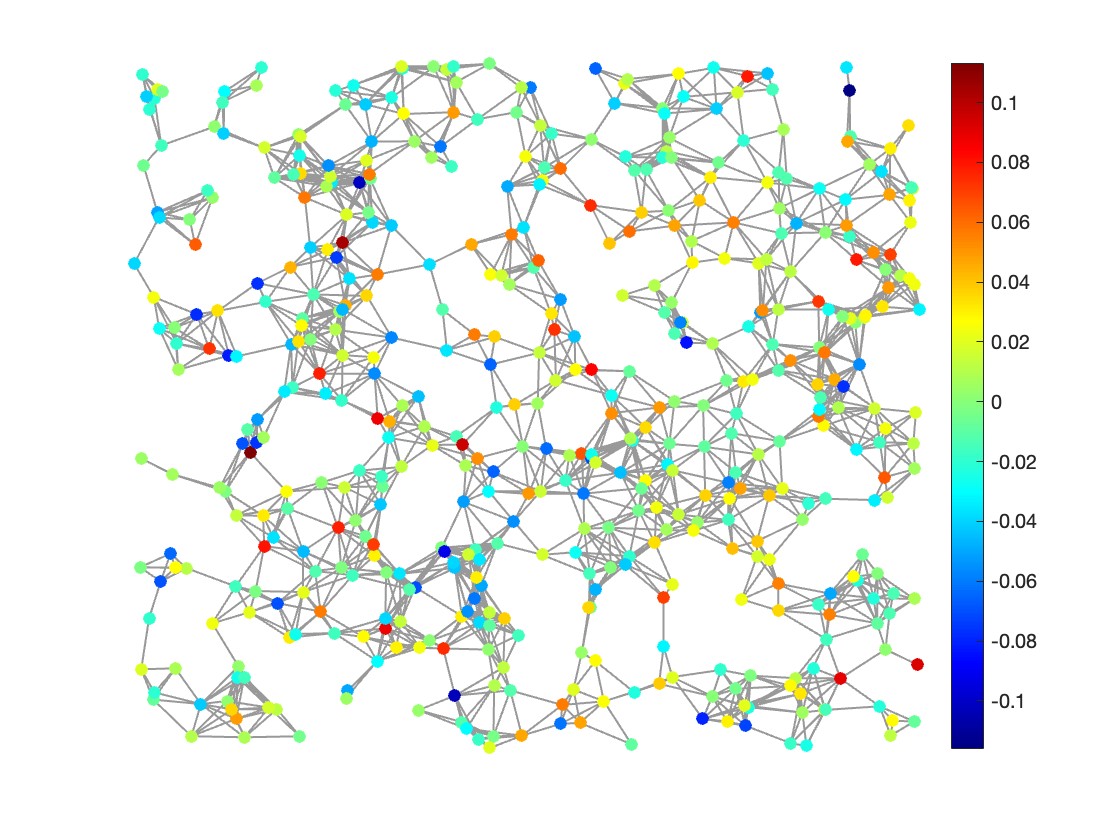} 
			\parbox{3cm}{\tiny (d) David sensor network recovery from GLCT sampling.}
		\end{minipage}
		\begin{minipage}[t]{0.23\linewidth}
			\centering
			\includegraphics[width=\linewidth]{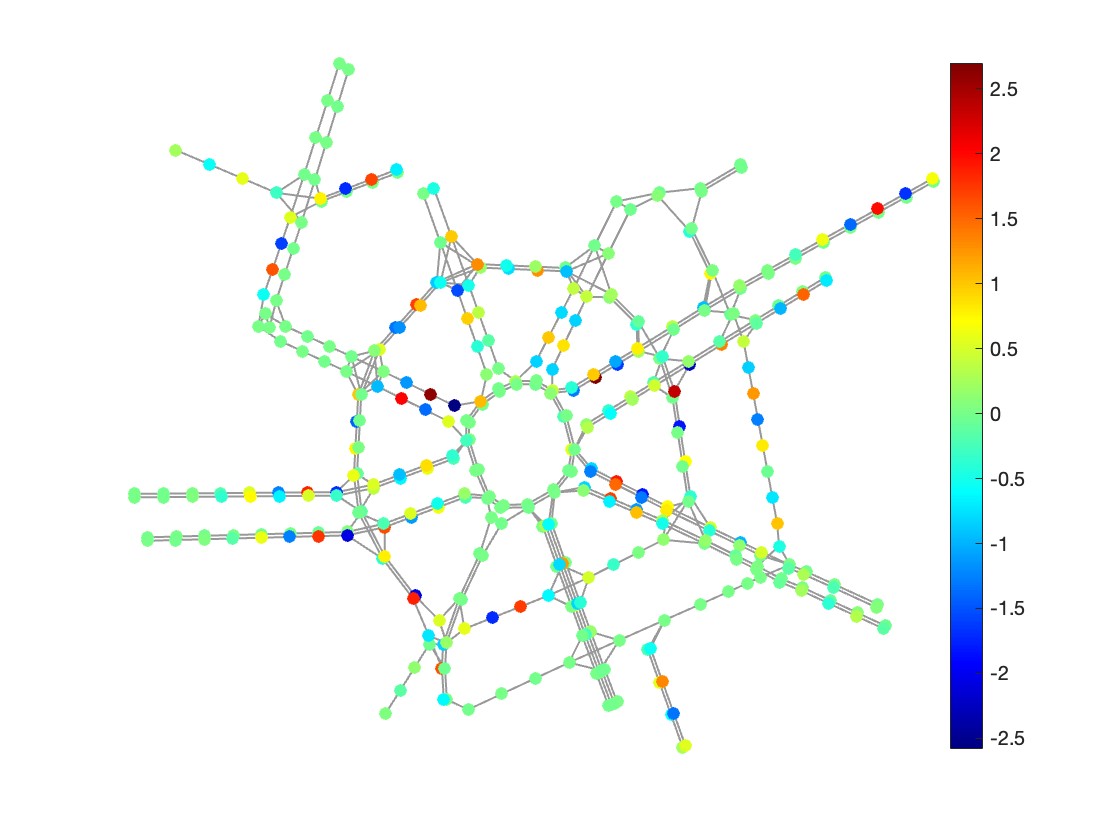}
			\parbox{3.41cm}{\tiny (e) Roman vehicular traffic data recovery from Laplacian-based GFT sampling.}
		\end{minipage}
		\begin{minipage}[t]{0.23\linewidth}
			\centering
			\includegraphics[width=\linewidth]{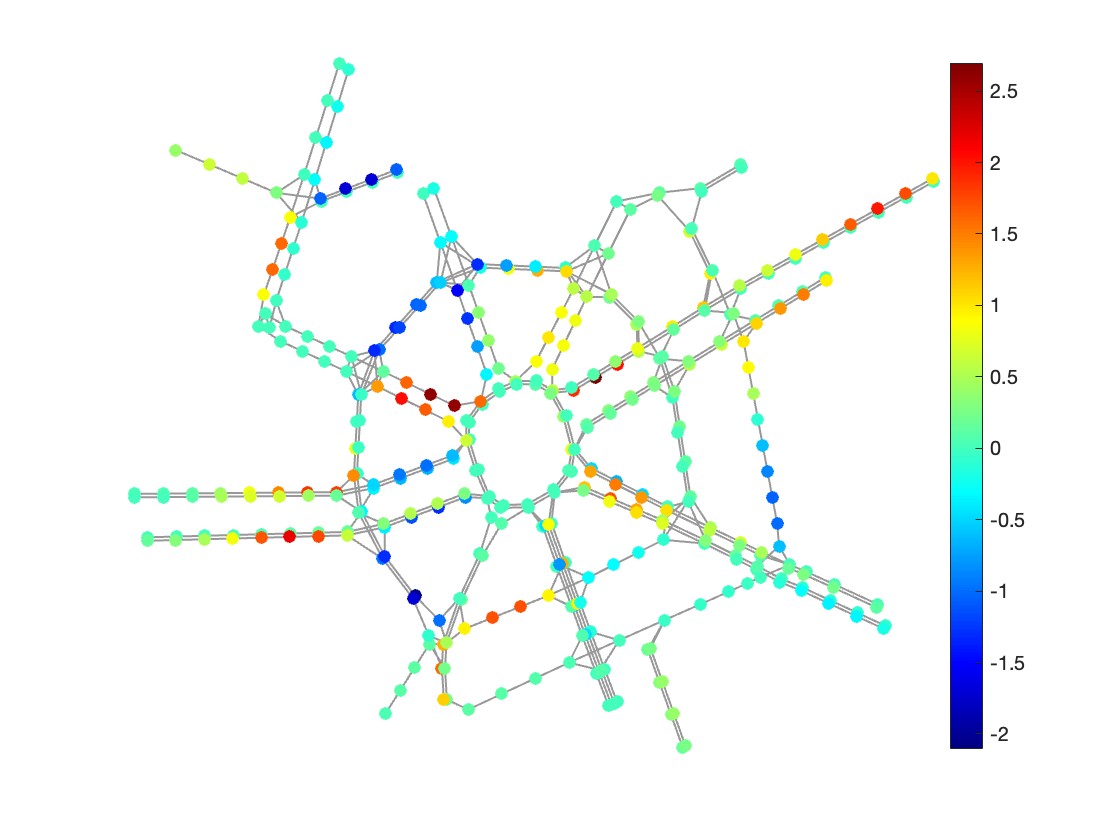}
			\parbox{2.65cm}{\tiny (f) Roman vehicular traffic data recovery from GFT sampling.}
		\end{minipage}
		\begin{minipage}[t]{0.23\linewidth}
			\centering
			\includegraphics[width=\linewidth]{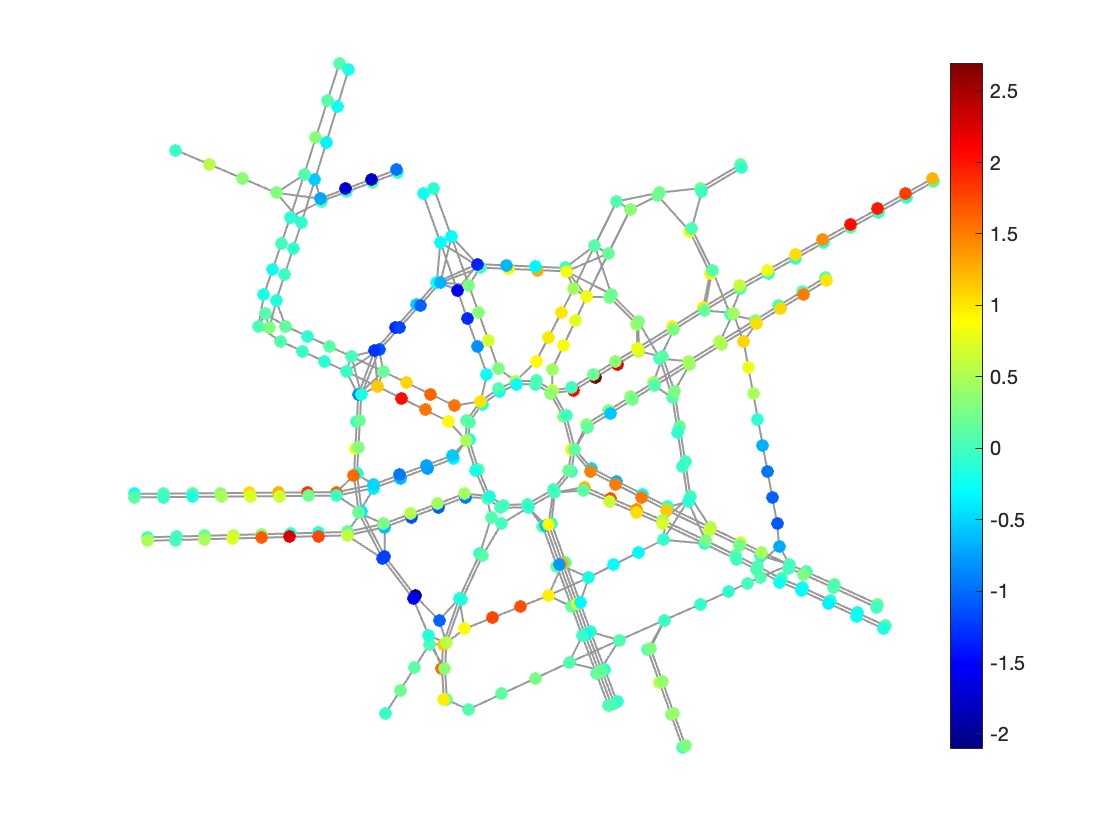}
			\parbox{3cm}{\tiny (g) Roman vehicular traffic data recovery from GFRFT sampling.}
		\end{minipage}
		\begin{minipage}[t]{0.23\linewidth}
			\centering
			\includegraphics[width=\linewidth]{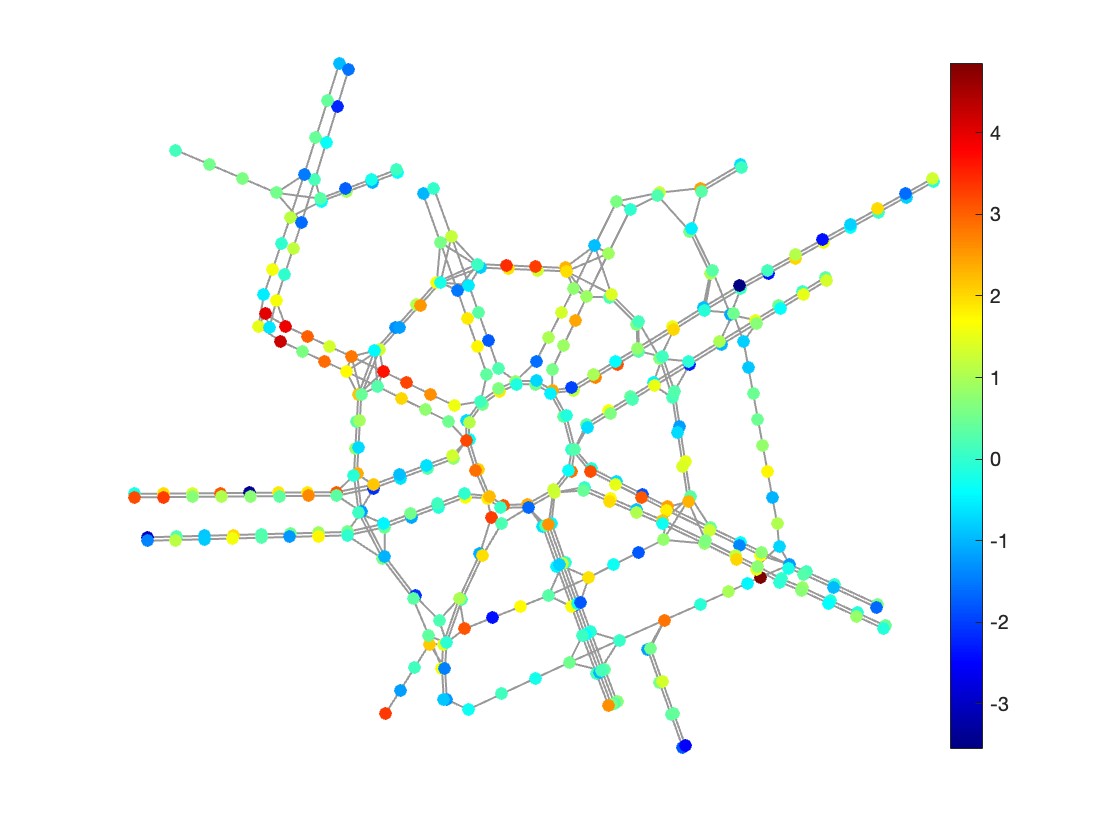} 
			\parbox{3cm}{\tiny (h) Roman vehicular traffic data recovery from GLCT sampling.}
		\end{minipage}
	\end{center}
	\caption{Recovered graph signals.}
	\vspace*{-3pt}
	\label{fig8}
\end{figure}

Since both A-optimal and T-optimal designs target optimizing the trace of an operator, MinFro and MaxSig yield identical results. Although MaxVertex is also the A-optimal design, it uses different optimal objective function and has different results. Meanwhile, MaxSpec and MaxVol essentially reflect D-optimal design, yielding similar results. Likewise, MinPinv and MaxSigMin essentially embodying E-optimal design, produce similar outcomes. It is evident that Laplacian-based GFT, the GFT, and the GFRFT sampling fail to perfectly recover the graph signals, as these signals, while $\mathbf{M}$-bandlimited, are not strictly bandlimited in the Laplacian-based GFT, the GFT and the GFRFT domains. Only the GLCT sampling achieves perfect signal recovery.
\begin{table}[h]
	\centering
	\caption{NMSE vs. different sampling frames for David sensor network based on sampling strategies corresponding to optimal designs}\label{tab2}
		\footnotesize
	\begin{tabular}{ccccc}
		\toprule
		Sampling methods & Laplacian-based GFT &  GFT  &GFRFT & GLCT\\
		\midrule
		MinFro& $5.5784\times 10^{12}$ & $1.6643\times 10^{15}$  &$1.7736$&$3.4424\times 10^{-21}$\\
		MaxVertex& $3.9576\times 10^{12}$ & $9.1507\times 10^{9}$  &$471.0998$&$8.0523\times 10^{-19}$\\
		MaxSpec& $6.7287\times 10^{12}$ & $5.9499\times 10^{12}$  &$1.7736$&$1.0796\times 10^{-19}$\\
		MaxVol& $6.7287\times 10^{12}$ & $1.6188\times 10^{7}$  &$5.6375$&$1.2084\times 10^{-22}$\\
	    MinPinv& $1.0260$ & $1.5349$  &$1.7736$&$2.9277\times 10^{-24}$\\
	    MaxSigMin& $1.0024$& $1.4591$ &$1.8291$&$3.0375\times 10^{-24}$\\
	    MaxSig& $5.5784\times 10^{12}$ & $1.6643\times 10^{15}$  &$1.7736$&$3.4424\times 10^{-21}$\\
		\bottomrule
	\end{tabular}
\end{table}

\begin{table}[h]
	\centering
	\caption{NMSE vs. different sampling frames for Roman vehicular traffic data based on sampling strategies corresponding to optimal designs}\label{tab3}
		\footnotesize
	\begin{tabular}{ccccc}
		\toprule
		Sampling methods & Laplacian-based GFT &  GFT  &GFRFT & GLCT\\
		\midrule
		MinFro& $1.9346\times 10^{12}$& $2.1664\times 10^{11}$  &$5.1485\times 10^{3}$&$1.8502\times 10^{-23}$\\
		MaxVertex& $5.1846\times 10^{17}$& $2.2107\times 10^{13}$  &$1.1505\times 10^{5}$&$1.3938\times 10^{-22}$\\
		MaxSpec& $6.5921\times 10^{11}$& $9.8408\times 10^{12}$  &$19.5859$&$3.7200\times 10^{-19}$\\
		MaxVol& $6.5921\times 10^{11}$  &$9.8408\times 10^{12}$   &$20.3122$&$1.7204\times 10^{-23}$\\		
		MinPinv& $1.0305$ & $1.0860$  &$1.1172$&$5.2989\times 10^{-25}$\\
		MaxSigMin& $1.0281$& $1.0902$ &$1.1216$&$1.8111\times 10^{-24}$\\
		MaxSig& $1.9346\times 10^{12}$& $2.1664\times 10^{11}$  &$5.1485\times 10^{3}$&$1.8502\times 10^{-23}$\\
       \bottomrule
	\end{tabular}
\end{table}

\subsection{Classifying Online Blogs}
In this section, we classify a real-world political blog dataset comprising $N=1224$ nodes, categorizing them into conservative or liberal factions \cite{PolBlog}. Nodes within the graph denote blogs, whereas edges symbolize hyperlinks interlinking the blogs. Label signals are defined such that 0 represents conservative and 1 denotes liberal, with these label signals spanning the entire frequency band. By employing sampling to reduce the number of vertices, we facilitate easier storage and transmission. Subsequently, through signal recovery, we reconstruct the complete set of blogs and evaluate the performance of our sampling method based on classification accuracy. In this case, we employ the MaxSigMin sampling strategy and adjust the parameters of GLCT to determine the optimal performance and, as noted in \cite{GFRFTsamp}, we know that the best results are achieved when the parameter $\alpha$ approaches 1, albeit not precisely equal to 1. Therefore, while adjusting other parameters, we approximately recover the labeled signal using the lowest $|\mathcal{F}|$ frequency components by solving the following optimization problem
\begin{equation}
	\mathbf{\hat{x}^{\mathrm{opt}}_{|\mathcal{F}|}} = \arg\min_{\mathbf{\hat{x}}_{|\mathcal{F}|}\in\mathbb{R}^{|\mathcal{F}|}}~\left\Vert \mathrm{thres}\left(\mathbf{DO^{-M}_{|\mathcal{F}|}\hat{x}_{|\mathcal{F}|}}\right) - \mathbf{x}_\mathcal{S}\right\Vert^2_2,
\end{equation}
where, the threshold function $\mathrm{thres}(\cdot)$ assigns a value of 1 to all values greater than 0.5 and 0 otherwise. Subsequently, we set another threshold of 0.5 \cite{GFTsamp} to assign labels to the recovered signal
\begin{equation}
	\mathbf{x}^{\mathrm{opt}}_{\mathcal{R}} = \mathrm{thres}\left(\mathbf{O^{-M}_{|\mathcal{F}|}}\mathbf{\hat{x}^{\mathrm{opt}}_{|\mathcal{F}|}}\right).
\end{equation}

Since the GFT and the GFRFT are special cases of the GLCT, by adjusting the parameters of the GLCT, all methods can be encompassed. In our previous work \cite{GLCT}, the parameter $\xi$ only influenced the real and imaginary components of the spectrum, while the absolute value of the overall signal remained unchanged. Thus, we set $\xi=0.5k+1, k=1,...,N$ in this study. We varied the parameter $\alpha$ from $0.99$ to $1$. For the parameter $\beta=m \times 2^{n}, n \in \mathbb{R}$, as the accuracy remains constant when $m$ is fixed, we considered variations in $m$, setting $m=1,3,5,N$. Fig. \ref{fig9} illustrates the spectral representations with different bandwidths $|\mathcal{F}|$ of the GLCT with parameters $\xi=0.5k+1,\beta=2, \alpha=0.995$: (a) shows $|\mathcal{F}|=2$, (b) $|\mathcal{F}|=612$, and (c) $|\mathcal{F}|=1224$. Fig. \ref{fig10}(a) illustrates the change in classification accuracy for nodes with a frequency bandwidth $|\mathcal{F}|=2$ and a sampling size $|\mathcal{S}|=\mathrm{M}=5$ as we vary the parameter $\alpha$. The accuracy reaches its maximum value of 94.6078\% when $\alpha=0.995$ and $\beta=2^{n}$. In Fig. \ref{fig10}(b), we present a comparison of classification accuracies with increasing sampling size $\mathrm{M}$ for GFT, GFRFT \cite{GFRFTsamp}, and GLCT sampling. For the GFRFT and GLCT methods, the parameters selected correspond to the optimal values identified in Fig. \ref{fig10}(a). It is evident that GLCT outperforms GFRFT and GFT across varying sampling sizes $\mathrm{M}$.
% figure 9
\begin{figure}[h]
	\begin{center}
		\begin{minipage}[t]{0.31\linewidth}
			\centering
			\includegraphics[width=\linewidth]{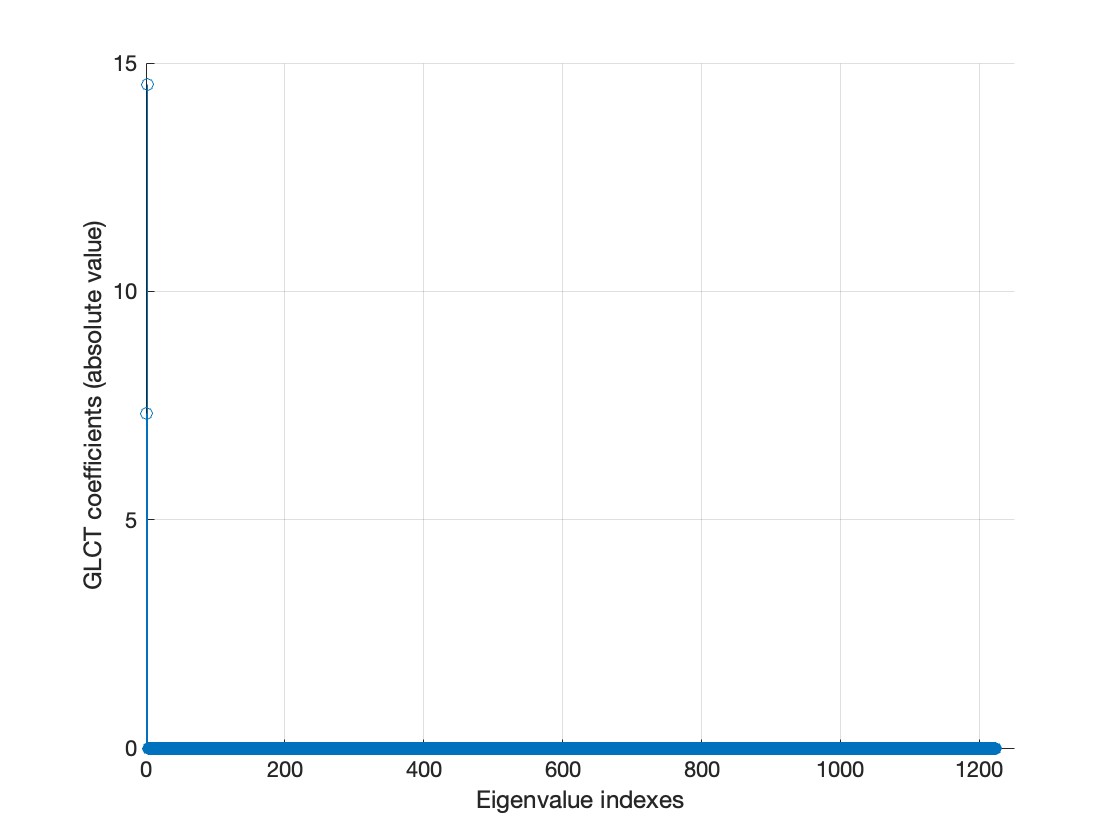}
			\parbox{4cm}{\tiny(a) The spectrum of the GLCT with parameters $\xi=0.5k+1, \beta=2, \alpha=0.995$, and $|\mathcal{F}|=2$.}
		\end{minipage}
		\begin{minipage}[t]{0.31\linewidth}
			\centering
			\includegraphics[width=\linewidth]{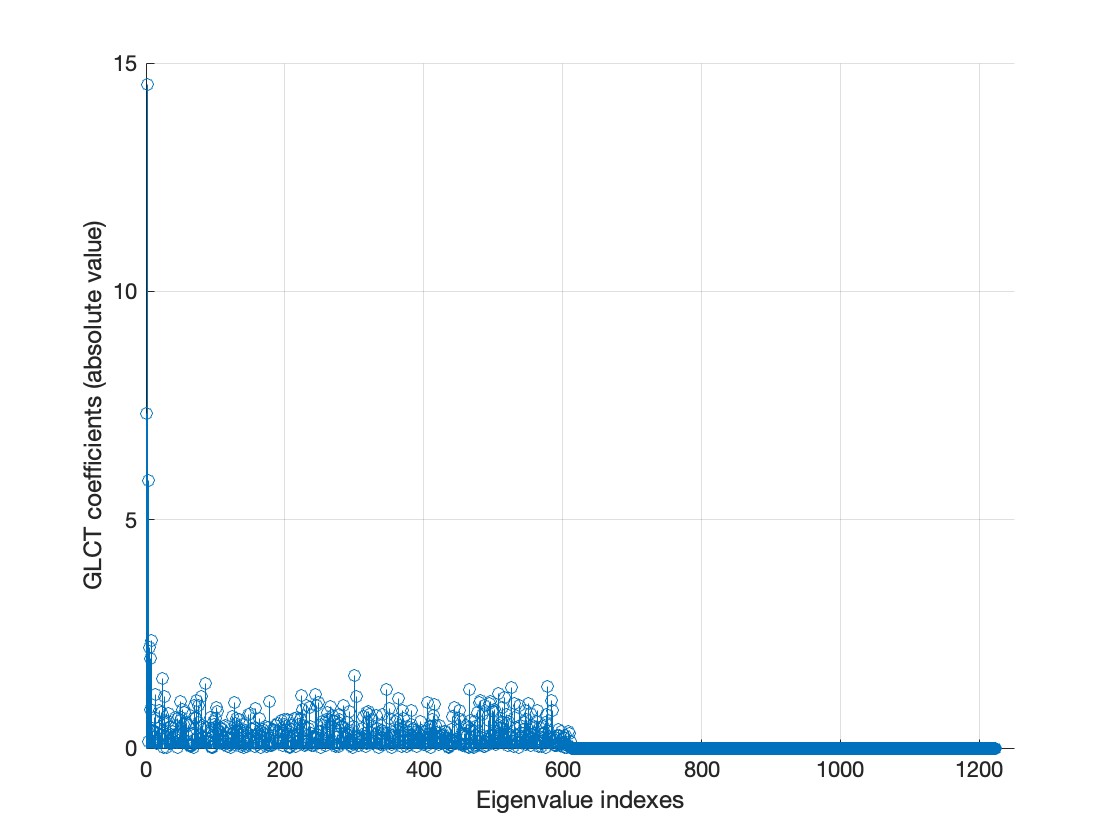}
			\parbox{4cm}{\tiny (b) The spectrum of the GLCT with parameters $\xi=0.5k+1, \beta=2, \alpha=0.995$, and $|\mathcal{F}|=612$.}
		\end{minipage}
		\begin{minipage}[t]{0.31\linewidth}
			\centering
			\includegraphics[width=\linewidth]{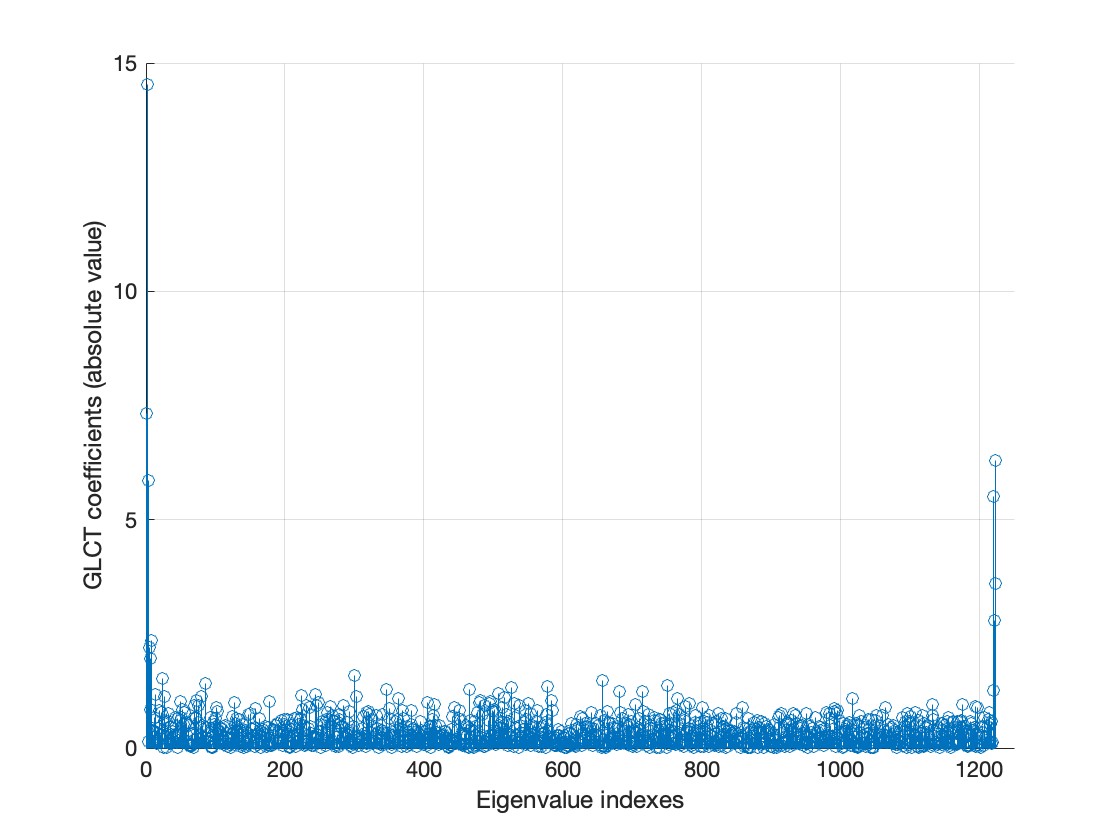}
			\parbox{4cm}{\tiny (c) The spectrum of the GLCT with parameters $\xi=0.5k+1, \beta=2, \alpha=0.995$, and $|\mathcal{F}|=1224$.}
		\end{minipage}
	\end{center}
	\caption{Bandlimited spectrum of online political blogs.}
	\vspace*{-3pt}
	\label{fig9}
\end{figure}
% figure 10
\begin{figure}[h]
	\begin{center}
		\begin{minipage}[t]{0.51\linewidth}
			\centering
			\includegraphics[scale=0.22]{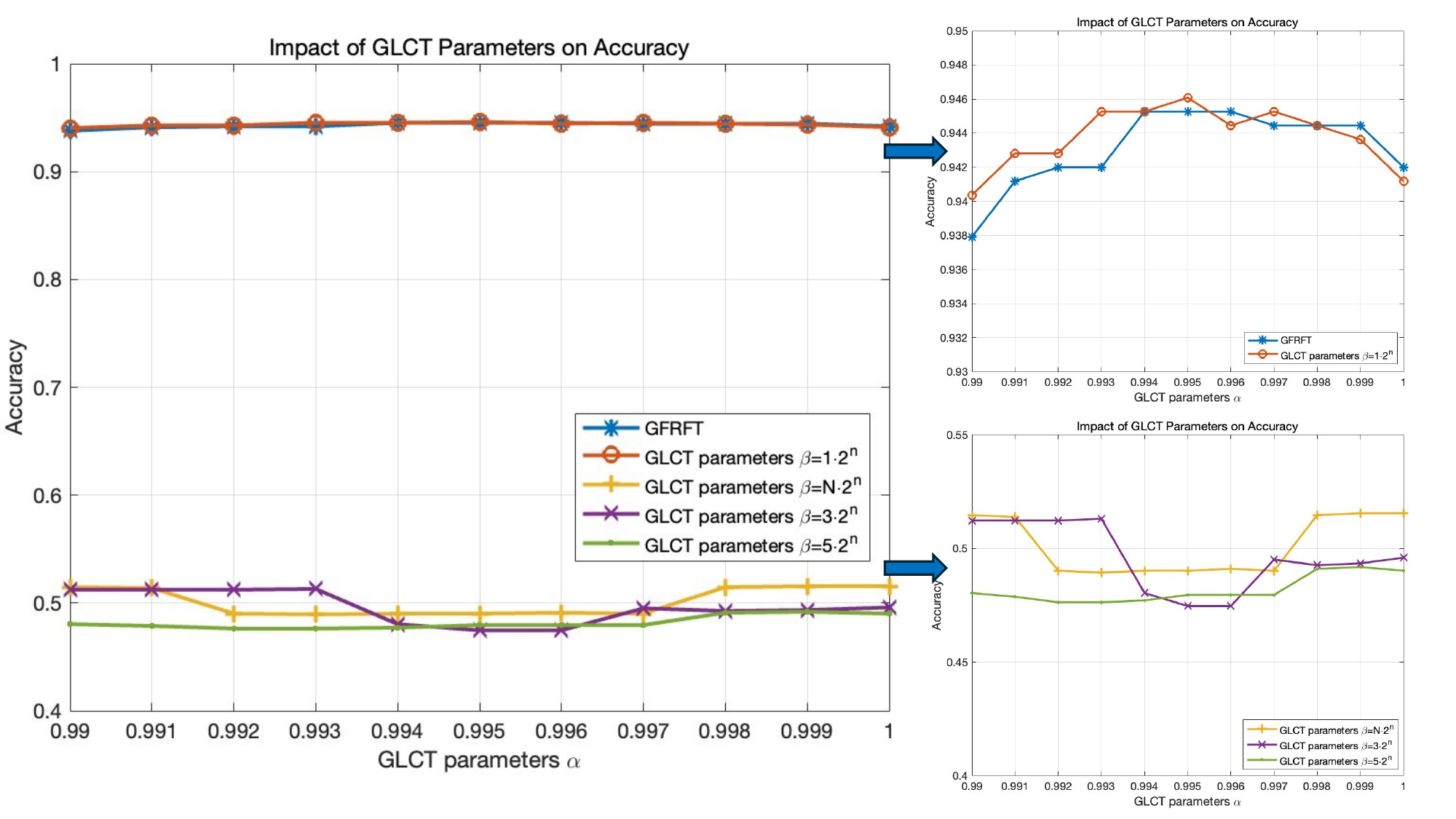}
			\parbox{6cm}{\tiny(a) Classification accuracy as a function of fractional order $\alpha$. The top-right graph illustrates the detailed classification accuracy of GFRFT and GLCT with $\beta=2^n$. The bottom-right graph represents the detailed classification accuracy with $\beta=3\cdot 2^n, 5\cdot 2^n, N\cdot 2^n$.}
		\end{minipage}
		\begin{minipage}[t]{0.45\linewidth}
			\centering
			\includegraphics[scale=0.15]{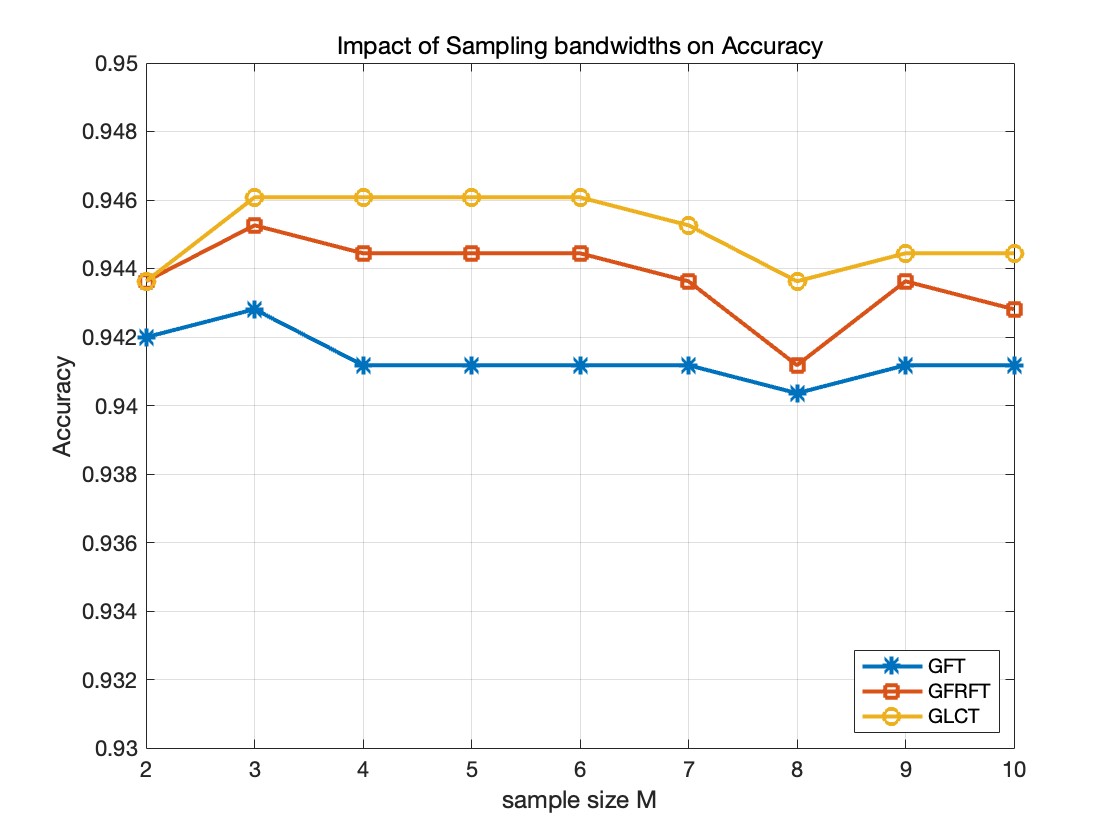}
			\parbox{5.5cm}{\tiny (b) Classification accuracy as a function of sample size $\mathrm{M}$ for GFT, GFRFT, and GLCT sampling. the blue line denotes GFT, the red line denotes the GFRFT of $\alpha=0.996$ \cite{GFRFTsamp}, and the yellow line denotes the GLCT of $\xi=0.5k+1,\beta=2^{n}, \alpha=0.995$.}
		\end{minipage}
	\end{center}
	\caption{Classification of online political blogs.}
	\vspace*{-3pt}
	\label{fig10}
\end{figure}

\subsection{Clustering of Bus Test Case}
The second part of our study focuses on sampling within an electrical grid, specifically utilizing the IEEE 118 bus test case, which represents a section of the Midwestern US power system as of December 1962. This network is composed of 118 vertices (i.e., buses) interconnected by edges (i.e., transmission lines) as depicted in Fig. \ref{fig11}(a) \cite{IEEEbusgraph}. The coloring of nodes is indicative of the components of eigenvectors, which are aligned with the eigenvalues of the GLCT operator matrix in an ascending order. These eigenvalues are derived under the GLCT parameters configured as $\xi=0, \beta = 2, \alpha=1$. The graph signal is represented by $\mathbf{x}=\mathbf{O}^{-\mathbf{M}}_N$, where the notation signifies the $N$th column of the matrix. As suggested in \cite{IEEE118}, the dynamics of generators yield smooth graph signals, making the assumption of $\mathbf{M}$-bandlimitedness plausible, albeit in an approximate sense, as demonstrated by the spectral plot of a bandlimited graph signal in Fig. \ref{fig11}(b). Network parameters are available at \url{http://www.cise.ufl.edu/research/sparse/matrices/}, with their layout derived through a graph drawing algorithm described in \cite{IEEEnetworks}.
\begin{figure}[h]%figure11
	\begin{center}
		\begin{minipage}[t]{0.45\linewidth}
			\centering
			\includegraphics[width=\linewidth]{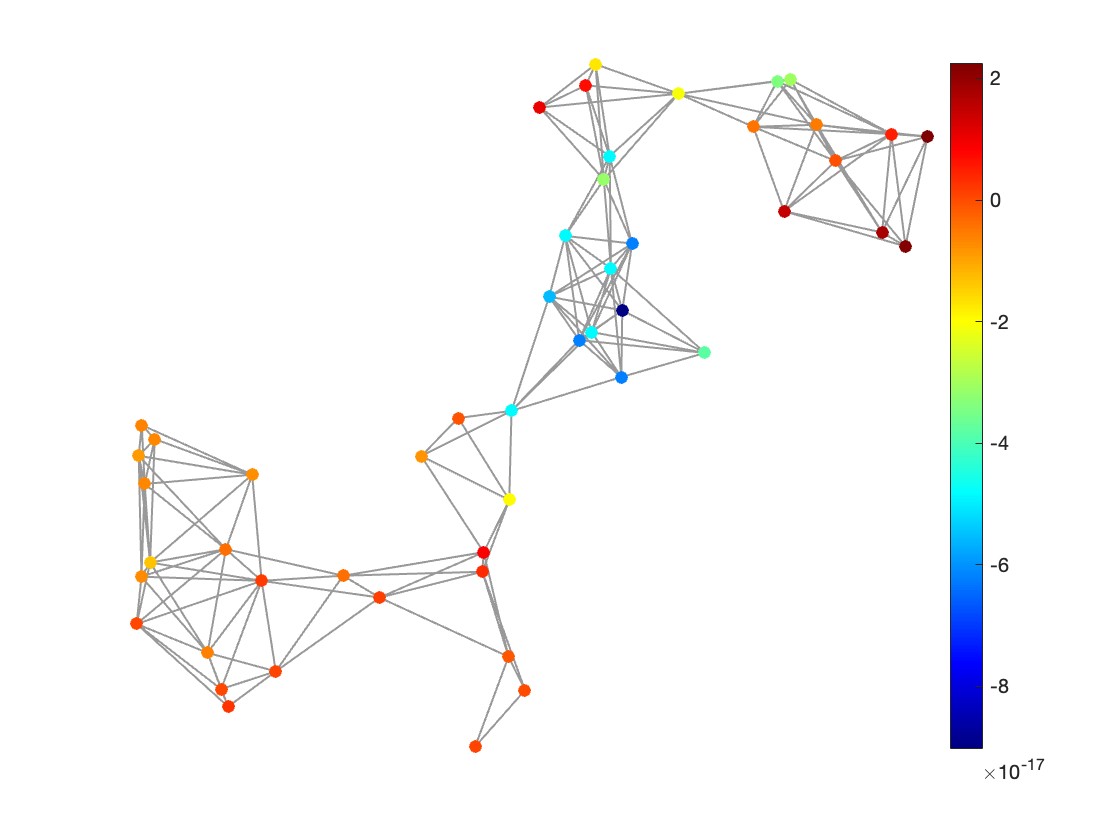}
			\parbox{5.5cm}{\tiny (a) The graph signal characterized by vertex signals $\mathbf{x}=\mathbf{O}^{-\mathbf{M}}_N$, with the parameters set as $\xi=0, \beta=2, \alpha=1$.}
		\end{minipage}
		\begin{minipage}[t]{0.45\linewidth}
			\centering
			\includegraphics[width=\linewidth]{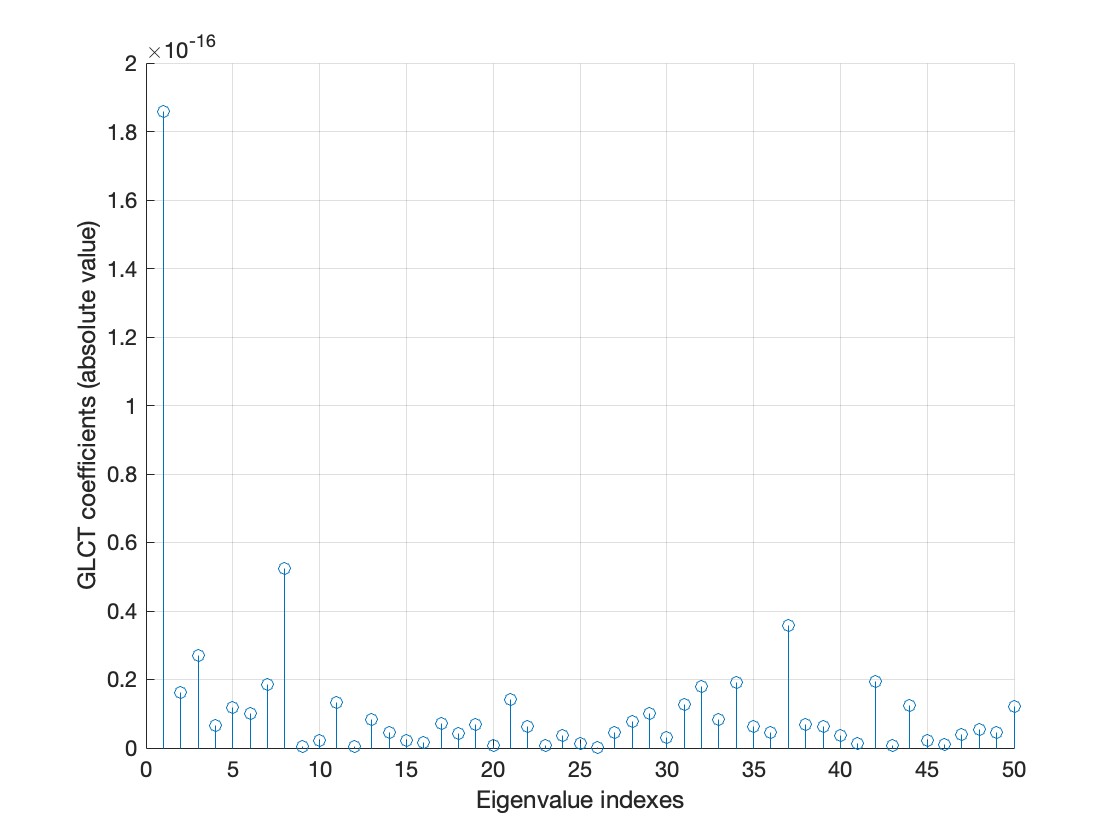}
			\parbox{5.5cm}{\tiny (b) The spectral representation of the graph signal $\mathbf{x}$ after $\mathbf{M}$-bandlimited.}
		\end{minipage}
	\end{center}
	\caption{Graph signal representation of IEEE 118 bus test cases.}
	\vspace*{-3pt}
	\label{fig11}
\end{figure}

In our demonstration, we generate random low-pass signals with a graph spectral bandwidth $|\mathcal{F}| = 8$ and collect a sample size $|\mathcal{S}| = 8$. We then count the number of triangles in the graph to select optimal sampling parameters using different GLCT settings, employing the MaxSigMin sampling strategy, and perform $K$-means clustering \cite{Cluster}, choosing $K=8$, with results illustrated in Fig. \ref{fig12}. And the signal recovery formula is Eq. \eqref{recovery}. In the context of the IEEE 118 bus test case, the pivotal utility of cluster analysis lies in discerning the functional zones and critical nodes within the power system. This discernment is instrumental in guiding the optimization of grid scheduling and bolstering the system's stability and efficiency. Leveraging the characteristics of these categorized groups enables the refinement of power generation scheduling strategies, aimed at minimizing the overall generation costs while ensuring the system's stable operation. Moreover, cluster analysis serves a crucial role in identifying nodes that may significantly impact the system's stability. This identification forms a foundational basis for the development of fault prevention measures and the formulation of emergency response plans \cite{Clustering1,Clustering2,Cluster}.
\begin{figure}[h]%figure12
	\begin{center}
		\begin{minipage}[t]{0.6\linewidth}
			\centering
			\includegraphics[width=\linewidth]{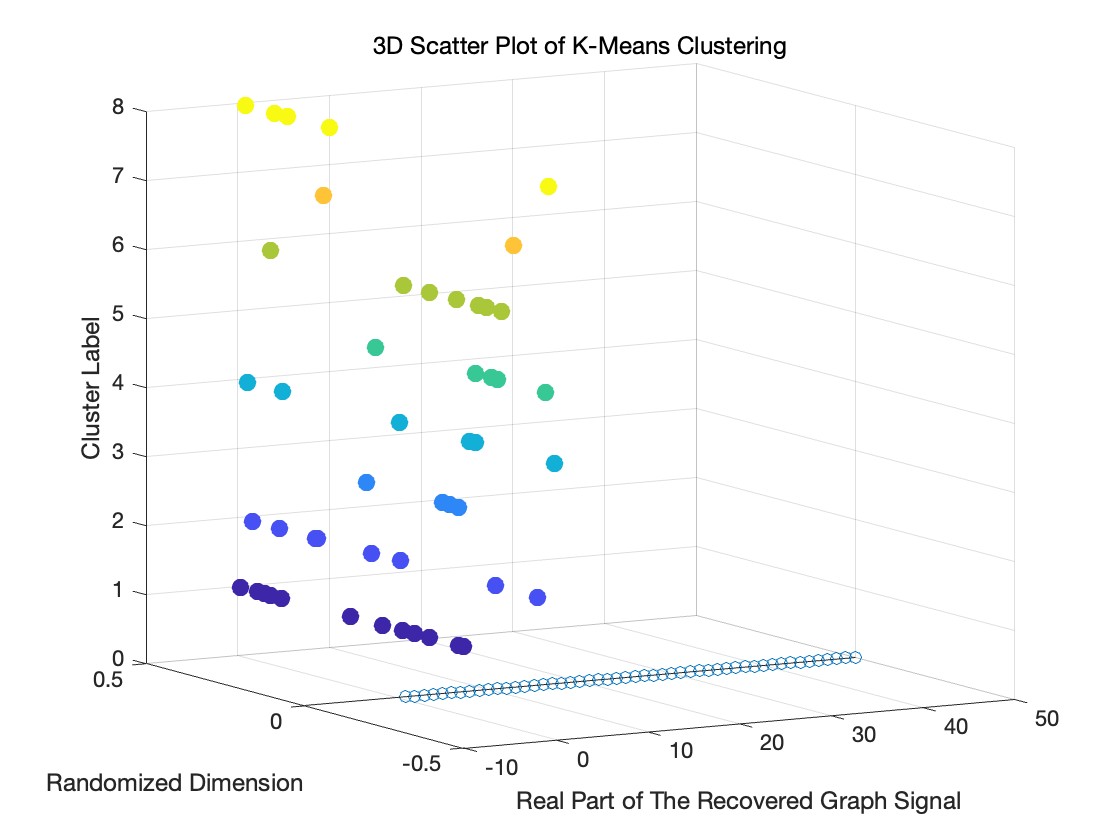}
			\parbox{5.5cm}{\tiny(a) $K$-means clustering of IEEE 118 bus test cases.}
		\end{minipage}
%		\begin{minipage}[t]{0.45\linewidth}
%			\centering
%			\includegraphics[width=\linewidth]{IEEE_eigenvaule.jpg}
%			\parbox{5.5cm}{\tiny (b) }
%		\end{minipage}
	\end{center}
	\caption{$K$-means clustering of IEEE 118 bus test cases.}
	\vspace*{-3pt}
	\label{fig12}
\end{figure}

Finally, we evaluate the clustering results using
\[ \text{Silhouette Score} = \frac{b(i) - a(i)}{\max\{a(i), b(i)\}}, \]
where, $a(i)$ represents the average distance of sample $i$ to other samples within the same cluster (intra-cluster compactness), and $b(i)$ represents the average distance of sample $i$ to the nearest samples in other clusters (inter-cluster separation). A result closer to 1 indicates that samples are correctly assigned to clusters, with good separation between clusters.
% figure 13
\begin{figure}[h]
	\begin{center}
		\begin{minipage}[t]{0.45\linewidth}
			\centering
			\includegraphics[width=\linewidth]{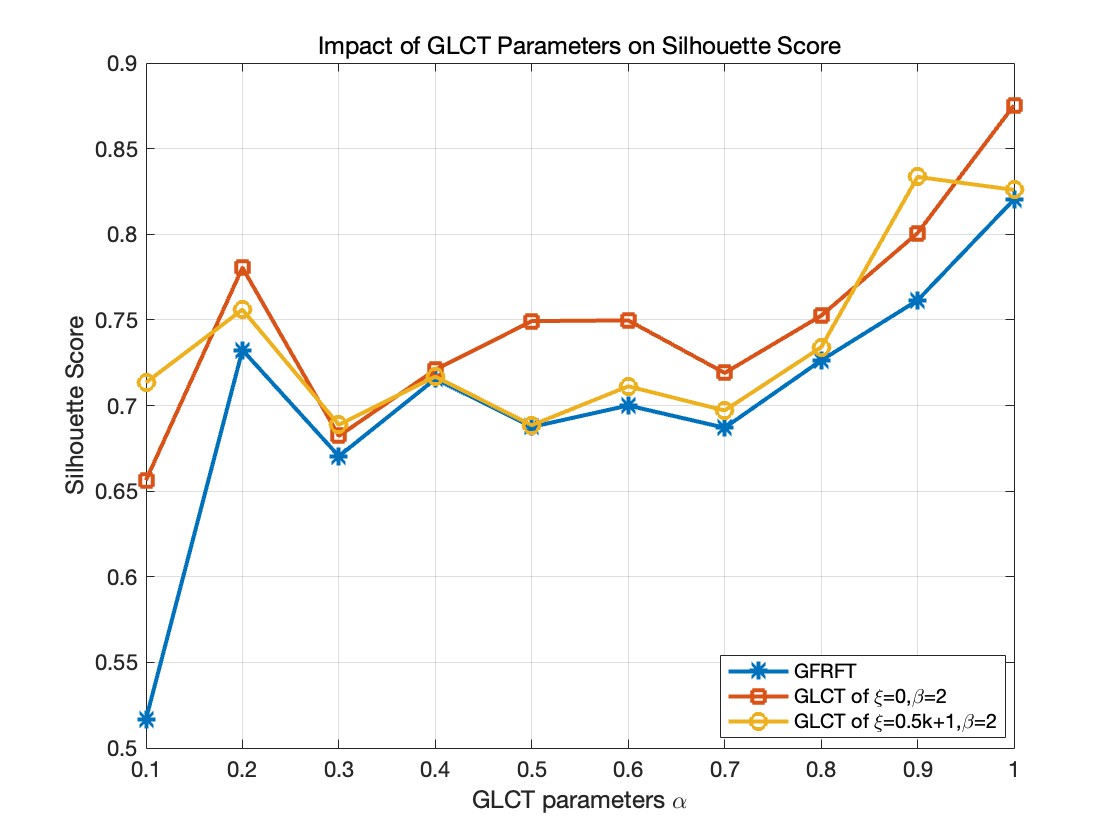}
			\parbox{5.5cm}{\tiny(a) Silhouette score as a function of fractional order $\alpha$, the blue line denotes GFRFT, the red line denotes the GLCT of $\xi=0,\beta=2$, and the yellow line denotes the GLCT of $\xi=0.5k+1,\beta=2$.}
		\end{minipage}
		\begin{minipage}[t]{0.45\linewidth}
			\centering
			\includegraphics[width=\linewidth]{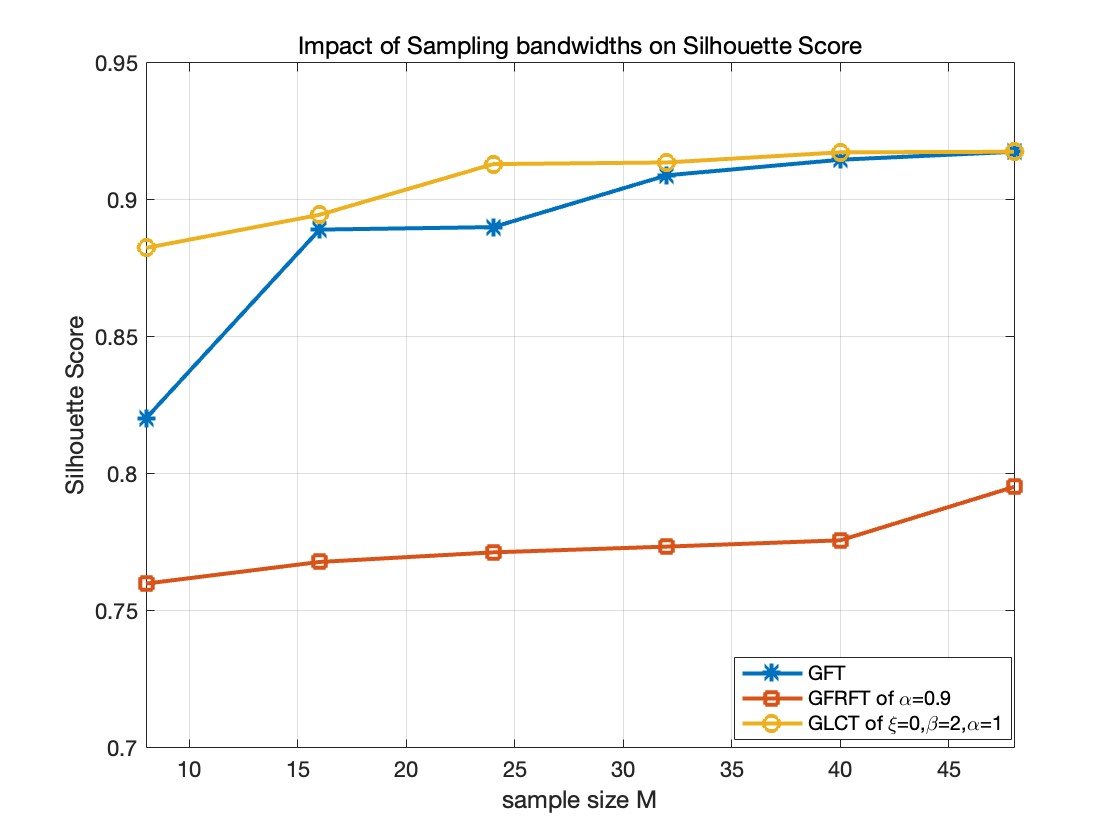}
			\parbox{5.5cm}{\tiny (b) Silhouette score as a function of sample size $\mathrm{M}$, the blue line denotes GFRFT of $\alpha=1$ (GFT), the red line denotes the GFRFT of $\alpha=0.9$, and the yellow line denotes the GLCT of $\xi=0,\beta=2,\alpha=1$.}
		\end{minipage}
	\end{center}
	\caption{Clustering of IEEE 118 Bus Test Case.}
	\vspace*{-3pt}
	\label{fig13}
\end{figure}

We conducted experiments by employing the method of controlling variables, adjusting the GLCT parameters and the number of sampling notes. Fig. \ref{fig13}(a) displays the silhouette score of fixed sampling $\mathrm{M}=|\mathcal{F}|=8$ nodes with varying parameter $\alpha$. The silhouette score reaches its maximum value of $0.8753$ when $\xi=0,\beta=2$ and $\alpha=1$. In (b), we compare the silhouette score of the GFT, GFRFT and GLCT with optimal parameters as the sampling size $\mathrm{M}$ varies. Overall, the GLCT exhibits superior performance.

\section{Conclusion}
\label{6}
In this paper, we propose the uncertainty principle and sampling theory in the GLCT domain. We demonstrate that the uncertainty principle in the GLCT domain has a broader scope compared to the GFT domain and shares interesting connections with sampling theory. We show that $\mathbf{M}$-bandlimited graph signals in the GLCT domain can achieve perfect recovery through the GLCT with $\mathbf{M}$ matrix coefficients. We employ experimentally designed optimal sampling strategies to ensure perfect recovery while maximizing robustness against noise. Various proposed sampling strategies are compared. Subsequently, we conduct simulation experiments and test GLCT sampling and recovery in a semi-supervised classification application and clustering of bus test cases. When comparing its performance with GFT and GFRFT sampling, we observe that GLCT sampling achieves superior classification accuracy under the optimal $\mathbf{M}$ matrix.
% \linenumbers 

\appendix
\section{Proof of Lemma 1}
\label{AA}
The proof for this statement is available in \cite{GUncertainty}. By utilizing Eq. \eqref{Dxx} and Eq. \eqref{Bxx}, we derive
\[
\mathbf{B^{M}DB^{M}x=B^{M}Dx=B^{M}x=x}.
\]
Consequently, $\lambda_{\max}\left(\mathbf{B^{M}DB^{M}} \right) =1$. On the other hand, if $\mathbf{B^{M}DB^{M}x=x}$ is known, then $\mathbf{B^{M}DB^{M}x=B^{M}x}$, as $( \mathbf{B^{M}})^{2}=\mathbf{B^{M}}$. This further implies $\mathbf{B^{M}x=x}$. Regarding vertex localization, utilizing the Rayleigh-Ritz theorem, we obtain
\[
\max_{\mathbf{x}} \mathbf{\frac{x^{\ast }Dx}{x^{\ast }x} }=\max_{\mathbf{x}} \mathbf{\frac{x^{\ast }B^{-M}DB^{M}x}{x^{\ast }x}} =1.
\]
Therefore, $\mathbf{Dx=x}$, indicating perfect localization in both the vertex and spectral domains.

\section{Proof of Lemma 2}
\label{AB}
The proof of this bound requires the use of curves \cite{GshapesUC}
\begin{equation}
	\gamma \left( \zeta\right)  :=\left( \left( \zeta\lambda_{\max } \right)^{\frac{1}{2} }  +\left( \left( 1-\zeta\right)  \left( 1-\lambda_{\max } \right)  \right)^{\frac{1}{2} }  \right)^{2}  ,\  \zeta\in \left[ \lambda_{\max } ,1\right]  .
\end{equation}
For a normalized vector signal $\mathbf{x}$, we consider two normalized vectors
\begin{equation*}
	\mathbf{y}=\frac{\mathbf{Dx}}{\left| \left| \mathbf{Dx}\right|  \right|_{2}  } ,\  \text{and} \ \  \mathbf{z}=\frac{\mathbf{B}^{\mathbf{M}}\mathbf{x}}{\left| \left| \mathbf{B}^{\mathbf{M}}\mathbf{x}\right|  \right|_{2}  } .
\end{equation*}
The usual definition of inner product $\left< \mathbf{y,z}\right> =\mathbf{y^{\ast }z}$, angular distance is the measure of vectors on the unit sphere, we can define the angle distance between two vectors as
\begin{equation}
	\angle \left( \mathbf{y,z}\right)  =\arccos \mathbb{R}\left< \mathbf{y,z}\right>,
\end{equation}
where, $\angle$ denotes the angular distance, $\mathbb{R}\left< \mathbf{y,z}\right> $ denotes the real part. In particular, the sum of the angular distances between vectors $\mathbf{y}$ and $\mathbf{x}$, and $\mathbf{z}$ and $\mathbf{x}$ is always larger than the angular distance between $\mathbf{y}$ and $\mathbf{z}$, i.e.
\begin{equation}
	\arccos \mathbb{R}\left< \mathbf{y,x}\right>  +\arccos \mathbb{R}\left< \mathbf{z,x}\right>  \geq \arccos\mathbb{R}\left< \mathbf{y,z}\right>  .\label{arccosyz}
\end{equation}
For $\mathbb{R}\left< \mathbf{y,z}\right>$ , the upper bound is given by the Cauchy-Schwarz inequality
\begin{equation*}
	\begin{aligned}\mathbb{R}\left< \mathbf{y,z}\right>  \leq &\left| \mathbb{R}\left< \mathbf{y,z}\right>  \right|  \leq \left| \left< \mathbf{y,z}\right>  \right|  \\ 
		=&\frac{\left| \left< \mathbf{Dx},\mathbf{B}^{\mathbf{M}}\mathbf{x}\right>  \right|  }{\left| \left| \mathbf{Dx}\right|  \right|_{2}  \left| \left| \mathbf{B}^{\mathbf{M}}\mathbf{x}\right|  \right|_{2}  } =\frac{\left| \left< \mathbf{D}^{2}\mathbf{x},\mathbf{B}^{\mathbf{M}}\mathbf{x}\right>  \right|  }{\left| \left| \mathbf{Dx}\right|  \right|_{2}  \left| \left| \mathbf{B}^{\mathbf{M}}\mathbf{x}\right|  \right|_{2}  } \\ 
		=&\frac{\left| \left< \mathbf{Dx,DB}^{\mathbf{M}}\mathbf{x}\right>  \right|  }{\left| \left| \mathbf{Dx}\right|  \right|_{2}  \left| \left| \mathbf{B}^{\mathbf{M}}\mathbf{x}\right|  \right|_{2}  } \leq \frac{\left| \left| \mathbf{Dx}\right|  \right|_{2}  \left| \left| \mathbf{DB}^{\mathbf{M}}\mathbf{x}\right|  \right|_{2}  }{\left| \left| \mathbf{Dx}\right|  \right|_{2}  \left| \left| \mathbf{B}^{\mathbf{M}}\mathbf{x}\right|  \right|_{2}  } \\ 
		=&\frac{\left| \left| \mathbf{Dx}\right|  \right|_{2}  \sqrt{\left< \mathbf{D}^{2}\mathbf{B}^{\mathbf{M}}\mathbf{x,B^{M}x}\right>  } }{\left| \left| \mathbf{Dx}\right|  \right|_{2}  \left| \left| \mathbf{B}^{\mathbf{M}}\mathbf{x}\right|  \right|_{2}  } =\frac{\left| \left| \mathbf{Dx}\right|  \right|_{2}  \sqrt{\left< \mathbf{DB}^{\mathbf{M}}\mathbf{x},(\mathbf{B^{M}})^{2}\mathbf{x}\right>  } }{\left| \left| \mathbf{Dx}\right|  \right|_{2}  \left| \left| \mathbf{B}^{\mathbf{M}}\mathbf{x}\right|  \right|_{2}  } \\ 
		=&\frac{\left| \left| \mathbf{Dx}\right|  \right|_{2}  \sqrt{\left< \mathbf{B^{M}DB^{M}x,B^{M}x}\right>  } }{\left| \left| \mathbf{Dx}\right|  \right|_{2}  \left| \left| \mathbf{B}^{\mathbf{M}}\mathbf{x}\right|  \right|_{2}  } \leq \frac{\left| \left| \mathbf{Dx}\right|  \right|_{2}  \left| \left| \mathbf{B}^{\mathbf{M}}\mathbf{x}\right|  \right|_{2}   \left| \left| \mathbf{B^{M}DB^{M}x}\right|  \right|_{2} }{\left| \left| \mathbf{Dx}\right|  \right|_{2}  \left| \left| \mathbf{B}^{\mathbf{M}}\mathbf{x}\right|  \right|_{2}  }\\  \leq
		& \sqrt{\lambda_{\max } \left( \mathbf{B^{M}DB^{M}}\right)  }. \end{aligned} 
\end{equation*}
Since we assume that $\lambda_{\max } \leq \zeta^{2} \eta^{2}$ and the right-hand expression in \eqref{arccosyz} is less than 1, we can get
\begin{equation*}
	\mathbb{R}\left< \mathbf{y,z}\right>  \leq \sqrt{\lambda_{\max } \left( 
		\mathbf{B^{M}DB^{M}}\right)  } \leq 1,\  \text{and} \  \mathbb{R}\left< \mathbf{y,x}\right>  =\zeta ,\  \mathbb{R}\left< \mathbf{z,x}\right>  =\eta^{\mathbf{M}} .
\end{equation*}

So that formula \eqref{arccoszeta_eta} is proved. Next, we prove formula \eqref{etaM}.
Because the function $\arccos(x)$ decreases monotonically at $(0, 1]$, and according to formula \eqref{arccoszeta_eta}, we get
\begin{equation*}
	\arccos \eta^{\mathbf{M}} \geq \arccos \sqrt{\lambda_{\max } } -\arccos \zeta,
\end{equation*}
substituting both sides into the cosine function respectively
\begin{equation*}
	\eta^{\mathbf{M}} \leq \cos \left(  \arccos \sqrt{\lambda_{\max } } -\arccos \zeta \right)  .
\end{equation*}
Applying the trigonometric identity $\cos (a-b)=\cos a\cos b+\sin a\sin b$, we finally get the inequality
\begin{equation*}
	\eta^{\mathbf{M}} \leq \sqrt{\lambda_{\max } } \zeta +\sqrt{1-\lambda_{\max } } \sqrt{1-\zeta^{2} } .
\end{equation*}

\section{Proof of Theorem 3}
\label{AC}
	Since $	\mathrm{rank}( \mathbf{DO}^{-\mathbf{M}}_{|\mathcal{F}|}) =|\mathcal{F}|$ and 
	$\mathrm{rank}((\mathbf{PD}\mathbf{O}^{-\mathbf{M}}_{|\mathcal{F}|})=|\mathcal{F}|$, 
	$\mathrm{rank}(\mathbf{P})=|\mathcal{F}|$, i.e. $\mathbf{P}$ spans $\mathbb{C}^{|\mathcal{F}|}$. Hence, $\mathbf{R}= \mathbf{O}^{-\mathbf{M}}_{|\mathcal{F}|} \mathbf{P}$ spans $\mathrm{BL}_{|\mathcal{F}|}(\mathbf{O}^{\mathbf{M}})$.
	
	Let $\mathbf{J=RD}$, we have
	\begin{align*}
		\mathbf{J}^2&=\mathbf{RDRD} =\mathbf{O}^{-\mathbf{M}}_{|\mathcal{F}|}\left(\mathbf{PD} \mathbf{O}^{-\mathbf{M} }_{|\mathcal{F}|}\right)\mathbf{PD} \\
		&= \left(\mathbf{O}^{-\mathbf{M} }_{|\mathcal{F}|}\mathbf{I}_{|\mathcal{F}|\times |\mathcal{F}|} \mathbf{P} \right) \mathbf{D=RD}=\mathbf{J},
	\end{align*}
	thus $\mathbf{J=PD}$ is a projection operator.
	
Since $\mathbf{RD}$ is a projection operator, $\mathbf{RDx}$ is an approximation of $\mathbf{x}$ in the space of $\mathrm{BL}_{|\mathcal{F}|} \left(  \mathbf{O}^{\mathbf{M}}  \right)$. We achieve perfect recovery when $\mathbf{x}$ is in the space of $\mathrm{BL}_{|\mathcal{F}|} \left(  \mathbf{O}^{\mathbf{M}}  \right)$.
\section*{Declaration of competing interest}
The authors declare that they have no known competing financial interests or
personal relationships that could have appeared to influence the work reported
in this paper.

\section*{Acknowledgments}
This work were supported by the National Natural Science Foundation of China [No. 62171041], the BIT Research and Innovation Promoting Project [No.2023YCXY053], and Natural Science Foundation of Beijing Municipality [No. 4242011]. The authors would like to thank the editor, reviewers, and Prof. Linyu Peng of Keio University for their comments that improved the technical quality of this work.

%\section*{References}

\

\

\
%\textbf{Yu Zhang} was born in Huhhot, China, in 1998. He received the B.S. degree from China University of Geosciences, Beijing, China, in 2021. He is currently working towards the Ph.D. degree in Beijing Institute of Technology, Beijing, China. His research interests include the linear canonical transform, graph signal processing.\

%\textbf{Bing-Zhao Li} was born in 1975. He received the B.S. degree from Shandong Normal University in 1998, and the M.S. and Ph.D. degrees from Beijing Institute of Technology, in 2001 and 2007, respectively. He is currently a Professor with the School of Mathematics and Statistics, Beijing Institute of Technology, Beijing, China. He was supported by Henry Fok Education Foundation Young Teacher Award and the New Century Excellent Talents in University in 2010 and 2012. His current research interests include the theory and applications of fractional Fourier transform, linear canonical transform, and fractional calculus.\

\end{document}